\documentclass[a4paper,sort&compress,3p]{elsarticle}

\usepackage{mathtools}
\usepackage{amsthm}
\usepackage{amsmath}
\usepackage{amssymb}
\usepackage{bm}
\usepackage{subcaption}
\usepackage{enumitem}
\usepackage{stmaryrd}

\newtheorem{ass}{Assumption}
\newtheorem{prop}{Proposition}
\newtheorem{thm}{Theorem}
\newtheorem{defi}{Defintion}
\newtheorem{rmk}{Remark}
\newtheorem{problem}{Problem}

\newtheorem{con}{Conjecture}

\newcommand{\f}[1]{\mathbf{#1}}
\newcommand{\norm}[1]{\left\lVert#1\right\rVert}
\newcommand{\jump}[1]{\left\llbracket#1\right\rrbracket}
\newcommand{\average}[1]{\left\{#1\right\}}

\journal{}

\begin{document}

\begin{frontmatter}

\title{Construction of approximate $C^1$ bases for isogeometric analysis on two-patch domains} 

\author[JKU]{Pascal Weinm\"uller}
\ead{pascal.weinmueller@jku.at}
\author[JKU]{Thomas Takacs}
\ead{thomas.takacs@jku.at}
\address[JKU]{Institute of Applied Geometry, Johannes Kepler University Linz, Altenberger Str. 69, 4040 Linz, Austria}

\begin{abstract}
In this paper, we develop and study approximately smooth basis constructions for isogeometric analysis over two-patch domains. One key element of isogeometric analysis is that it allows high order smoothness within one patch. However, for representing complex geometries, a multi-patch construction is needed. In this case, a $C^0$-smooth basis is easy to obtain, whereas $C^1$-smooth isogeometric functions require a special construction. Such spaces are of interest when solving numerically fourth-order PDE problems, such as the biharmonic equation and the Kirchhoff-Love plate or shell formulation, using an isogeometric Galerkin method.
 
With the construction of so-called analysis-suitable $G^1$ (in short, AS-$G^1$) parametrizations, as introduced in~\cite{collin2016analysis}, it is possible to construct $C^1$ isogeometric spaces which possess optimal approximation properties, cf.~\cite{kapl2019argyris}. These geometries need to satisfy certain constraints along the interfaces and additionally require that the regularity $r$ and degree $p$ of the underlying spline space satisfy $1 \leq r \leq p-2$. The problem is that most complex geometries are not AS-$G^1$ geometries. Therefore, we define basis functions for isogeometric spaces by enforcing approximate $C^1$ conditions following the basis construction from~\cite{kapl2017dimension}. For this reason, the defined function spaces are not exactly $C^1$ but only approximately.

We study the convergence behaviour and define function spaces that converge optimally under $h$-refinement, by locally introducing functions of higher polynomial degree and lower regularity. The convergence rate is optimal in several numerical tests performed on domains with non-trivial interfaces. While an extension to more general multi-patch domains is possible, we restrict ourselves to the two-patch case and focus on the construction over a single interface.
\end{abstract}

\end{frontmatter}

\section{Introduction}

Isogeometric Analysis (IGA), which is introduced in \cite{hughes2005isogeometric}, is a method for numerical simulation combining Finite Elemente Analysis (FEA) with Computer Aided Design (CAD). Within the IGA framework, the same spline functions are used for the exact representation of the CAD geometry and for the approximation of the FEA solution. IGA offers many advantages over classical (piecewise linear) finite elements by providing a basis of high smoothness and high polynomial degree. It is therefore ideal for solving high order partial differential equations (PDEs) over geometries comprised of a single patch. However, most geometries of interest are not given as a single patch, but are represented by a collection of patches forming a so-called multi-patch domain. Note that, in general, CAD models are composed of trimmed patches, cf.~\cite{marussig2018review}, which we do not address here.

In this paper, we assume that the geometry is represented by a two-patch parametrization where the patch parametrizations are matching along the interface. On such $C^0$-matching, two- or multi-patch domain, one can construct a $C^0$-smooth basis in a rather straighforward way, see e.g.~\cite{scott2014isogeometric,beiraodaveiga2014mathematical}. Enforcing higher order smoothness over multi-patch domains is however non-trivial, except in regular regions, as in~\cite{sangalli2016unstructured,buchegger2016adaptively}. As a consequence, standard basis constructions, which are only $C^0$-smooth over patch interfaces, cannot be used directly for solving high-order PDEs. In the following, we focus on fourth order problems, such as the biharmonic equation or a Kirchhoff-Love plate or shell formulation. There are several different methods to overcome the problem of reduced smoothness.

One way is to use a $C^0$ multi-patch basis and to solve the fourth order problem using Nitsche's method. This approach is studied e.g. in~\cite{apostolatos2014nitsche, ruess2014weak, guo2015nitsche}. Due to the reduced regularity, additional integral terms are derived over all interfaces and a penalty term is introduced to the problem statement. In this way, the $C^1$-error, i.e., the jump of the normal derivative across the interface, is penalized. Thus, using Nitsche's method, one has to derive a more complicated variational formulation, depending on the underlying PDE and discretization space, and one has to find a good choice for the penalty parameter, which is also not always straightforward.
Another approach is using the mortar method, see e.g.~\cite{horger2019hybrid} or~\cite{brivadis2015isogeometric} for $C^0$-coupling. The mortar method is based on constraint minimization where the coupling constraints are enforced using Lagrange multipliers. The correct choice of the discrete Lagrange multiplier space, leading to a stable formulation, is non-trivial. Moreover, the mortar method results in a saddle-point problem of larger size than the original problem.

A different possibility to solve forth order problems over multi-patch domains is to perform strong $C^1$-coupling, where the basis functions are coupled strongly across the interfaces, thus creating a $C^1$-smooth space over the multi-patch domain. The first work, which is following the idea of strong $C^1$-coupling, is the so-called bending strip method, see \cite{kiendl2009isogeometric, kiendl2010bending}. The idea was later generalized to construct $C^1$ bases over multi-patch domains as in~\cite{nguyen2014comparative,KaViJu15,collin2016analysis,mourrain_dimension_2016,kapl2017isogeometric,karvciauskas2016generalizing}. See also~\cite{hna2021} for a summary of related approaches.

We follow the constructions in~\cite{kapl2017dimension,kapl2019argyris}, which are based upon the findings in~\cite{KaViJu15,collin2016analysis}, where an explicit formula for constructing a $C^1$ basis at the interface is stated. As developed in~\cite{GrPe15}, the $C^1$ continuity of an isogeometric function is equivalent to the $G^1$ geometric continuity of its graph surface. This geometric continuity, cf.~\cite{Pe02}, yields so-called gluing data for each interface from which one can construct a $C^1$ basis. However, within the isogeometric framework, this construction is only possible for analysis-suitable $G^1$ (in short AS-$G^1$) geometries which were characterized in~\cite{collin2016analysis}. AS-$G^1$ is defined by having linear gluing data for each interface. This class of geometries contains for instance bilinear patches. However, for most geometries the gluing data is not linear.

Hence, all approaches based on strong $C^1$-coupling across interfaces have a significant problem: they can only be applied to certain geometries. If the geometry is not AS-$G^1$, one may locally raise the polynomial degree or reduce the continuity requirements to obtain a sufficiently large space. In~\cite{chan2018isogeometric, chan2019strong} the authors follow the former strategy, by constructing a $C^1$-smooth spline space at the interface using spline functions of a higher polynomial degree. A similar strategy is also proposed in~\cite{kapl2019isogeometric}. In this article, we intend to follow the latter strategy by properly reducing the continuity requirements.

In Figure~\ref{fig::motivation} we compare the approximation powers of different example parametrizations. While the geometry depicted in Subfigure~\ref{fig::ASGeo} is AS-$G^1$, the geometry in Subfigure~\ref{fig::nonASGeo} is not. Constructing the $C^1$-smooth space for both geometries and solving the biharmonic equation, we observe the following behaviour as plotted in Subfigure~\ref{fig::motivationerrorplot}: the discretization using the $C^1$-smooth space over the AS-$G^1$ geometry yields optimal convergence rates and the $C^1$-smooth space over the non AS-$G^1$ geometry does not allow any convergence. This lack of convergence can be circumvented as follows: instead of using the gluing data, we introduce so-called \textit{approximated gluing data} which is then used to construct the basis functions at the interface. Since the gluing data is now approximated, the resulting space is only approximately $C^1$-smooth. However, with the correct choice of approximated gluing data, the optimal convergence rate is restored. This can be observed in Subfigure~\ref{fig::motivationerrorplot}. In this paper, we focus on two-patch domains and extend the construction of $C^1$ basis functions from AS-$G^1$ geometries to general two-patch geometries.

\begin{figure}[h!]
  \begin{subfigure}[b]{0.223\textwidth}
  \centering
        \includegraphics[width=\textwidth]{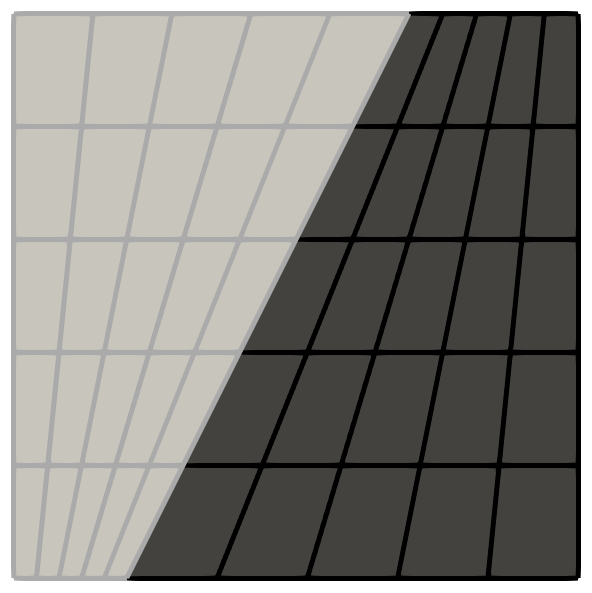}
        \caption{AS-$G^1$ geometry} \label{fig::ASGeo}
  \end{subfigure}
  \begin{subfigure}[b]{0.223\textwidth}
  \centering
        \includegraphics[width=\textwidth]{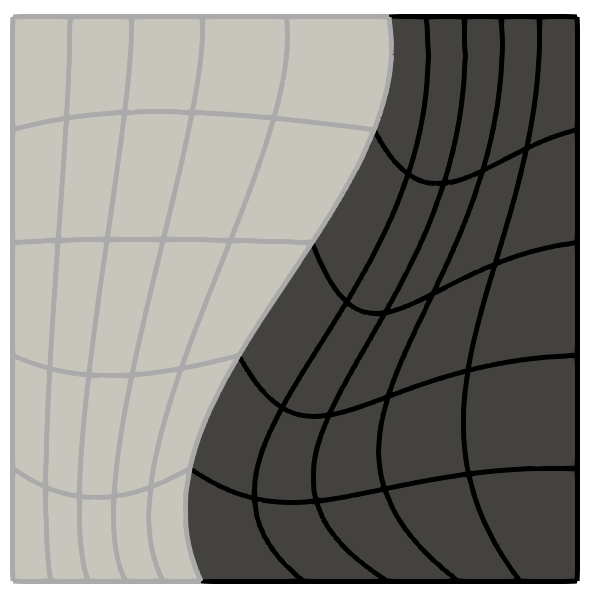}
        \caption{Non AS-$G^1$ geometry} \label{fig::nonASGeo}
  \end{subfigure}
  \begin{subfigure}[b]{0.55\textwidth}
  \centering
    \includegraphics[width=\textwidth]{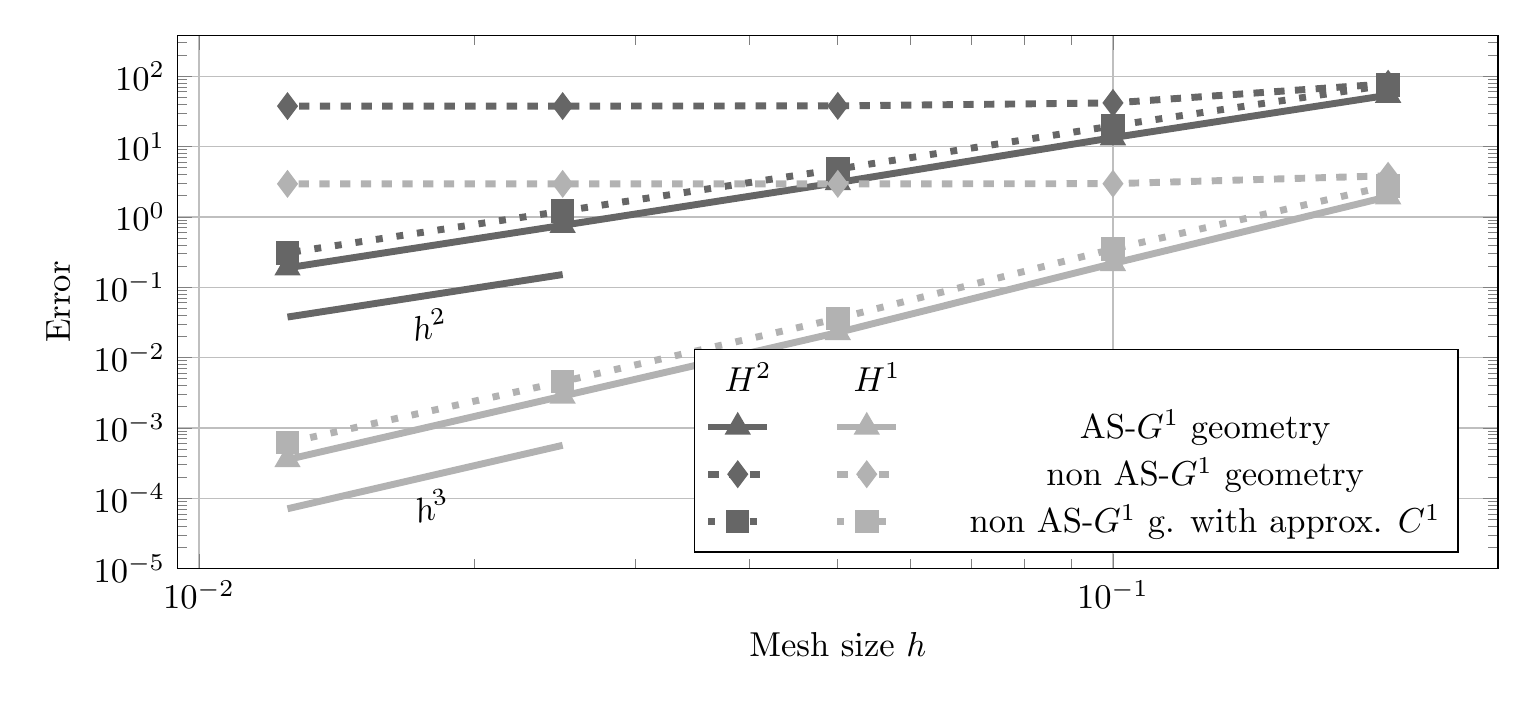}
    \caption{$H^1$- and $H^2$-errors of the discrete solutions for polynomial degree $p=3$.} \label{fig::motivationerrorplot}
  \end{subfigure}

     \caption{A numerical example to motivate the idea: we solve the biharmonic equation for the two parametrizations depicted in (a) and (b). In (c) we compare the resulting convergence rates. While the convergence rate is optimal for the AS-$G^1$ geometry, the $C^1$-smooth discretization over the non AS-$G^1$ geometry does not allow convergence. However, a properly constructed approximate $C^1$-smooth discretization over the same geometry again yields optimal rates.} \label{fig::motivation}
\end{figure}

The outline of the paper is as follows. We start with the definition of the model problem, more precisely the biharmonic equation, in Section~\ref{sec::modelproblem}. In Section~\ref{sec::bsplines}, we recall the definition of B-splines and introduce the notation we use. The description of the geometry mapping is given in Section~\ref{sec::twopatchdomain}. The main part of the basis construction is explained in Section~\ref{sec::isogeometricdiscretization}. Here, we define the gluing data and the spaces which are used for solving the biharmonic problem. In Section~\ref{sec::propertiesinterfacespaces} the properties of the approximate $C^1$ spaces are discussed. The discrete problem is stated in Section~\ref{sec::twopatchformulation} which is used to obtain the numerical results shown in Section~\ref{sec::numericalexperiments}.

\section{Model problem}\label{sec::modelproblem}

In this paper, we focus on the biharmonic equation. Let $\Omega$ be a bounded open subset of $\mathbb{R}^2$ with a sufficiently smooth (piecewise Lipschitz) boundary $\partial \Omega$ and a given source function $f$. We consider the fourth order problem
\begin{align}
  \Delta^2 \varphi = f \quad \text{ in } \Omega \label{eq::biharmonic}
\end{align}
with the boundary conditions
\begin{align}
  \varphi &= g_0 \quad \text{ on } \partial \Omega \quad \text{ and } \label{eq::boundary1}\\
  \Delta \varphi &= g_1 \quad \text{ on } \partial \Omega , \label{eq::boundary2}
\end{align}
where $g_0$ and $g_1$ are given. One can use different boundary conditions such as $\varphi = g_0$ and $\partial_n \varphi = g_1$ on $\partial \Omega$, but in this paper we focus on the boundary conditions stated in \eqref{eq::boundary1} and \eqref{eq::boundary2}. We assume that all functions $f$, $g_0$ and $g_1$ are sufficiently smooth, i.e., $f\in H^{-2}(\Omega)$, $g_0 \in H^{-3/2}(\partial\Omega)$, $g_1 \in H^{-1/2}(\partial\Omega)$. Note that~\eqref{eq::boundary1} is enforced as an essential boundary condition, which can be eliminated. Hence, we assume from now on that the problem is homogeneous. Let 
\begin{equation*}
\begin{array}{lll}
\mathcal{V}_0 &\coloneqq H^2 (\Omega) \cap H^1_0 (\Omega) = \{ \psi \in H^2 (\Omega) \;|\; \psi = 0 \text{ on } \partial \Omega \}.
\end{array}
\end{equation*}
The weak formulation of the problem (\ref{eq::biharmonic})-(\ref{eq::boundary2}) is the following.
\begin{problem}\label{problem:model-problem}
Find $\varphi \in \mathcal{V}_0$ such that
\begin{align}
  a(\varphi, \psi ) = \langle F,\psi \rangle, \qquad \forall \; \psi \in \mathcal{V}_0, \label{eq::weakformulation}
\end{align}
where the bilinear form is defined as 
\begin{align*}
a(\varphi,\psi) & \coloneqq \int_{\Omega} \Delta \varphi \; \Delta \psi \; \mathrm{d}\f{x}
\end{align*}
and the right hand side as
\begin{align*}
\langle F,\psi \rangle  & \coloneqq \int_{\Omega} f \psi  \; \mathrm{d}\f{x} + \int_{\partial \Omega} g_1 \partial_{\f n} \psi \; \mathrm{d}\f{s} ,
\end{align*}
where $\partial_{\f n}$ is the normal derivative at the boundary.
\end{problem}

We solve Problem~\ref{problem:model-problem} using an isogeometric discretization, cf.~\cite{hughes2005isogeometric}. As it is common in IGA, we assume that the domain $\Omega$ is parametrized with B-spline patches. A discretization space can then be defined on $\Omega$ based on the same B-spline space as the geometry parametrization. Thus, we recall the definition of B-splines in the next section. A more detailed introduction to IGA can be found, e.g., in \cite{beiraodaveiga2014mathematical,CottrellBook}.

\section{B-spline spaces} \label{sec::bsplines}

In this section, a brief overview of B-splines is stated. Given positive integers $p$, $r$ and $n$, and an (uniform) mesh, with mesh size $h = 1/n$,
we define the (univariate) spline space $\mathcal{S} (p,r,h)$ of degree $p$ and regularity $r$, with $r<p$, as
\begin{equation}
\mathcal{S} (p,r,h)  \coloneqq \{ f \in C^r([0,1]) : f|_{(i h,(i+1) h)} \in \mathbb{P}^p \mbox{ for all }i=0,\ldots,n-1 \}. \label{eq::uniform-spline-space}
\end{equation}
The open knot vector $\Xi \coloneqq ( \xi_1,...,\xi_{N+p+1} )$ with $N=p+1+(p-r)(n-1)$ satisfies 
\begin{equation*}  
\Xi =
(\underbrace{0,\ldots,0}_{(p+1)-\mbox{\scriptsize times}},
\underbrace{\textstyle h,\ldots ,h}_{(p-r) - \mbox{\scriptsize times}}, 
\underbrace{\textstyle 2h,\ldots ,2h}_{(p-r) - \mbox{\scriptsize times}},\ldots, 
\underbrace{\textstyle (n-1)h,\ldots ,(n-1)h}_{(p-r) - \mbox{\scriptsize times}},
\underbrace{1,\ldots,1}_{(p+1)-\mbox{\scriptsize times}}).
\end{equation*}
We mention here that, for simplicity, the regularity is assumed to be the same at all interior knots and consequently all the knots have the same multiplicity. Given the knot vector $\Xi$ and a polynomial degree $p$ one can define the B-spline functions denoted as $b_i$, $1 \leq i \leq N$, using the Cox--de Boor recursion, see~\cite{prautzsch2002}. We have
\begin{align*}
\mathcal{S} (p,r,h)  = \text{span} \{ b_i, \; i = 1,...,N \}.
\end{align*}
A conforming discretization of Problem~\ref{problem:model-problem} requires $H^2$-regularity of the discretization space. To this end, we assume that the underlying spline space is in $H^2$, which is equivalent to $C^1$-smoothness.
\begin{ass}[Minimum regularity] \label{ass::minimumregularity}
  We assume that the spline space $\mathcal{S} (p,r,h)$ is at least $C^1$-smooth, i.e., $r \geq 1$.
\end{ass}
The definitions can be extended to the two-dimensional case by means of a tensor-product structure. Let $\f h = \left( h_1, h_2 \right)$ be the pair of (uniform) mesh-sizes and $\left( \Xi_1, \Xi_2 \right)$ be the two knot vectors, one for each direction. Additionally, we define the bivariate B-spline functions as $b_{\boldsymbol{i}} \coloneqq b_{1,i_1} \otimes b_{2,i_2}$ where any univariate B-spline function has the degree $p_{s}$ and the regularity $r_{s}$, $s \in \{1,2\}$. The tensor-product spline space $ \boldsymbol{ \mathcal{S} }  (\f p, \f r, \f h) $ is spanned by the bivariate B-spline functions, yielding
\begin{align*}
 \boldsymbol{ \mathcal{S} } (\f p, \f r, \f h) = \mathcal{S}_1 (p_1,r_1,h_1) \otimes \mathcal{S}_2 (p_2,r_2,h_2) = \text{span} \{ b_{\boldsymbol{i}} \}_{(1,1)\leq \f i\leq (N_1,N_2)},
\end{align*}
where $\f p=(p_1,p_2)$ and $\f r=(r_1,r_2)$.

\section{The two-patch geometry} \label{sec::twopatchdomain}

We assume that the domain $\Omega$ is given as the union of two non-overlapping subdomains, i.e., we have open subdomains $\Omega^{(S)}$, for $S \in \{L,R\}$, such that
\begin{align*}
  \overline{\Omega} = \bigcup_{S \in \{L,R\}} \overline{\Omega}^{(S)}, \quad \Omega^{(L)} \cap \Omega^{(R)} = \emptyset,
\end{align*}
with a single interface $\Gamma$ which is defined as 
\begin{align*}
  \Gamma = \partial \Omega^{(L)} \cap \partial \Omega^{(R)} \cap \Omega.
\end{align*}
Here $\overline{\Omega}^{(S)}$ denotes the closure of ${\Omega}^{(S)}$. In this paper, we always consider the notation $S \in \{L,R\}$ where $L$ denotes the left patch and $R$ the right patch. Furthermore, each $\Omega^{(S)}$ is a spline patch with the geometry mapping $\f F^{(S)} \in (\boldsymbol{ \widehat{\mathcal{S}} }^{(S)})^2 $ with 
\begin{align*}
  \f F^{(S)} : \widehat{\Omega} \to \overline{\Omega}^{(S)},
\end{align*}
where $\boldsymbol{ \widehat{\mathcal{S}} }^{(S)} = \boldsymbol{ \mathcal{S} } (\hat{\f p}^{(S)}, \hat{\f r}^{(S)}, \hat{\f h}^{(S)})$ is a tensor-product spline space as defined in Section~\ref{sec::bsplines} and $\widehat{\Omega} = [0,1]^2$.  We assume that the mappings are regular, i.e.,
\[
 | \det (\nabla \f F^{(S)})(u,v) | \geq C > 0, \qquad \forall \; (u,v)\in\widehat\Omega.
\]
Moreover, we assume that the patch interface is along an entire edge of both patches. Without loss of generality, on each patch the interface is parametrized by $(u,v)\in \{0\}\times (0,1)$, which can be achieved by a simple reparametrization (a combination of translation, rotation and symmetry). Furthermore, we assume that the patch parametrizations agree along the interface, summarized in the following.
\begin{ass}[$C^0$-conformity at the interface] \label{ass::c0conformityatinterface}
The parametrizations of the two
patches meet $C^0$ along the interface, i.e.,
  \begin{align}
  \f F^{(L)} (0,v) = \f F^{(R)} (0,v) \quad \forall \; v \in [0,1]. \label{eq::c0assumption}
\end{align}
\end{ass}
For simplicity, we assume $\widehat{\mathcal{S}}_2^{(L)} = \widehat{\mathcal{S}}_2^{(R)}$. As a consequence, the left and the right patch share the same tangential derivative
\begin{align}
  \boldsymbol{t}(v) \coloneqq \partial_v \f F^{(L)} (0,v) = \partial_v \f F^{(R)} (0,v) \label{eq::tangentvector}
\end{align}
along the interface. Hence, the unit tangent vector is given by
\begin{align}
  \boldsymbol{t}_0(v) \coloneqq \frac{\boldsymbol{t}(v)}{\tau(v)}, \label{eq::unittangentvector}
\end{align}
where $\tau(v) = \|\boldsymbol{t}(v)\|$. We denote the outward pointing unit normal vector to $\Omega^{(S)}$ by $\f n^{(S)}$. Along the interface $\Gamma$, one can compute the normal vector $\f n = \f n^{(L)} = - \f n^{(R)} $, which satisfies the following proposition.
\begin{prop} \label{prop::normalvector}
Given the geometry mapping $\f F^{(S)}$, the normal vector can be expressed as
\begin{align}
 \f n = \frac{1}{\det(\partial_u \f F^{(S)}, \f t_0 )} \left(\partial_u \f F^{(S)} - \left(\partial_u \f F^{(S)} \cdot \f t_0 \right) \; \f t_0 \right). \label{eq::normalvector}
\end{align}
\end{prop}

\begin{proof}
Since $\f n \perp \f t_0$, the vector $\partial_u \f F^{(S)}$ can be  uniquely described as a linear combination of $\f n$ and $\f t_0$, i.e.,
\begin{align*}
\partial_u \f F^{(S)}   = \lambda \f t_0 + \mu \f n, 
\end{align*}
where $\lambda$ and $\mu$ are the two unknown factors. Using the vector projection of $\partial_u \f F^{(S)}$ onto $\f t_0$ gives us the first unknown
\[
 \lambda = \partial_u \f F^{(S)} \cdot \f t_0.
\]
Since $\f n$ and $\f t_0$ are unitary vectors, we have
\[
  \mu^2 = \| \partial_u \f F^{(S)} \|^2 \|\f t_0\|^2 - (\partial_u \f F^{(S)} \cdot \f t_0)^2 = \det(\partial_u \f F^{(S)}, \f t_0)^2.
\]
Then, the desired result, including the sign of $\mu$, follows directly from the definition of $\f n = \f n^{(L)} = - \f n^{(R)} $.
\end{proof}
Figure \ref{fig::geometrical_mapping} gives an overview of the domain setting.
\begin{figure}[ht!]
    \centering
    \includegraphics[scale=1]{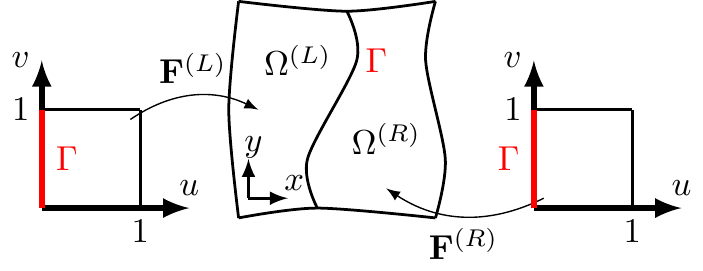}
    \caption{Example of the general setting for the two-patch parametrization.}
    \label{fig::geometrical_mapping}
\end{figure}

\section{The isogeometric discretization} \label{sec::isogeometricdiscretization}

In this section, we define isogeometric functions over two-patch domains and discuss their continuity conditions. As developed in~\cite{KaViJu15,collin2016analysis, kapl2017dimension, kapl2019argyris}, the class of analysis-suitable $G^1$ (in short AS-$G^1$) geometries allows optimal approximation. In~\cite{kapl2019argyris}, the Argyris isogeometric space $\mathcal{A}$ is introduced as the direct sum of the single patch-interior, edge and vertex components. Its name is derived from the fact that the vertex space is obtained from interpolating $C^2$-data at every vertex, similar to the Argyris finite element. Moreover, the edge space can be split in degrees of freedom for function values as well as normal derivative (or general crossing derivative) values along the interface. 

However, the construction for the space $\mathcal{A}$ is only possible for certain geometries, i.e., for AS-$G^1$ parametrizations, which are discussed in Section~\ref{sec::propertiesinterfacespaces}, Definition~\ref{def::asg1}. For general geometries, a different approach to construct an (approximate) $C^1$ isogeometric space is introduced in this section and discussed in more detail in Section~\ref{sec::propertiesinterfacespaces}. Since only two-patch domains are considered, a slight modification of the structure of the space is performed: there is no need to define separate vertex spaces, thus the space $\mathcal{A}$ is split into the patch-interior spaces and the interface space containing all functions that have non-vanishing trace or crossing derivative at the entire interface, including the vertices.

\subsection{Spline spaces of mixed regularity}

In Definition~\ref{defi::subsetwithhat} we introduce spline spaces of mixed regularity. Following the definition in~\eqref{eq::uniform-spline-space}, the uniform 
spline space $\mathcal{S} (p, r, h)$ has $n = 1 / h$ polynomial segments with the distinct inner knots
\[
 \{ h, 2h, ..., (n-1)h \}.
\]
The regularity of each knots is given by ${r}$. Similarly, the spline space $ \mathcal{S} (\hat{p}, \hat{r}, \hat{h})$ is $C^{\hat{r}}$-smooth across the inner knots
\[
 \{ \hat{h}, 2\hat{h}, ..., (\hat{n}-1)\hat{h} \},
\]
with $\hat{h} = k \cdot h$ and $\hat{n} = n / k$ for some positive integer $k$. We construct isogeometric functions based on the space $\mathcal{S} (p, (r,\hat{r}), (h,\hat{h}))$, satisfying 
\[
 \mathcal{S} (p, r, h) \subseteq \mathcal{S} (p, (r,\hat{r}), (h,\hat{h})) \quad \mbox{ and }\quad \mathcal{S} (\hat{p}, \hat{r}, \hat{h}) \subseteq \mathcal{S} (p, (r,\hat{r}), (h,\hat{h})),
\]
for $p\geq \hat{p}$, which is defined in the following.
\begin{defi} \label{defi::subsetwithhat}
 Let $\hat{h} = k \cdot h$, with $k\in \mathbb{N}^+$. We denote by 
\begin{align*}
 \mathcal{S} (p, (r,\hat{r}), (h,\hat{h})) 
\end{align*}
the space of splines that are polynomial of degree $p$ on each interval $(ih,(i+1)h)$, for $i=0,\ldots,n-1$, and across each inner knot $ih$, for $i=1,\ldots,n-1$ continuous of order 
\begin{align*}
 \begin{cases}
       \hat{r}, & \text{if } \exists \; j \in \mathbb{N} \text{ such that } ih = j\hat{h}, \\ 
       r, & \text{otherwise}.\\
 \end{cases}
\end{align*}
\end{defi}

\begin{figure}[h!]
\centering
 \begin{center}
  \includegraphics[width=0.5\textwidth]{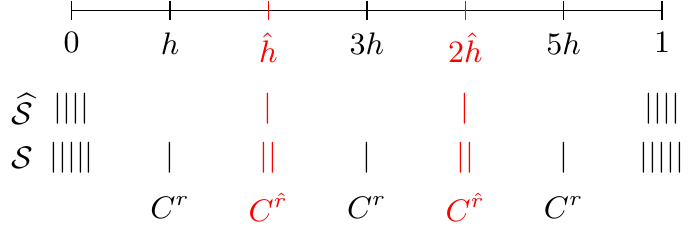}
\end{center}
\caption{We show an example of a knot vector corresponding to the space $\mathcal{S} = \mathcal{S} (4, (3,2), (1/6,1/3))$. Let us assume that the geometry is from the space $\widehat{\mathcal{S}} = {\mathcal{S}} (\hat{p} , \hat{r}, \hat{h}) = {\mathcal{S}} (3, 2, 1/3)$. Following Definition~\ref{defi::subsetwithhat}, the space $\mathcal{S}$ has regularity $r=3$ at the inner knots $h$, $3h$ and $5h$ and regularity $\hat{r}=2$ at the inner knots $2h = \hat{h}$ and $4h = 2\hat{h}$.} \label{fig::spacesExample}
\end{figure}
Note that the notation $ \mathcal{S}(p, (r,\hat{r}), (h,\hat{h})) $ simplifies to $\mathcal{S}(p, r, h) $ if the geometry space $\mathcal{S} (\hat{p}, \hat{r}, \hat{h})$ has no inner knots, i.e., $\hat{h}=1$, or if $r=\hat{r}$. In Figure~\ref{fig::spacesExample} an example is depicted.

Corresponding to the spaces $\boldsymbol{ \widehat{\mathcal{S}} }^{(S)}$ containing the geometry mappings, we define the discretization spaces as $\boldsymbol{ \mathcal{S} }^{(L)}$ and $\boldsymbol{ \mathcal{S} }^{(R)}$, with
\[
  \boldsymbol{ \mathcal{S} }^{(S)} = \mathcal{S}_1^{(S)} \otimes \mathcal{S}_2^{(S)} = \mathcal{S} (p_1^{(S)}, (r_1^{(S)},\hat{r}_1^{(S)}), (h_1^{(S)},\hat{h}_1^{(S)})) \otimes \mathcal{S}(p_2^{(S)}, (r_2^{(S)},\hat{r}_2^{(S)}), (h_2^{(S)},\hat{h}_2^{(S)})),      \quad S \in \{L,R\}.
\]
For the sake of simplicity, we only consider discretizations where the spline spaces of both patches are matching at the interface. This leads to the following restriction on the discrete space.
\begin{ass}[Matching two-patch discretization]
  We assume that the discrete spline spaces are matching at the interface, i.e., we have $\mathcal{S}_2 \coloneqq \mathcal{S} (p_2, (r_2,\hat{r}_2), (h_2,\hat{h}_2)) = \mathcal{S}_2^{(L)} = \mathcal{S}_2^{(R)}$, with $p_2 \coloneqq   p_2^{(L)} = p_2^{(R)}$, $r_2 \coloneqq r_2^{(L)} = r_2^{(R)}$, $\hat{r}_2 \coloneqq \hat{r}_2^{(L)} = \hat{r}_2^{(R)}$, $h_2 \coloneqq h_2^{(L)} = h_2^{(R)}$ and $\hat{h}_2 \coloneqq \hat{h}_2^{(L)} = \hat{h}_2^{(R)}$.
  \label{ass::c0conforming}
\end{ass}
In the following we define the $C^0$ isogeometric space, followed by the definition of the $C^1$ isogeometric space. Section~\ref{sec::gluingdata} introduces the gluing data which is needed for the construction of the approximate $C^1$ basis, which is described in Section~\ref{sec::constructionapproximateC1basis}.

\subsection{The space of $C^0$ isogeometric functions}

The space of $C^0$ isogeometric functions on $\Omega$ is given as
\begin{align*}
  \mathcal{V}^0_h = \{ \varphi \in C^0(\Omega) \text{ such that } f^{(S)} = \varphi \circ \f F^{(S)} \in \boldsymbol{\mathcal{S}}^{(S)}, \; S \in \{ L,R \} \}.
\end{align*}
  Following standard FEM notation, we denote the discrete space with a subscript $h$. Here $h$ represents the mesh size of the spline spaces $\boldsymbol{\mathcal{S}}^{(S)}$, for $S \in \{ L,R \}$. The mesh size of $\mathcal{V}^0_h$ in physical space is always of the same order as $h$. 
As the spline spaces are matching at the interface, one can easily construct a basis for the $C^0$ isogeometric space since for each function with non-vanishing trace on one side there exists exactly one function on the other side having the same trace.

\subsection{The space of $C^1$ isogeometric functions}

The space of $C^1$ isogeometric functions on $\Omega$ is given as   $\mathcal{V}^1_h = \mathcal{V}^0_h \cap C^1(\Omega)$. One can describe the $C^1$ continuity of a function at the interface by studying the geometric continuity of its graph surface. The graph surface $\boldsymbol{\Sigma} \subset \Omega \times \mathbb{R}$ of an isogeometric function $\varphi : \Omega \to \mathbb{R}$ consists of the two graph surface patches
\[
\boldsymbol{\Sigma}^{(S)} : [0,1]^2 \to \Omega^{(S)} \times \mathbb{R}, \quad S \in \{ L,R \}, \quad \boldsymbol{\Sigma}^{(S)} (u,v) = \left(\f F^{(S)} (u,v),f^{(S)} (u,v) \right)^T. 
\]
Considering only regularly parametrized patches, one can see that the $C^1$ continuity of an isogeometric function at the interface $\Gamma$ is equivalent to the $G^1$ geometric continuity of its graph parametrization, i.e., there exists for each point at the interface a well-defined tangent plane to the graph surface. The tangent plane is well-defined if and only if the graph surfaces fullfill for all $v \in [0,1]$
\begin{align*}
   \det \left(\partial_u \boldsymbol{\Sigma}^{(L)} (0,v) , \; \partial_u \boldsymbol{\Sigma}^{(R)} (0,v) , \; \partial_v \boldsymbol{\Sigma}^{(R)} (0,v)\right) = 0.
\end{align*}
This condition is known as $G^1$ (geometric) continuity, cf.~\cite{Pe02}. Note that $\partial_v \boldsymbol{\Sigma}^{(R)} (0,v) = \partial_v \boldsymbol{\Sigma}^{(L)} (0,v)$, due to the $C^0$ condition in Assumption~\ref{ass::c0conformityatinterface} together with the definition of $\mathcal{V}^0_h$. Figure~\ref{fig::g1continuity} illustrates the $G^1$ continuity. 
\begin{prop}
  An isogeometric function $\varphi \in \mathcal{V}^0_h$ belongs to $\mathcal{V}^1_h$ if and only if its graph surface $\boldsymbol{\Sigma}$ is geometrically continuous of order $1$, in short~$G^1$, at the interface $\Gamma$. \label{prop::g1continuity}
\end{prop}
For a discussion and generalizations of the equivalence above, see~\cite{KaViJu15,GrPe15,collin2016analysis}. 

\begin{figure}[h!]
 \centering
 \includegraphics[width=0.35\textwidth]{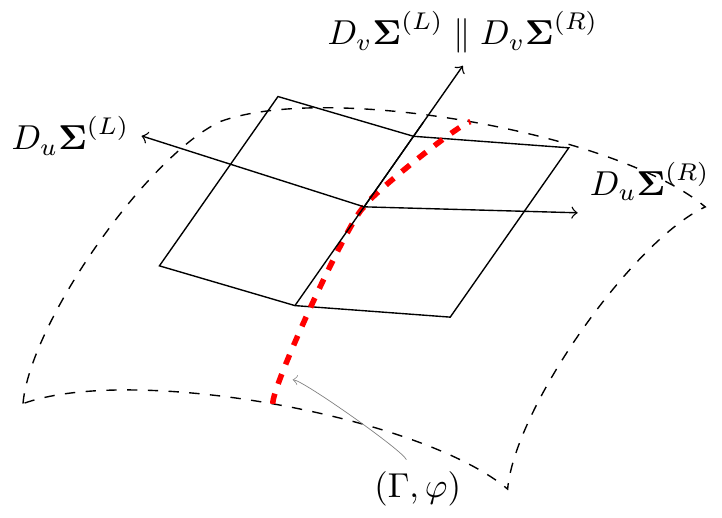}
 \caption{A visualization of $G^1$ continuity of the graph surface. To obtain $G^1$ continuity, the four vectors, two of them being equal, are coplanar (and span the tangent plane) for each point along the interface.} \label{fig::g1continuity}
\end{figure}

\subsection{The $C^1$ condition across the interface $\Gamma$}\label{sec::gluingdata}

If the graph surface is $G^1$ continuous, then there exist functions $\alpha^{(L)}$, $\alpha^{(R)}$, $\beta : [0,1] \to \mathbb{R}$, with
\begin{align*}
 \alpha^{(L)} (v) < 0 \quad \mbox{ and } \quad \alpha^{(R)} (v) > 0
\end{align*}
satisfied for all $v \in [0,1]$, such that 
\begin{align}
     \alpha^{(R)} (v) \partial_u \boldsymbol{\Sigma}^{(L)} (0,v) - \alpha^{(L)} (v) \partial_u \boldsymbol{\Sigma}^{(R)} (0,v) + \beta (v) \partial_v \boldsymbol{\Sigma}^{(R)} (0,v) = \boldsymbol{0} .\label{eq::exactgluingdatag1graph}
\end{align}
One can uniquely determine the functions $\alpha^{(L)}$, $\alpha^{(R)}$ and $\beta$ up to a common function $\gamma : [0,1] \to \mathbb{R}$ (with $\gamma (v) \neq 0$) by 
\begin{align}
 \left. \begin{array}{r l}
 \alpha^{(R)} (v) &= \gamma (v) \det \left( \partial_u \f F^{(R)} (0,v), \f t(v) \right), \\
 \alpha^{(L)} (v) &= \gamma (v) \det \left( \partial_u \f F^{(L)} (0,v), \f t(v) \right), \\
 \beta (v) &= \gamma (v) \det \left( \partial_u \f F^{(L)} (0,v), \partial_u \f F^{(R)} (0,v) \right), \end{array} \right. \label{def::alpha_beta}
\end{align}
where $\f t$ is defined as in~\eqref{eq::tangentvector}. Furthermore, there exist non-unique functions $\beta^{(L)}$, $\beta^{(R)} : [0,1] \to \mathbb{R}$ such that
  \begin{align}
   \beta (v) = \alpha^{(L)} (v) \beta^{(R)} (v) - \alpha^{(R)} (v) \beta^{(L)} (v) \label{eq::betaLbetaR}
  \end{align}
is satisfied for all $v \in [0,1]$. One possible choice for the functions is
\begin{align}
\left. \begin{array}{r l}
  \beta^{(L)} (v) &= \frac{\partial_u \f F^{(L)} (0,v) \cdot \f t_0(v)}{\tau (v)}, \\
  \beta^{(R)} (v) &= \frac{\partial_u \f F^{(R)} (0,v) \cdot \f t_0(v)}{\tau (v)}, \end{array} \right. \label{def::beta_S}
\end{align}
cf.~\cite[Proposition 1]{collin2016analysis}. The functions $\alpha^{(S)}$ and $\beta$ or more generally, the functions $\alpha^{(S)}$ and $\beta^{(S)}$ are called \emph{gluing data}. The first two lines of~\eqref{eq::exactgluingdatag1graph} are equivalent to
\begin{align}
     \alpha^{(R)} (v) \partial_u \f F^{(L)} (0,v) - \alpha^{(L)} (v) \partial_u \f F^{(R)} (0,v) + \beta (v) \partial_v \f F^{(R)} (0,v) = \boldsymbol{0} . \label{eq::exactgluingdatag1smooth}
\end{align}
Proposition \ref{prop::g1continuity} can be reformulated with the help of the gluing data.
\begin{prop}
  The isogeometric function $\varphi$ belongs to $\mathcal{V}^1_h$ if and only if the functions $f^{(S)} = \varphi \circ \f F^{(S)}$ fulfill 
\begin{align}
     \frac{1}{\alpha^{(L)} (v)} \left( \partial_u f^{(L)} (0,v) -\beta^{(L)} (v) \partial_v f^{(L)} (0,v) \right) = \frac{1}{\alpha^{(R)} (v)} \left( \partial_u f^{(R)} (0,v) - \beta^{(R)} (v) \partial_v f^{(R)} (0,v) \right). \label{eq::C1conditiontest}
\end{align}
\end{prop}
\begin{rmk} \label{rmk::generalspacegd}
  Note that for general patches $\f F^{(S)} \in \boldsymbol{ \widehat{\mathcal{S}} }^{(S)}$, assuming $\gamma(v)\equiv 1$, the functions $\alpha^{(S)}$ and $\beta$ fulfill $\alpha^{(S)} \in \mathcal{S}(2\hat{p}^{(S)}_2-1,\hat{r}^{(S)}_2-1,\hat{h}_2^{(S)})$ and $\beta \in \mathcal{S}(2\hat{p}_2^{(S)},\hat{r}_2^{(S)},\hat{h}_2^{(S)})$. The functions $\beta^{(L)}$ and $\beta^{(R)}$ are in general piecewise rational functions with regularity $\hat{r}^{(S)}_2-1$.
\end{rmk}
In order to obtain an optimal convergence rate for the gluing data, see Proposition~\ref{prop::bounds-E1-E2}, we need the following smoothness condition for the gluing data.
\begin{ass} \label{ass::smoothgluingdata}
 We assume that the gluing data satisfies $\alpha^{(S)} (v),\beta^{(S)} (v) \in C^1([0,1])$, for $S\in\{L,R\}$.
\end{ass}
Note that $\hat{r}^{(S)}_2 \geq 2$ is a sufficient condition for Assumption~\ref{ass::smoothgluingdata}.

\subsection{Construction of an approximate $C^1$ basis} \label{sec::constructionapproximateC1basis}

A basis construction for AS-$G^1$ two-patch geometries was developed in~\cite{kapl2017dimension}. In the following, we provide a variation of that approach, which extends the construction to general geometries by relaxing the smoothness condition.
Instead of constructing the $C^1$-smooth isogeometric space exactly, we define a basis of isogeometric functions which are only approximately $C^1$-smooth. We call the resulting space $\widetilde{\mathcal{V}}^{1}_h$ the \emph{approximate $C^1$ isogeometric space} on $\Omega$. It is defined as 
\begin{align*}
\widetilde{\mathcal{V}}^{1}_h = \widetilde{\mathcal{A}}_{\Gamma} \oplus \mathcal{A}^{(L)}_{\circ} \oplus \mathcal{A}^{(R)}_{\circ},
\end{align*}
where the \emph{interface space} $\widetilde{\mathcal{A}}_{\Gamma}$ is the space of functions which have non-vanishing traces or derivatives at the interface and $\mathcal{A}^{(S)}_{\circ}$ are the \emph{patch-interior spaces}, which have support on $\Omega^{(S)}$ and have vanishing value and normal derivative on $\Gamma$, hence, they satisfy $\mathcal{A}^{(S)}_{\circ}\subset \mathcal{V}^1_h$. The interface space $\widetilde{\mathcal{A}}_{\Gamma}$ is of the form 
\begin{align*}
\widetilde{\mathcal{A}}_{\Gamma} = \widetilde{\mathcal{A}}_{\Gamma,+} \oplus \widetilde{\mathcal{A}}_{\Gamma,-},
\end{align*}
where $\widetilde{\mathcal{A}}_{\Gamma,+}$ spans certain traces along the interface $\Gamma$ and $\widetilde{\mathcal{A}}_{\Gamma,-}$ has vanishing trace and (approximately) spans certain normal derivatives along $\Gamma$. 
We use the notation $\widetilde{\cdot}$ to signify that the spaces are only approximately $C^1$ and in general $\widetilde{\mathcal{V}}^{1}_h \not\subset \mathcal{V}^1_h$. Details on the behaviour of the normal derivative across the interface are discussed in Section~\ref{sec::propertiesinterfacespaces}. 
In Figure~\ref{fig::constructionC1space} the construction of the spaces is illustrated.

\begin{figure}[h!]
\centering
\begin{subfigure}[t]{0.3\textwidth}
 \begin{center}
  \includegraphics[width=\textwidth]{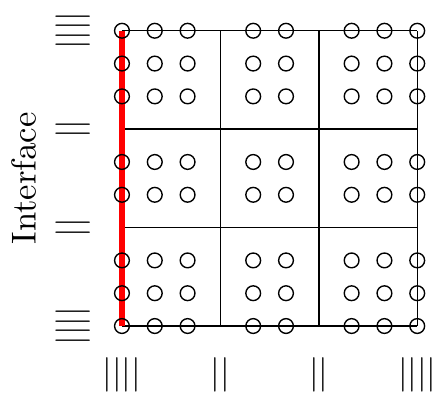}
\end{center}
\caption{The basis for the $C^0$ space.} \label{fig::initialSpace}
\end{subfigure}
~
\begin{subfigure}[t]{0.3\textwidth}
 \begin{center}
  \includegraphics[width=\textwidth]{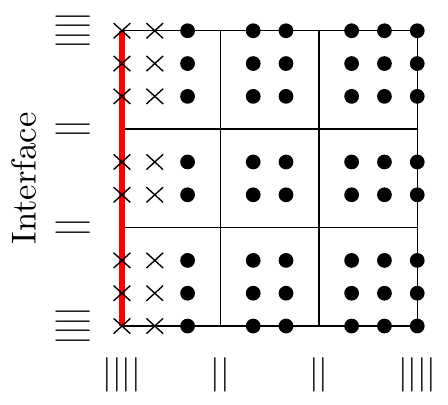}
\end{center}
\caption{Marking of basis functions to be eliminated.} \label{fig::selectionBasisfunctions}
\end{subfigure}
~
\begin{subfigure}[t]{0.3\textwidth}
 \begin{center}
  \includegraphics[width=\textwidth]{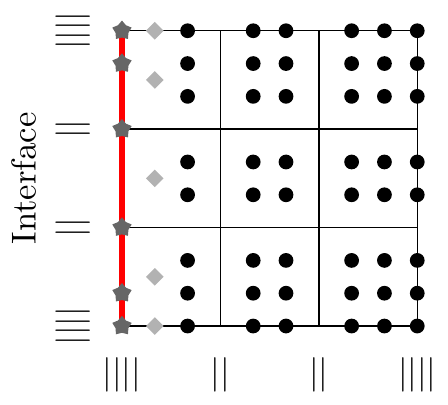}
\end{center}
\caption{New functions spanning the approximate $C^1$ space.} \label{fig::addingnewbasisfunctions}
\end{subfigure}
\caption{An example for the construction of the approximate $C^1$ space: we consider the spline space with $p_1 = p_2 = 3$ and $r_1 = r_2 =1$. The basis for the initial $C^0$~space is depicted in~(a), where each circle represents one basis function. Then all the basis functions which have non-vanishing values and normal derivatives at the interface are selected and eliminated. In~(b) the corresponding DOFs are depicted as crosses. The remaining black dots correspond to the DOFs of the patch-interior space $\mathcal{A}_\circ^{(S)}$. The eliminated basis functions are replaced with approximate $C^1$ basis functions which span the space $\widetilde{\mathcal{A}}_{\Gamma}$, see~(c). Here, the basis functions, which are represented with stars, span certain traces along the interface and those, which are visualized as diamonds, span certain (approximate) normal derivatives along the interface. The corresponding spaces are denoted with $\widetilde{\mathcal{A}}_{\Gamma,+}$ and $\widetilde{\mathcal{A}}_{\Gamma,-}$, respectively.} \label{fig::constructionC1space}
\end{figure}

\subsubsection{The patch-interior spaces}

The patch-interior spaces are spanned by those isogeometric basis functions which have vanishing function values and derivatives at the interface $\Gamma$, that is,
\begin{align}
\mathcal{A}^{(S)}_{\circ} = \text{span} \{ B^{(S)}_{\boldsymbol{j}} : \boldsymbol{j} \in \mathcal{I}^{(S)}_{\circ} \} \label{eq::spaceInteriorBasisfunction}
\end{align}
with
\begin{align*}
B^{(S)}_{\boldsymbol{j}} = \begin{cases}
                ( b^{(S)}_{\boldsymbol{j}} \circ (\f F^{(S)})^{-1} ) (\boldsymbol{x}) & \text{if } \boldsymbol{x} \in \Omega^{(S)} \\
                0 & \text{otherwise},
                  \end{cases}
\end{align*}
where $\mathcal{I}^{(S)}_{\circ} = \{(i_1,i_2)\in\mathbb{Z}^2:3\leq i_1 \leq N^{(S)}_1, 1 \leq i_2 \leq N^{(S)}_2\}$ and $\{b^{(S)}_{\boldsymbol{j}}\}$ are the basis functions of the space~$\boldsymbol{\mathcal{S}}^{(S)}$ of dimension $N^{(S)}_1\times N^{(S)}_2$. In contrast to the patch-interior space in \cite{kapl2019argyris}, the basis functions at the boundary (with no influence at the interface) are included in the space $\mathcal{A}^{(S)}_{\circ}$.

\subsubsection{The approximated gluing data}

In order to construct the space $\widetilde{\mathcal{A}}_{\Gamma}$ we introduce an approximation of the gluing data. The approximated gluing data is taken from the spline space $\mathcal{S}(\widetilde{p},(\widetilde{r},\hat{r}_2-1),(h_2,\hat{h}_2))$, where we prescribe the polynomial degree $\widetilde{p} \geq 2$ and the regularity $\widetilde{r} \geq 1$ in advance. Then $\widetilde{\alpha}^{(S)},\widetilde{\beta}^{(S)}$ are computed by a projection operator $P_{h}$ onto $\mathcal{S}(\widetilde{p},(\widetilde{r},\hat{r}_2-1),(h_2,\hat{h}_2))$ with 
\begin{align}
\widetilde{\alpha}^{(S)} \coloneqq P_{h}({\alpha}^{(S)}) \quad \text{and} \quad \widetilde{\beta}^{(S)} \coloneqq P_{h}({\beta}^{(S)}), \label{eq:projector-gluingdata}
\end{align}
where $\alpha^{(S)}$ is defined in~\eqref{def::alpha_beta} and $\beta^{(S)}$ in~\eqref{def::beta_S}. The functions $\widetilde{\alpha}^{(S)}$ and $\widetilde{\beta}^{(S)}$ are called the \textit{approximated gluing data}. They do not fulfill \eqref{eq::exactgluingdatag1graph} exactly, but approximately. For a suitable projection operator, we have
\begin{align}
     \norm{\widetilde{\alpha}^{(R)} (v) \partial_u \f F^{(L)} (0,v) - \widetilde{\alpha}^{(L)} (v) \partial_u \f F^{(R)} (0,v) +  \widetilde{\beta} (v) \partial_v \f F^{(R)} (0,v)} = O(h_2^{\widetilde{p}+1}) \label{eq::approxgluingdatag1smooth}
\end{align}
where $\widetilde{\beta} = \widetilde{\alpha}^{(L)} \widetilde{\beta}^{(R)} - \widetilde{\alpha}^{(R)} \widetilde{\beta}^{(L)}$. If the gluing data satisfies ${\alpha}^{(S)},{\beta}^{(S)} \in \mathcal{S}(\widetilde{p},(\widetilde{r},\hat{r}_2-1),(h_2,\hat{h}_2))$, then \eqref{eq::approxgluingdatag1smooth} is exactly zero. One can also set $\widetilde{p} = \widetilde{r} \geq 0$, if the gluing data is a polynomial function of degree $\widetilde{p}$, i.e., $\alpha^{(S)}, \beta^{(S)} \in \mathbb{P}^{\widetilde{p}}$. In this case, the requirement $\widetilde{p} \geq 2$ can be dropped. Using the approximated gluing data we construct basis functions along the interface.

\subsubsection{The interface space}

We define a basis following the approach presented in~\cite[Section~5.2]{kapl2017dimension}, where we replace the gluing data with the approximated gluing data. Let $\{ b_j^+\}_{1 \leq j \leq N_+}$ be the basis for the spline space
\[
\mathcal{S}^+ = \mathcal{S}(p^+,(r^+,\hat{r}_2),(h_2,\hat{h}_2))                                                                                                                                                                                                                    \]
and let $\{b_j^-\}_{1 \leq j \leq N_-}$ be the basis for the spline space
\[
 \mathcal{S}^- = \mathcal{S}(p^-,(r^-,\hat{r}_2-1),(h_2,\hat{h}_2)) ,                                       
\]
where $N_\pm$ are the dimensions of the corresponding spaces. The optimal choice for the degrees $p^+$ and $p^-$ and the regularities $r^+$ and $r^-$ will be discussed later. The interface space with approximate $C^1$ continuity is given as
\begin{align}
\widetilde{\mathcal{A}}_{\Gamma} = \widetilde{\mathcal{A}}_{\Gamma,+} \oplus \widetilde{\mathcal{A}}_{\Gamma,-}, \label{eq::spaceOfApproxBasisfunction} 
\end{align}
with
\begin{align*}
\widetilde{\mathcal{A}}_{\Gamma,+} = \text{span} \{ \widetilde{B}_{(j, +)} : 1 \leq j \leq N_+ \} \quad\mbox{ and }\quad \widetilde{\mathcal{A}}_{\Gamma,-} = \text{span} \{ \widetilde{B}_{(j, -)} : 1 \leq j \leq N_- \},
\end{align*}
where
\begin{align*}
\widetilde{B}_{(j, \pm)} = \begin{cases}
                ( \widetilde{f}^{(L)}_{(j, \pm)} \circ (\f F^{(L)})^{-1} ) (\boldsymbol{x}) & \text{if } \boldsymbol{x} \in \Omega^{(L)} \\
                ( \widetilde{f}^{(R)}_{(j, \pm)} \circ (\f F^{(R)})^{-1} ) (\boldsymbol{x}) & \text{if } \boldsymbol{x} \in \Omega^{(R)}
                  \end{cases}
\end{align*}
and
\begin{align}
\begin{split}
  \widetilde{f}^{(S)}_{(j,+)} (u,v) &= b_j^+ (v) \left( b_{1,1}^{(S)} (u) + b_{1,2}^{(S)} (u) \right) +   \widetilde{\beta}^{(S)} (v) (b_j^+) ' (v) \frac{h^{(S)}_1}{p_1^{(S)}} b_{1,2}^{(S)} (u) \label{eq::approxc1basisfunction} \\
  \widetilde{f}^{(S)}_{(j,-)} (u,v) &= \widetilde{\alpha}^{(S)} (v) b_j^- (v) \frac{h^{(S)}_1}{p_1^{(S)}} b_{1,2}^{(S)} (u) .
  \end{split}
\end{align}
We have by definition 
\begin{align}
 \widetilde{f}^{(S)}_{(j,\pm)} \in \mathcal{S}(p^{(S)}_1,r^{(S)}_1,h^{(S)}_1) \otimes \mathcal{S}(p_2^*,(r_2^*,\hat{r}_2-1),(h_2,\hat{h}_2)), \label{eq::spaceforbasisfunctions}
\end{align}
where
\[
 p_2^* = \max(p^+,p^+ +\widetilde{p}-1,p^- + \widetilde{p})
\]
and
\[
 r_2^* = \min(\widetilde{r},r^+-1,r^-).
\]
Depending on $p_2^*$, $r_2^*$ and $\hat{r}_2$, the interface space $\widetilde{\mathcal{A}}_{\Gamma}$ is clearly not necessarily a subspace of $\mathcal{V}^0_h$ and therefore does not yield an isoparametric discretization. Note that the isogeometric concept is violated by using a spline space of lower regularity and (in general) higher degree near the interface. A brief overview of the steps for constructing the interface spaces is shown in Figure~\ref{fig::mappinginterfacespaces}.

\begin{figure}[h!]
\centering
\includegraphics[width=0.8\textwidth]{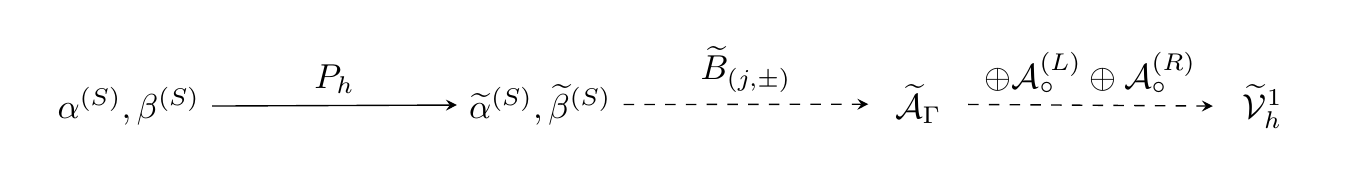}
\caption{The steps for constructing the approximate $C^1$ space $\widetilde{\mathcal{V}}^{1}_h$: the dashed arrows illustrate the construction steps and the solid line corresponds to the projection step.} \label{fig::mappinginterfacespaces}
\end{figure}

\subsubsection{Optimal choice of the spline parameters}

In order to be refineable spline spaces, the polynomial degrees and regularities of $\mathcal{S}^+$ and $\mathcal{S}^-$ need to satisfy
\[
 r^+ \leq p^+ - 1 \quad \mbox{ and } \quad r^- \leq p^- - 1.
\]
The functions $\widetilde{f}^{(S)}_{(j,\pm)}$ as in~\eqref{eq::spaceforbasisfunctions} have to satisfy $\widetilde{f}^{(S)}_{(j,\pm)} \in C^1(\widehat\Omega)$, in accordance with Assumption~\ref{ass::minimumregularity}. Thus, the regularity $r_2^*$ needs to satisfy
\[
 r_2^* = \min(\widetilde{r},r^+-1,r^-) \geq 1.
\]
To be able to reproduce traces and (approximate) normal derivatives of optimal order, that is, of degree $p_2$ and $p_2-1$, respectively, the degrees for $\mathcal{S}^+$ and $\mathcal{S}^-$ need to satisfy $p^+= p_2$ and $p^- = p_2-1$, respectively. Note that a higher polynomial degree for $\mathcal{S}^+$ and $\mathcal{S}^-$ will not improve the global approximation properties.

To summarize, we obtain $\widetilde{r} \geq 1$ as well as
\[
 2 \leq r^+ \leq p^+ - 1 = p_2 - 1
\]
and
\[
 1 \leq r^- \leq p^- - 1 = p_2 - 2.
\]
From now on, we choose the maximum regularity for the spaces $S^+$ and $S^-$ in order to achieve the smallest number of degrees of freedom, i.e., $r^+ = p_2 -1$ and $r^- = p_2-2$. We conclude from these restrictions that the degree $p_2$ needs to satsify $p_2 \geq 3$.
\begin{ass}[Minimum polynomial degree at the interface] \label{ass::minimumdegree}
 We assume that the polynomial degree for the discrete space at the interface fulfills $p_2 \geq 3$.
\end{ass}

In Section~\ref{sec::propertiesinterfacespaces}, we discuss the role of the approximated gluing data. Theorem~\ref{thm::boundsjumptV1} and Conjecture~\ref{con::aprioirierror} and the numerical experiments show how to choose the degree $\widetilde{p}$ and regularity $\widetilde{r}$ of the approximated gluing data in order to get optimal convergence rates. To obtain a sufficiently smooth spline approximation of the gluing data, we need the following.
\begin{ass}[Requirement on the approximation of the gluing data] \label{ass::spaceGluingData}
 If the gluing data are not polynomial functions, we assume that  the approximated gluing data is computed from $\mathcal{S}(\widetilde{p},(\widetilde{r},\hat{r}_2-1),(h_2,\hat{h}_2))$, with $1 \leq \widetilde{r} \leq \widetilde{p}-1$.
\end{ass}
If the gluing data are polynomials of low degree, they can be reproduced exactly using the space $\mathbb{P}^{\widetilde{p}}=\mathcal{S}(\widetilde{p},\infty,1)$, whereas if they are polynomials of high degree, they can be approximated using splines from $\mathcal{S}(\widetilde{p},\widetilde{p}-1,h_2)$. Throughout the following section we assume that the gluing data are not polynomial functions.

\section{Properties of the approximate $C^1$ space} \label{sec::propertiesinterfacespaces}

In this section, the properties of the approximated gluing data as well of the approximate $C^1$ space $\widetilde{\mathcal{V}}^{1}_h$ are studied. The properties of the projector $P_{h}$, which is used to define the approximated gluing data, are described in Subsection~\ref{sec::projector}. In Subsection~\ref{sec::errorforapproximatespace}, the boundedness of the jump of the normal derivative at the interface is proven. We can show that the convergence rate of this error depends only on the polynomial degree $\widetilde{p}$ of the approximated gluing data. In Subsection~\ref{sec::vanishing-jump}, we study the special case when the jump of the normal derivative vanishes and we introduced the AS-$G^1$ case. We conclude this section with a brief discussion of the two possible cases at the boundary, see Subsection~\ref{sec::smoothboundary}.

\subsection{Properties of the projection operator defining the approximate gluing data}\label{sec::projector}

The approximated gluing data is constructed using a projection operator \[
 P_{h} : C^{\hat{r}_2-1} ([0,1]) \to \mathcal{S}(\widetilde{p},(\widetilde{r}, \hat{r}_2-1),(h_2,\hat{h}_2)). 
\]
We assume that $1\leq \widetilde{r} \leq \widetilde{p}-1$ and that the operator satisfies the following properties:
\begin{itemize}
  \item it preserves splines, i.e., 
  \begin{align}
    P_{h} s_h = s_h, \qquad \forall \; s_h \in \mathcal{S}(\widetilde{p},(\widetilde{r},\hat{r}_2-1),(h_2,\hat{h}_2)), \label{projection::preservespline}
  \end{align}
  \item it is $L^\infty$-stable, i.e., there exists a constant $C>0$, such that
    \begin{align}
    \| P_{h} s \|_{L^\infty ([0,1])} \leq C \| s \|_{L^\infty ([0,1])}, \qquad   \forall \; s \in C^{\hat{r}_2-1}([0,1]),\label{eq::infinitiy-stable-projector}
  \end{align}
  \item it interpolates at the boundary, i.e.,
    \begin{align}
    [P_{h} s] (\bar{v}) = s (\bar{v}), \qquad   \mbox{ for } \bar{v} \in \{0,1\}, \; \forall \; s \in C^{\hat{r}_2-1} ([0,1]),\label{eq::requirement-boundary}
  \end{align}
  and
  \item it satisfies the following estimate: there exists a constant $C>0$, such that for all $0\leq j \leq {1}/{\hat{h}_2}-1$ we have
    \begin{align}
    \| s - P_{h} s \|_{L^\infty (\hat{\omega}_j)} \leq C h_2^{\widetilde{p}+1} \| s^{(\widetilde{p}+1)} \|_{L^\infty(\hat{\omega}_j)}, \qquad   \forall \; s \mbox{ with }s|_{\hat{\omega}_i} \in C^{\infty}(\hat{\omega}_j),\label{eq::approximation-projector}
  \end{align}
  where $\hat{\omega}_j = [j\hat{h}_2,(j+1)\hat{h}_2]$ and $s^{(\widetilde{p}+1)}$ is the derivative of $s$ of order $\widetilde{p}+1$.
\end{itemize}
For each subinterval $\hat{\omega}_j$ there exists a local projector satisfying~\eqref{eq::infinitiy-stable-projector} and~\eqref{eq::approximation-projector}, cf.~\cite[Theorem 6.25]{schumaker_2007}. A modification of that projector, interpolating function values at the global boundary and derivatives up to order $\hat{r}_2-1$ at all inner knots $j\hat{h}_2$, yields a global projector satisfying~\eqref{projection::preservespline} and~\eqref{eq::requirement-boundary}. Such a construction is similar to the one presented in~\cite[Proposition 3.2]{BeiraodaVeiga2012anisotropic}.

In practice, the required properties can be relaxed, since the projector is used only to prove a bound as in Proposition~\ref{prop::bounds-E1-E2}. Note that~\eqref{eq::requirement-boundary} is required to simplify the imposition of boundary conditions,   as described in Section~\ref{sec::smoothboundary}. The condition may be dropped, as discussed in Remark~\ref{rem::approximate-kernel}. A desirable property is that applying the operator to a piecewise rational function should be computationally cheap. In the following, the influence of the parameters $\widetilde{p}$ and $\widetilde{r}$ on the normal jump of the interface space is studied.

\subsection{Estimating the jump of the normal derivative of the approximate $C^1$ space} \label{sec::errorforapproximatespace}

Since the space $\widetilde{\mathcal{V}}^{1}_h$ is not exactly $C^1$-smooth but only approximately, we want to estimate the jump of the normal derivative across the interface. Let~$\f n$ be the normal vector to the interface, as in~\eqref{eq::normalvector}, and let $\varphi_h^{(S)} = \varphi_h |_{\Omega^{(S)} \cup \Gamma}$ with $\varphi_h \in \widetilde{\mathcal{V}}^{1}_h$. We denote by 
\[
 \partial_{\f n} \varphi_h^{(S)} (\f x) = \lim_{\Omega^{(S)}\ni \f x^{(S)}\rightarrow \f x}\nabla \varphi_h^{(S)}  (\f x^{(S)}) \cdot \f n
\]
the normal derivative of $\varphi_h^{(S)}$ at $\f x \in \Gamma$, defined as a limit. Similarly, the tangential derivative along the interface is expressed by 
\[
 \partial_{\f t_0} \varphi_h (\f x) = \partial_{\f t_0} \varphi_h^{(S)} (\f x) = \lim_{\Omega^{(S)}\ni \f x^{(S)}\rightarrow \f x}\nabla \varphi_h^{(S)}  (\f x^{(S)}) \cdot \f t_0.
\]
Note that the tangential derivative is continuous and therefore the limit is well-defined and does not depend on the side. The jump of the normal derivative is defined as
\[
 \jump{\partial_{\f n} \varphi_h}(\f x) = \partial_{\f n} \varphi_h^{(R)} (\f x) - \partial_{\f n} \varphi_h^{(L)} (\f x).
\]
It satisfies the following bound.
\begin{thm}\label{thm::boundsjumptV1}
Let $h_2$ be small enough and the gluing data $\alpha^{(S)}, \beta^{(S)} \in C^{\hat{r}_2-1} ([0,1])$
. Then we have for all $\varphi_h \in \widetilde{\mathcal{A}}_{\Gamma}$ that
\begin{align*}
 \left\| \jump{\partial_{\f n} \varphi_h} \right\|_{L^2(\Gamma)} \leq C h_2^{\widetilde{p}+1} \left( \left\| \varphi_h \right\|^2_{H^2(\Omega^{(R)})} + \left\| \varphi_h  \right\|^2_{H^2(\Omega^{(L)})} \right)^{1/2},
\end{align*}
where $C>0$ depends on the geometry, but not on the mesh size.   For certain configurations the jump may vanish. This is characterized in Section~\ref{sec::vanishing-jump}.
\end{thm}
Note that $h_2$ needs to be sufficiently small, such that 
\[
 \frac{|\alpha^{(S)} (v) - \widetilde{\alpha}^{(S)} (v)|}{|\alpha^{(S)} (v)|} \leq \varepsilon < 1
\]
for all $v\in[0,1]$ and for $S\in\{L,R\}$. Such a bound on $h_2$ always exists, since $\widetilde{\alpha}^{(S)}$ converges to ${\alpha}^{(S)}$ pointwise and $1/{\alpha}^{(S)}$ is bounded from above. Before we can state the proof of Theorem~\ref{thm::boundsjumptV1}, some preliminary estimates are needed. 
The jump of the normal derivative satisfies the following pointwise representation.
\begin{prop}\label{prop::jump-representation}
We have for $ \varphi_h \in \widetilde{\mathcal{A}}_{\Gamma} $ and for all $\f x \in \Gamma$, with $\f x = \f F^{(L)}(0,v)$ and $v \in [0,1]$, that
\begin{align*}
 \jump{\partial_{\f n} \varphi_h} (\f x)
 & = E_1^{(L)}(v) \, \partial_{\f t_0} \varphi_h(\f x) + E_2^{(L)}(v) \, \partial_{\f n} \varphi_h^{(L)}(\f x) 
\end{align*}
where
\begin{align*}
 E_1^{(L)} = \tau^2\left(\frac{\widetilde{\alpha}^{(R)} (\beta^{(L)}-\widetilde{\beta}^{(L)})}{{\alpha}^{(R)} \widetilde{\alpha}^{(L)} }-\frac{\beta^{(R)}-\widetilde{\beta}^{(R)}}{{\alpha}^{(R)}}\right)
 \quad \text{and} \quad
 E_2^{(L)} = \left( \frac{\widetilde{\alpha}^{(R)} (\alpha^{(L)} - \widetilde{\alpha}^{(L)} )}{\alpha^{(R)}\widetilde{\alpha}^{(L)} } - \frac{{\alpha}^{(R)} -\widetilde{\alpha}^{(R)}}{\alpha^{(R)}} \right).
\end{align*}
Due to the symmetry of the construction, a similar statement is valid with switched sides.
\end{prop}
\begin{proof}
Using the definition of $\varphi_h$ we have $f_h^{(S)}(0,v) = \varphi_h \circ \f F^{(S)}(0,v)$, for $S\in\{L,R\}$. Applying the chain rule results in
\[
    \partial_v f_h^{(S)} (0,v) = \nabla \varphi_h \circ \f F^{(S)}(0,v) \cdot \partial_v \f F^{(S)}(0,v) = \tau(v) \partial_{\f t_0}\varphi_h \circ \f F^{(S)}(0,v).
\]
Similarly, we have
\[
\partial_u f_h^{(S)} (0,v) = \nabla \varphi_h \circ \f F^{(S)}(0,v) \cdot \partial_u \f F^{(S)}(0,v).
\]
Furthermore, one can describe $\partial_u \f F^{(S)}(0,v)$ as a linear combination of $\f t_0$ and $\f n$, like in the proof of Proposition~\ref{prop::normalvector},
\begin{align*}
 \partial_u \f F^{(S)}(0,v) &= \beta^{(S)} (v)\tau(v) \f t_0(v) + \frac{\alpha^{(S)}(v) }{\tau(v)} \f n(v) 
\end{align*}
and it follows
\begin{align*}
\partial_u f_h^{(S)} (0,v) &= \beta^{(S)} (v) \tau(v) \partial_{\f t_0} \varphi_h \circ \f F^{(S)}(0,v) + \frac{\alpha^{(S)}(v) }{\tau(v)} \partial_{\f n} \varphi_h \circ \f F^{(S)}(0,v).
\end{align*}
By construction, cf.~\eqref{eq::approxc1basisfunction}, ${f}_{h}^{(S)}$ satisfies
\begin{align*}
  \partial_v {f}_{h}^{(S)} (0,v) &= G_1 (v) \\
  \partial_u {f}_{h}^{(S)} (0,v) &= \widetilde{\beta}^{(S)} (v) G_1 (v) + \widetilde{\alpha}^{(S)} (v) G_2(v)
\end{align*}
for some functions $G_1$ and $G_2$, which are independent of the side $S\in\{L,R\}$. Hence, we get 
\[
 G_1(v) = \tau(v) \partial_{\f t_0}\varphi_h \circ \f F^{(S)}(0,v)
\]
and
\begin{align*}
  \partial_u {f}_{h}^{(S)} (0,v) &= \widetilde{\beta}^{(S)} (v) G_1 (v) + \widetilde{\alpha}^{(S)} (v) G_2(v) = \beta^{(S)} (v) G_1(v) + \frac{\alpha^{(S)}(v) }{\tau(v)} \partial_{\f n} \varphi_h \circ \f F^{(S)}(0,v)
\end{align*}
and consequently
\begin{align*}
  G_2(v) = G_2^{(S)}(v) = \left( \frac{\alpha^{(S)}(v)}{\widetilde{\alpha}^{(S)} (v)\tau(v)} \right)\partial_{\f n} \varphi_h^{(S)} \circ \f F^{(S)}(0,v)  + \left(\tau(v)\frac{\beta^{(S)}(v)-\widetilde{\beta}^{(S)}(v)}{\widetilde{\alpha}^{(S)} (v)}\right)\partial_{\f t_0} \varphi_h \circ \f F^{(S)}(0,v),
\end{align*}
independent of the side $S\in\{L,R\}$. We obtain
\begin{align*}
  0 = G_2^{(R)}(v) - G_2^{(L)}(v) = & \left( \frac{\alpha^{(R)}(v)}{\widetilde{\alpha}^{(R)} (v)\tau(v)} \right)\partial_{\f n} \varphi_h^{(R)}(\f x)  + \left(\tau(v)\frac{\beta^{(R)}(v)-\widetilde{\beta}^{(R)}(v)}{\widetilde{\alpha}^{(R)} (v)}\right)\partial_{\f t_0} \varphi_h(\f x) \\
  & - \left( \frac{\alpha^{(L)}(v)}{\widetilde{\alpha}^{(L)} (v)\tau(v)} \right)\partial_{\f n} \varphi_h^{(L)}(\f x)  - \left(\tau(v)\frac{\beta^{(L)}(v)-\widetilde{\beta}^{(L)}(v)}{\widetilde{\alpha}^{(L)} (v)}\right)\partial_{\f t_0} \varphi_h(\f x).
\end{align*}
In the following we replace $\partial_{\f n} \varphi_h^{(R)}$ by
\begin{align}
 \partial_{\f n} \varphi_h^{(R)}(\f x) = \partial_{\f n} \varphi_h^{(L)}(\f x) + \jump{\partial_{\f n} \varphi_h}(\f x).\label{eq::replacing-phiR}
\end{align}
If instead $\partial_{\f n} \varphi_h^{(L)}$ is replaced by $\partial_{\f n} \varphi_h^{(R)}-\jump{\partial_{\f n} \varphi_h}$, we get a representation for the right patch. Using~\eqref{eq::replacing-phiR} we obtain
\begin{align*}
  \left( \frac{\alpha^{(R)}(v)}{\widetilde{\alpha}^{(R)} (v)\tau(v)} \right)\jump{\partial_{\f n} \varphi_h}(\f x) = & \left( \frac{\alpha^{(L)}(v)}{\widetilde{\alpha}^{(L)} (v)\tau(v)}-\frac{\alpha^{(R)}(v)}{\widetilde{\alpha}^{(R)} (v)\tau(v)} \right)\partial_{\f n} \varphi_h^{(L)}(\f x)  \\
  & + \tau(v)\left(\frac{\beta^{(L)}(v)-\widetilde{\beta}^{(L)}(v)}{\widetilde{\alpha}^{(L)} (v)}-\frac{\beta^{(R)}(v)-\widetilde{\beta}^{(R)}(v)}{\widetilde{\alpha}^{(R)} (v)}\right)\partial_{\f t_0} \varphi_h(\f x).
\end{align*}
Multiplying both sides with $\widetilde{\alpha}^{(R)} \tau / {\alpha}^{(R)}$ yields
\begin{align*}
  \jump{\partial_{\f n} \varphi_h}(\f x) = & \left( \frac{\widetilde{\alpha}^{(R)} (v)(\alpha^{(L)}(v) - \widetilde{\alpha}^{(L)} (v))}{\alpha^{(R)}(v)\widetilde{\alpha}^{(L)} (v)} - \frac{{\alpha}^{(R)} (v)-\widetilde{\alpha}^{(R)}(v)}{\alpha^{(R)}(v)} \right)\partial_{\f n} \varphi_h^{(L)}(\f x)  \\
  & + \tau(v)^2\left(\frac{\widetilde{\alpha}^{(R)} (v)(\beta^{(L)}(v)-\widetilde{\beta}^{(L)}(v))}{{\alpha}^{(R)} (v)\widetilde{\alpha}^{(L)} (v)}-\frac{\beta^{(R)}(v)-\widetilde{\beta}^{(R)}(v)}{{\alpha}^{(R)} (v)}\right)\partial_{\f t_0} \varphi_h(\f x),
\end{align*}
which concludes the proof.
\end{proof}
Due to the approximation using the projector $P_{h}$, the following estimates for the factors $E_1^{(L)}$ and $E_2^{(L)}$ are obtained.
\begin{prop}\label{prop::bounds-E1-E2}
Let the assumptions of Theorem~\ref{thm::boundsjumptV1} be satisfied. Then we have 
\[
 \left\|E_1^{(L)} \right\|_{L^\infty([0,1])} \leq C h_2^{\widetilde{p}+1}
\]
as well as
\[
 \left\|E_2^{(L)}\right\|_{L^\infty([0,1])} \leq C h_2^{\widetilde{p}+1},
\]
where the constant $C$ depends on $p$, and on the geometry mappings $\f F^{(L)}$ and $\f F^{(R)}$, but not on the mesh size.
\end{prop}
\begin{proof}
Throughout the proof, all norms are to be considered ${L^\infty}$-norms. We have
\begin{align*}
 \left\|E_1^{(L)}\right\| &= \left\|\tau^2\left(\frac{\widetilde{\alpha}^{(R)} (\beta^{(L)}-\widetilde{\beta}^{(L)})}{{\alpha}^{(R)} \widetilde{\alpha}^{(L)} }-\frac{\beta^{(R)}-\widetilde{\beta}^{(R)}}{{\alpha}^{(R)}}\right)\right\|\\
 & \leq \|\tau \|^2 \left\|\frac{1}{{\alpha}^{(R)}} \right\| \left(  \left\|\frac{1}{\widetilde{\alpha}^{(L)}}\right\| \|\widetilde{\alpha}^{(R)}\| \|\beta^{(L)}-\widetilde{\beta}^{(L)}\| +\|\beta^{(R)}-\widetilde{\beta}^{(R)}\|\right)
\end{align*}
and similarly
\begin{align*}
 \left\|E_2^{(L)}\right\| \leq \left\|\frac{1}{{\alpha}^{(R)}} \right\| \left(  \left\|\frac{1}{\widetilde{\alpha}^{(L)}}\right\| \|\widetilde{\alpha}^{(R)}\| \|\alpha^{(L)}-\widetilde{\alpha}^{(L)}\| +\|\alpha^{(R)}-\widetilde{\alpha}^{(R)}\|\right).
\end{align*}
The terms
\[
  \|\tau \|^2 \quad\mbox{ and }\quad \left\|\frac{1}{{\alpha}^{(R)}} \right\|
\]
are bounded by definition and depend only on the geometry mapping $\f F^{(R)}$. Due to~\eqref{eq::infinitiy-stable-projector} the term $\|\widetilde{\alpha}^{(R)}\|$ is bounded from above by $C\|{\alpha}^{(R)}\|$, which in turn depends only on $\widetilde{p}$ and on $\f F^{(R)}$. Estimate~\eqref{eq::approximation-projector} yields 
\[
 \|\alpha^{(S)}-\widetilde{\alpha}^{(S)}\| \leq C_{\f F^{(S)}} h_2^{\widetilde{p}+1},
\]
and
\[
 \|\beta^{(S)}-\widetilde{\beta}^{(S)}\| \leq C_{\f F^{(S)}} h_2^{\widetilde{p}+1},
\]
where the constants depend only on $\widetilde{p}$ and on $\f F^{(S)}$. What remains to be shown is an estimate from above for $\left\|{1}/{\widetilde{\alpha}^{(L)}}\right\|$. Due to the regularity of patch $\f F^{(L)}$, we have 
\[
 0 < c \leq |\alpha^{(L)} (v)|.
\]
As $h_2\rightarrow 0$ we have $|\alpha^{(L)} (v) - \widetilde{\alpha}^{(L)} (v)| \rightarrow 0$. Hence, for all $\varepsilon>0$ there exists a $\delta>0$ such that for all $h<\delta$ we have $|\alpha^{(L)} (v) - \widetilde{\alpha}^{(L)} (v)| < c \, \varepsilon \leq |\alpha^{(L)} (v)| \, \varepsilon$. We have
\[
 (1-\varepsilon)|{\alpha}^{(L)} (v)| \leq |{\alpha}^{(L)} (v)| - |{\alpha}^{(L)} (v)-\widetilde{\alpha}^{(L)} (v)| \leq |\widetilde{\alpha}^{(L)} (v)|
\]
and consequently 
\[
 \left\|\frac{1}{\widetilde{\alpha}^{(L)}}\right\| \leq \frac{1}{1-\varepsilon} \left\| \frac{1}{{\alpha}^{(L)}} \right\|,
\]
which is bounded from above by a constant that depends only on $\f F^{(L)}$, if $|\alpha^{(L)} (v) - \widetilde{\alpha}^{(L)} (v)| < |\alpha^{(L)} (v)| \, \varepsilon$ for some $\varepsilon<1$. Such an $\varepsilon$ exists if $h_2$ is sufficently small. This concludes the proof.
\end{proof}

Based on the projection $P_h$ which interpolates the boundary, see~\eqref{eq::requirement-boundary}, it follows that $\widetilde{\alpha}^{(S)} (\bar{v}) = \alpha^{(S)} (\bar{v})$ and $\widetilde{\beta}^{(S)} (\bar{v}) = \beta^{(S)} (\bar{v})$ and consequently $E_1^{(S)} (\bar{v}) = E_2^{(S)} (\bar{v}) = 0$ for $\bar{v} \in \{0,1\}$. In addition, it may happen that for certain points $v \in (0,1)$ the approximated gluing data $\widetilde{\alpha}^{(S)} (v)$ and/or $\widetilde{\beta}^{(S)} (v)$ match with the gluing data $\alpha^{(S)} (v)$ and/or $\beta^{(S)} (v)$. In this case, the corresponding factor also vanishes at $v$. We can now proof Theorem~\ref{thm::boundsjumptV1}.
\begin{proof}[Proof of Theorem~\ref{thm::boundsjumptV1}]
From Proposition~\ref{prop::jump-representation} and \ref{prop::bounds-E1-E2} we obtain
\begin{align*}
 \| \jump{\partial_{\f n} \varphi_h} \|^2_{L^2(\Gamma)} = \int_{\Gamma}  \jump{\partial_{\f n} \varphi_h}^2 \mathrm{d}\f x &= \int_{\Gamma}  \left(E_1^{(L)} \, \partial_{\f t_0} \varphi_h + E_2^{(L)} \, \partial_{\f n} \varphi_h^{(L)}\right)^2 \mathrm{d}\f x \\
 &\leq \max(\|E_1^{(L)}\|^2_{L^\infty},\|E_2^{(L)}\|^2_{L^\infty}) \int_{\Gamma}  \left(\partial_{\f t_0} \varphi_h + \partial_{\f n} \varphi_h^{(L)}\right)^2 \mathrm{d}\f x \\
 &\leq C' h_2^{2\widetilde{p}+2} \| \nabla \varphi_h^{(L)} \|^2_{L^2(\Gamma)},
\end{align*}
where $C'$ depends only on $\widetilde{p}$ and on the geometry parametrizations, but not on $h$ or $\varphi_h$. For reasons of symmetry, the same bound is valid for $\varphi_h^{(R)}$. Using a standard Sobolev trace inequality, cf.~\cite{adams2003sobolev}, yields
\[
 \| \nabla \varphi_h^{(L)} \|_{L^2(\Gamma)} \leq C \| \varphi_h^{(L)} \|_{H^2(\Omega^{(L)})},
\]
which concludes the proof.
\end{proof}

\subsection{A special case: vanishing jumps}\label{sec::vanishing-jump}

From Proposition~\ref{prop::jump-representation} it can be concluded that the jump of the normal derivative is zero at $v \in [0,1]$ if the factors $E_1^{(L)} (v) = E_2^{(L)} (v) = 0$, which is the case when
\begin{align}
  \beta(v) = {\alpha}^{(L)}(v)\widetilde{\beta}^{(R)}(v)-{\alpha}^{(R)}(v)\widetilde{\beta}^{(L)}(v) \quad \text{ and } \quad \widetilde{\alpha}^{(S)} (v) = c\cdot \alpha^{(S)}(v) \qquad \forall \; S \in \{L,R\}, \label{eq::gluingdataequalapproxgluingdata}
\end{align}
for some $c\neq 0$. Obviously, a sufficient condition is $\widetilde{\alpha}^{(S)}(v) = \alpha^{(S)}(v)$ and $\widetilde{\beta}^{(S)}(v) = \beta^{(S)}(v)$ for $S\in \{L,R\}$. The condition holds for all $v \in [0,1]$ if the gluing data satisfies $\alpha^{(S)},\beta^{(S)}\in \mathcal{S}(\widetilde{p},(\widetilde{r},\hat{r}),(h_2,\hat{h}_2))$, see~\eqref{projection::preservespline}. Hence, we obtain the following Proposition.
\begin{prop}
  Let $\alpha^{(S)}$ and $\beta^{(S)}$ be in the spline space $\mathcal{S}(\widetilde{p},(\widetilde{r},\hat{r}),(h_2,\hat{h}_2))$ and \eqref{eq::gluingdataequalapproxgluingdata} be satisfied. Then the interface space $\widetilde{\mathcal{A}}_{\Gamma}$ is $C^1$-smooth and, consequently, $\widetilde{\mathcal{V}}^{1}_h \subseteq C^{1}(\Omega)$.
\end{prop}
Moreover, when the gluing data are linear polynomials, i.e., $\alpha^{(S)},\beta^{(S)} \in \mathbb{P}^1$, then the interface basis functions are by definition in the space $\mathcal{S}(p^{(S)}_1,r^{(S)}_1,h^{(S)}_1) \otimes \mathcal{S}(p_2,r_2,h_2)$ for $1\leq r_2 \leq p_2 - 2$, cf.~\eqref{eq::spaceforbasisfunctions}. Those geometries are also known as analysis-suitable $G^1$ geometries, as introduced in~\cite{collin2016analysis}. We repeat the definition here.
\begin{defi}\label{def::asg1}
 A two-patch geometry is called \emph{analysis-suitable $G^1$}, in short AS-$G^1$, if there exists linear gluing data, i.e., $\alpha^{(L)},\alpha^{(R)},\beta^{(L)},\beta^{(R)} \in \mathbb{P}^1$, such that~\eqref{eq::betaLbetaR} and~\eqref{eq::exactgluingdatag1smooth} are satisfied. 
\end{defi}
Note that the AS-$G^1$ condition requires the existence of linear gluing data. However, we define $\alpha^{(L)}$, $\alpha^{(R)}$ through the formulas in~\eqref{def::alpha_beta}, with $\gamma \equiv 1$, and $\beta^{(L)}$, $\beta^{(R)}$ through~\eqref{def::beta_S}. Hence, even though the gluing data we compute is not linear, there might exist linear gluing data for a different choice of $\gamma$ or a different splitting of $\beta = \alpha^{(L)} \beta^{(R)} - \alpha^{(R)} \beta^{(L)}$. Bilinear patches are AS-$G^1$ and they yield linear gluing data for $\gamma \equiv 1$. Thus, we have the following for bilinear patches.
\begin{prop}
   If the geometry is AS-$G^1$ with $\gamma \equiv 1$, e.g., if both patches are bilinear, then $\widetilde{\mathcal{V}}^{1}_h \subseteq \mathcal{V}^1_h$. \label{prop::asthanc1}
\end{prop}
Hence, when the parametrization is piecewise bilinear, then the isogeometric concept remains and the jump vanishes. It is shown numerically in~\cite{collin2016analysis, kapl2017dimension, kapl2019argyris}, that the convergence rates for general AS-$G^1$ geometries are optimal when solving fourth order problems. Moreover, optimal approximation error bounds were proven in~\cite{kapl2020family} for bilinear patches.

\subsection{Functions with vanishing trace at the domain boundary}\label{sec::smoothboundary}

In this section, we characterize the space of interior functions $\widetilde{\mathcal{V}}^1_{h,0} = \widetilde{\mathcal{V}}^1_{h} \cap H^1_0(\Omega)$ of functions with vanishing trace on the domain boundary. We denote its complement, used to impose non-homogeneous boundary conditions, by $\widetilde{\mathcal{V}}^1_{h,\partial\Omega}$. By definition, cf. in~\eqref{eq::approxc1basisfunction}, one can see that the interface basis functions which are, in general, not vanishing at the boundary corresponding to $\bar{v}=0$ that is $\{\f F^{(L)}(u,0)\}_{u\in[0,1]} \cup \{\f F^{(R)}(u,0)\}_{u\in[0,1]}$, are $\widetilde{B}_{(1,+)}$, $\widetilde{B}_{(2,+)}$ and $\widetilde{B}_{(1,-)}$. The pull-backs of those three functions satisfy 
\begin{align*}
 \widetilde{f}^{(S)}_{(1,+)} (u,0) &= b_{1,1}^{(S)} (u) + b_{1,2}^{(S)} (u) \left( 1 + \widetilde{\beta}^{(S)} (0) (b_1^+) ' (0) \frac{h^{(S)}_1}{p_1^{(S)}} \right), \\
 \widetilde{f}^{(S)}_{(2,+)} (u,0) &= \widetilde{\beta}^{(S)} (0) (b_2^+) ' (0) \frac{h^{(S)}_1}{p_1^{(S)}} b_{1,2}^{(S)} (u), \\
  \widetilde{f}^{(S)}_{(1,-)} (u,0) &= \widetilde{\alpha}^{(S)} (0) \frac{h^{(S)}_1}{p_1^{(S)}} b_{1,2}^{(S)} (u).
\end{align*}
Analogously, we can define the interface functions that do not vanish at the boundary corresponding to $\bar{v}=1$. To obtain the correct subspace $\mbox{span}\{\widetilde{B}_{(1,+)},\widetilde{B}_{(2,+)},\widetilde{B}_{(1,-)}\} \cap H^1_{0}(\Omega)$ we need to compute the kernel of the space, evaluated at the domain boundary, that is
\begin{align*}
 \text{ker} \langle \{ \widetilde{B}_{(1,+)}, \widetilde{B}_{(2,+)}, \widetilde{B}_{(1,-)} \} \rangle = \left\{ \varphi = \lambda_1 \widetilde{B}_{(1,+)} + \lambda_2 \widetilde{B}_{(2,+)} + \lambda_3 \widetilde{B}_{(1,-)},\mbox{ with }\; \lambda_i \in \mathbb{R}, \, i \in \{1,2,3\}: \varphi |_{\partial\Omega} = 0 \right\}.
\end{align*}
We can describe any function $\varphi \in \text{ker} \langle \{ \widetilde{B}_{(1,+)}, \widetilde{B}_{(2,+)}, \widetilde{B}_{(1,-)} \} \rangle$ by its pull-back to a patch and obtain
\begin{align*}
\lambda_1 \widetilde{f}^{(S)}_{(1,+)}(u,0) + \lambda_2 \widetilde{f}^{(S)}_{(2,+)}(u,0) + \lambda_3 \widetilde{f}^{(S)}_{(1,-)}(u,0) = 0
\end{align*}
for all $u\in [0,1]$ and $S\in\{L,R\}$. For $u = 0$ we have $\widetilde{f}^{(S)}_{(1,+)} (0,0) = 1$ and $\widetilde{f}^{(S)}_{(2,+)} (0,0) = \widetilde{f}^{(S)}_{(1,-)} (0,0) = 0$. As a consequence, we have
\[
\lambda_1 = 0. 
\]
As a result we obtain the conditions
\begin{align*}
\lambda_2 \widetilde{f}^{(S)}_{(2,+)}(u,0) + \lambda_3 \widetilde{f}^{(S)}_{(1,-)}(u,0) = 0
\end{align*}
for all $u \in [0,1]$ and $S\in\{L,R\}$. Using the definitions of the functions, we get
\begin{align*}
\lambda_2 \widetilde{\beta}^{(S)} (0) (b_2^+) ' (0) \frac{h^{(S)}_1}{p_1^{(S)}} b_{1,2}^{(S)} (u) + \lambda_3 \widetilde{\alpha}^{(S)} (0) \frac{h^{(S)}_1}{p_1^{(S)}} b_{1,2}^{(S)} (u) = 0.
\end{align*}
Since this equation needs to be satisfied for all $u$, we can cancel out the factor 
$\frac{h^{(S)}_1}{p_1^{(S)}} b_{1,2}^{(S)} (u)$ and get
\begin{align*}
\lambda_2 \widetilde{\beta}^{(L)} (0) \frac{p_2}{h_2} + \lambda_3 \widetilde{\alpha}^{(L)} (0) = 0
\end{align*}
and
\begin{align*}
\lambda_2 \widetilde{\beta}^{(R)} (0) \frac{p_2}{h_2} + \lambda_3 \widetilde{\alpha}^{(R)} (0) = 0.
\end{align*}
Consequently,
\begin{align*}
\lambda_3 = -\lambda_2 \frac{\widetilde{\beta}^{(L)} (0)}{\widetilde{\alpha}^{(L)} (0)} \frac{p_2}{h_2} = -\lambda_2 \frac{\widetilde{\beta}^{(R)} (0)}{\widetilde{\alpha}^{(R)} (0)} \frac{p_2}{h_2}.
\end{align*}
Hence, we obtain a non-trivial solution $(\lambda_1,\lambda_2,\lambda_3) \neq (0,0,0)$ if and only if
\begin{align*}
\frac{\widetilde{\beta}^{(L)} (0)}{\widetilde{\alpha}^{(L)} (0)} \frac{p_2}{h_2} = \frac{\widetilde{\beta}^{(R)} (0)}{\widetilde{\alpha}^{(R)} (0)} \frac{p_2}{h_2} = \psi,
\end{align*}
which is equivalent to $\widetilde{\beta}(0) = 0$, for $\widetilde{\beta}(v) \coloneqq \widetilde{\alpha}^{(L)} (v)\widetilde{\beta}^{(R)} (v) - \widetilde{\alpha}^{(R)} (v)\widetilde{\beta}^{(L)} (v)$. The kernel is then given by $(\lambda_1,\lambda_2,\lambda_3)=(0,c,-c \,\psi)$, $c \in \mathbb{R}$. Since the projector $P_h$ in~\eqref{eq:projector-gluingdata} interpolates at the boundary, we have $\widetilde{\beta}(0) = \beta(0)$. To summarize, we need to distinguish two cases:
\begin{enumerate}
 \item If ${\beta}(0) = 0$, which is equivalent to 
\begin{align}
 \partial_u \f F^{(L)} (0,0) \parallel \partial_u \f F^{(R)} (0,0), \label{con::partialuFS}
\end{align}
we have that $\text{ker} \langle \{ \widetilde{B}_{(1,+)}, \widetilde{B}_{(2,+)}, \widetilde{B}_{(1,-)} \} \rangle$ is spanned by 
\begin{align*}
\widetilde{B}_{(2,+)}^* \coloneqq \widetilde{B}_{(2,+)} - \psi \widetilde{B}_{(1,-)},
\end{align*}
 \item Otherwise, namely ${\beta}(0) \neq 0$, the kernel is empty.
\end{enumerate}
By properly modifying the functions $\widetilde{\beta}^{(L)}$ and $\widetilde{\beta}^{(R)}$, i.e., replacing $\widetilde{\beta}^{(S)}$ by
\[
 \widetilde{\beta}^{(S,*)} (v) = \widetilde{\beta}^{(S)} (v) - \frac{\widetilde{\beta}^{(L)} (0) }{ \widetilde{\alpha}^{(L)} (0) } \widetilde{\alpha}^{(S)} (v),
\]
for $S\in\{L,R\}$, we achieve $\psi = 0$. This simplifies the definition of the kernel, which is then spanned by $\widetilde{B}_{(2,+)}^* = \widetilde{B}_{(2,+)}$. Thus, if the boundary is smooth at $\bar{v}=0$, the function $\widetilde{B}_{(2,+)}^*$ belongs to $\widetilde{\mathcal{V}}^1_{h,0}$, whereas the functions $\widetilde{B}_{(1,+)}$ and $\widetilde{B}_{(1,-)}$ belong to $\widetilde{\mathcal{V}}^1_{h,\partial\Omega}$. If the boundary is not smooth, all functions belong to $\widetilde{\mathcal{V}}^1_{h,\partial\Omega}$. A similar modification can be achieved if the boundary is smooth at $\bar{v}=1$, that is, if $\beta(1)=0$, or if the boundary is smooth on both ends of the interface.

In Figure~\ref{fig::SeperationBoundary} an example of the two cases is depicted. There, the boundary is smooth at the lower end of the interface and non-smooth at the upper end.

\begin{figure}[h!]
 \centering
 \includegraphics[width=0.5\textwidth]{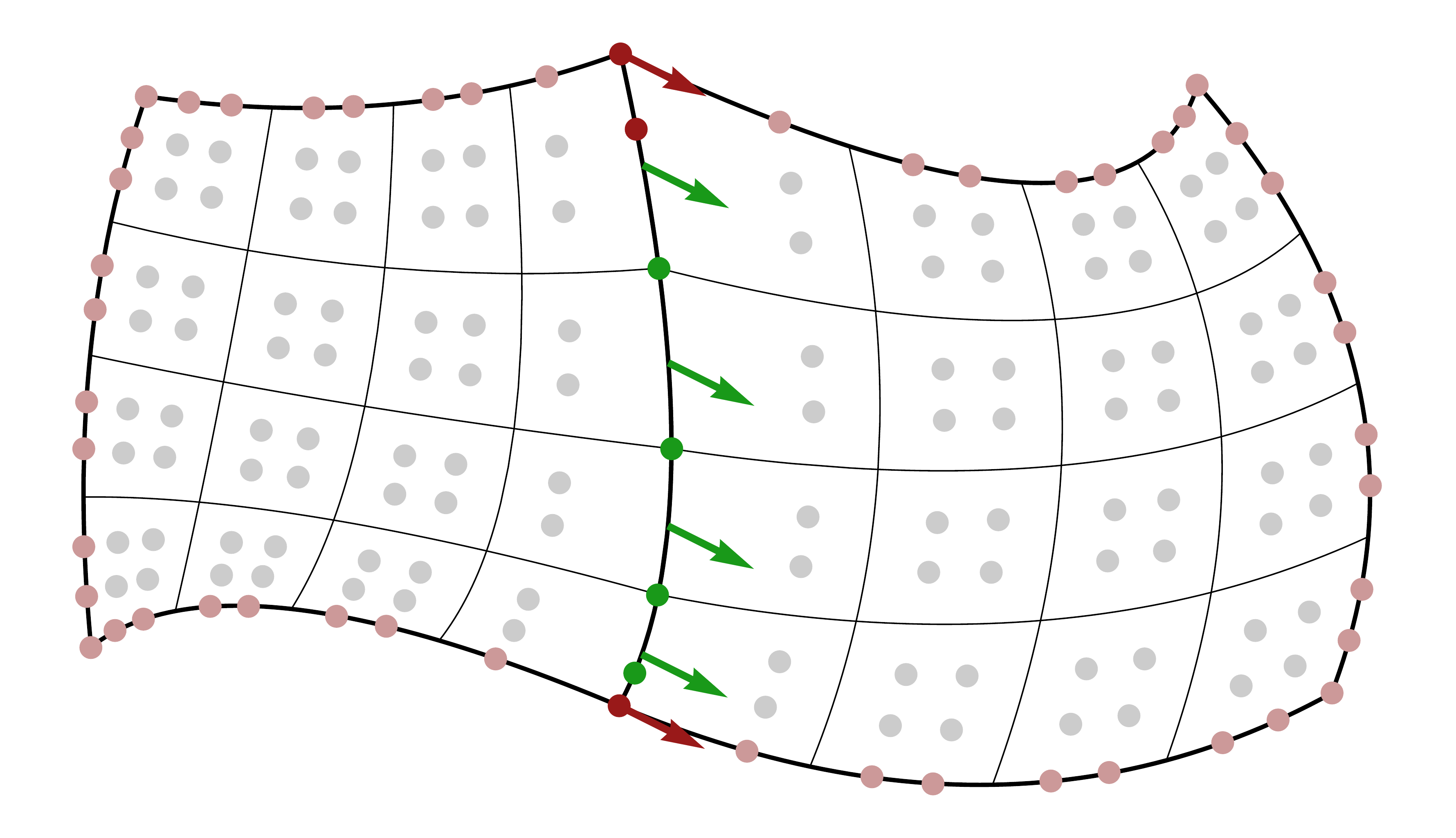}
 \caption{  Seperation of basis functions: the bullets away from the interface denote standard patch basis functions ${B}_{\f j}^{(S)}$, the bullets along the interface denote the trace functions $\widetilde{B}_{(j,+)}$ and the arrows the transversal derivative functions $\widetilde{B}_{(j,-)}$. Those basis functions depicted in red span the boundary space $\widetilde{\mathcal{V}}^1_{h,\partial\Omega}$ while the green and gray ones span the interior space $\widetilde{\mathcal{V}}^1_{h,0}$.} \label{fig::SeperationBoundary}
\end{figure}

\begin{rmk}
Note that the condition in~\eqref{eq::requirement-boundary}, requiring the projector $P_{h}$ to interpolate at the boundary, may be dropped. However, in that case the approximate gluing data $\widetilde{\beta}(v)= \widetilde{\alpha}^{(L)} (v)\widetilde{\beta}^{(R)} (v) - \widetilde{\alpha}^{(R)} (v)\widetilde{\beta}^{(L)} (v)$ in general does not satisfy $\widetilde{\beta}(\bar{v})= \beta(\bar{v})$ for $\bar{v}\in\{0,1\}$. We then have
\begin{align*}
 |\widetilde{\beta}(\bar{v}) - \beta(\bar{v})| \leq \norm{\widetilde{\beta} - \beta}_{L^\infty} \leq C h_2^{\widetilde{p}+1},
\end{align*}
for some constant $C$ that depends on the degree and on the exact geometry, similar to Proposition~\ref{prop::bounds-E1-E2}. Thus, one has to compute an approximate kernel up to an $h$-dependent tolerance.\label{rem::approximate-kernel}
\end{rmk}

\section{The two-patch formulation and discretization} \label{sec::twopatchformulation}

In this section, the discrete variational problem is stated. We consider a sequence of approximate $C^1$ spaces $\{\widetilde{\mathcal{V}}^1_h\}_h$, with ${\mathcal{V}}^0_{h} \subset {\mathcal{V}}^0_{h'}$ for $h'<h$. Note that the spaces $\widetilde{\mathcal{V}}^1_h$ are in general not nested, only their underlying $C^0$ spaces. The behaviour of the approximate solution as the mesh size $h$ goes to zero is studied numerically for different choices of polynomial degrees in Section~\ref{sec::numericalexperiments}.

\subsection{The non-conforming two-patch formulation}

We introduce the space
\begin{equation*}
\begin{array}{ll}
    \mathcal{X}_0 & \coloneqq \left\{ \psi \in H^1_0 (\Omega) \;|\; \psi|_{\Omega^{(S)}} \in H^2 (\Omega^{(S)}) \mbox{ for }S\in\{L,R\}  \right\},
\end{array}
\end{equation*} 
similar to a bent Sobolev space as introduced in~\cite{bazilevs2006isogeometric}, equipped with the norm
\begin{align*}
  \norm{\psi}^2_\mathcal{X} \coloneqq  \sum_{S \in \{L,R\} } \| \psi \|^2_{H^2(\Omega^{(S)})}.
\end{align*}
Let $\average{\circ}_\Gamma = \frac{1}{2}(\circ^{(L)}+\circ^{(R)})|_\Gamma$ denote the average and $\jump{\circ}_\Gamma =(\circ^{(R)}-\circ^{(L)})|_\Gamma$ denote the jump across the interface. We have the following problem.
\begin{problem}\label{problem:discontinuous-weakformulation}
 Find $\varphi \in \mathcal{V}_0$ such that
\begin{align}
  a_*(\varphi,\psi) = \langle F,\psi \rangle \quad \forall \psi \in \mathcal{X}_0,
\end{align}
where
\begin{align*}
a_*(\varphi,\psi) & \coloneqq \sum_{S \in \{L,R\}} \int_{\Omega^{(S)}} \Delta \varphi \; \Delta \psi \; \mathrm{d}\f{x} + \int_{\Gamma} \average{\Delta \varphi}_\Gamma \; \jump{\partial_{\f n} \psi}_\Gamma \; \mathrm{d}s
\end{align*}
and
\begin{equation*}
\langle F,\psi \rangle  = \sum_{S \in \{L,R\}}  \int_{\Omega^{(S)}} f \, \psi \; \mathrm{d}\f{x} +  \sum_{S \in \{L,R\}} \int_{\partial\Omega^{(S)}\cap \partial\Omega} g_1 \; \partial_{{\f n}^{(S)}} \psi \; \mathrm{d}s.
\end{equation*}
\end{problem}

\begin{rmk}
The bilinear form in Problem~\ref{problem:discontinuous-weakformulation} is non-symmetric. To obtain a symmetric bilinear form, one can   symmetrize and penalize, obtaining the following problem: Find $\varphi \in \mathcal{X}_0$ such that
\begin{align}
  a_{**}(\varphi,\psi) = \langle F,\psi \rangle \quad \forall \; \psi \in \mathcal{X}_0, \label{problem:full-discontinuous-weakformulation}
\end{align}
with
\begin{align*}
a_{**}(\varphi,\psi) & \coloneqq a_*(\varphi,\psi) + \int_{\Gamma} \jump{\partial_{\f n} \varphi}_\Gamma \;  \average{\Delta \psi}_\Gamma \; \mathrm{d}s + \rho \int_{\Gamma} \jump{\partial_{\f n} \varphi}_\Gamma \; \jump{\partial_{\f n} \psi}_\Gamma \; \mathrm{d}s,
\end{align*}
where $\rho \geq 0$ is a prescribed penalty parameter. The solutions of Problem~\ref{problem:discontinuous-weakformulation} and \eqref{problem:full-discontinuous-weakformulation} are also equivalent if $\jump{\partial_{\f n} \varphi}_\Gamma = 0$.
\end{rmk}

\subsection{The approximate $C^1$ isogeometric discretization}

We consider the discrete space $\widetilde{\mathcal{V}}^1_{h,0}$. Under the assumptions summarized in Figure~\ref{fig::summaryrequirements}, we have $\widetilde{\mathcal{V}}^1_{h,0} \subset \mathcal{X}_0$. We then solve the following discrete problem.
\begin{problem}\label{problem:discrete-problem}
Find $\varphi_h \in \widetilde{\mathcal{V}}^1_{h,0}$ such that
\begin{align}
  a_h(\varphi_h,\psi_h) = \langle F,\psi_h \rangle \quad \forall \; \psi_h \in \widetilde{\mathcal{V}}^1_{h,0}, \label{eq::discrete-weakformulation}
\end{align}
where 
\begin{equation*}
a_h(\varphi_h,\psi_h) = \sum_{S\in\{L,R\}}\int_{\Omega^{(S)}} \Delta \varphi_h \Delta \psi_h \; \mathrm{d}\f{x}
\end{equation*}
and 
\begin{equation*}
\langle F,\psi_h \rangle  = \sum_{S\in\{L,R\}}\int_{\Omega^{(S)}} f \psi_h \; \mathrm{d}\f{x} + \sum_{S \in \{L,R\}} \int_{\partial \Omega \cap \partial \Omega^{(S)}} g_2 \partial_n \psi_h \; \mathrm{d}s.
\end{equation*}
Furthermore, we have, by definition,
\begin{equation*}
\langle F,\psi_h \rangle  = \int_{\Omega} f \psi_h \; \mathrm{d}\f{x} + \int_{\partial \Omega} g_2 \partial_n \psi_h \; \mathrm{d}s.
\end{equation*}
\end{problem}
  Note that this discrete problem is not an exact discretization of Problem~\ref{problem:discontinuous-weakformulation}, since we omit the jump term, which vanishes in the limit by construction. The speed of convergence then depends on the bounds on the jump term as in Theorem~\ref{thm::boundsjumptV1}, which depend on the approximation of the gluing data. We expect that the two-patch model problem with approximate $C^1$-smoothness at the interface satisfies the following a-priori error estimate.
\begin{con} \label{con::aprioirierror}
Let the assumptions of Theorem~\ref{thm::boundsjumptV1} be satisfied and let $p_1^{(S)} \geq 2$, $p_2\geq 3$ and $r_1,r_2,\widetilde{r}\geq 1$. We set $q = \min (p_1^{(L)}-1, p_1^{(R)}-1 , p_2-1, \widetilde{p} + 1)$ if $\widetilde{\alpha}^{(S)},\widetilde{\beta}^{(S)} \notin \mathbb{P}^{\widetilde{p}}$ and $q = \min (p_1^{(L)}-1, p_1^{(R)}-1 , p_2-1)$, otherwise. Let $\varphi$ be the solution of Problem~\ref{problem:model-problem}, with $\varphi \in H^{5/2+\varepsilon}(\Omega)$ and $\varphi|_{\Omega^{(S)}}\in H^{q+2} (\Omega^{(S)})$, and let $\varphi_h$ be the solution of Problem~\ref{problem:discrete-problem}, then we have
\begin{align*}
  \norm{\varphi - \varphi_h}_{\mathcal{X}} \leq C h^{q} \sum_{S \in \{L,R\} } \norm{\varphi}_{H^{q+2}(\Omega^{(S)})}.
\end{align*}
\end{con}
All requirements for constructing the interface space to obtain an optimal convergence rate are summarized in Figure~\ref{fig::summaryrequirements}.

\begin{figure}[h!]
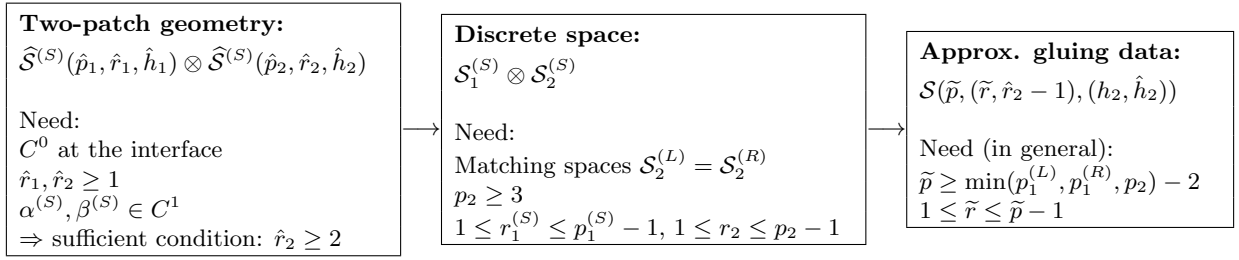

\centering
\small
\def\repeating#1{\vcenter{\offinterlineskip
  \halign{\vrule\strut\enspace##\hfil\quad\vrule\cr
    \noalign{\hrule}#1\noalign{\hrule}}}}
\def\title#1{\vphantom{$\Big($}\bf#1}
$ 
\repeating{\title{Two-patch geometry:}\cr $\widehat{\mathcal{S}}^{(S)} (\hat{p}_1,\hat{r}_1,\hat{h}_1) \otimes \widehat{\mathcal{S}}^{(S)} (\hat{p}_2,\hat{r}_2,\hat{h}_2)$ \cr \cr Need: \cr $C^0$ at the interface \cr $\hat{r}_1, \hat{r}_2 \geq 1$ \cr $\alpha^{(S)}, \beta^{(S)} \in C^1$\cr $\Rightarrow$ sufficient condition: $\hat{r}_2 \geq 2$\cr}
\mathord{\longrightarrow}
  \repeating{\title{Discrete space:}\cr $\mathcal{S}_1^{(S)} \otimes \mathcal{S}_2^{(S)}$ \cr \cr Need: \cr Matching spaces $\mathcal{S}_2^{(L)} = \mathcal{S}_2^{(R)}$\cr $p_2 \geq 3$ \cr $1 \leq r_1^{(S)} \leq p_1^{(S)} - 1$, $1 \leq r_2 \leq p_2 - 1$\cr}
\mathord{\longrightarrow}
  \repeating{\title{Approx. gluing data:}\cr $\mathcal{S} (\widetilde{p},(\widetilde{r},\hat{r}_2-1),(h_2,\hat{h}_2))$ \cr \cr Need (in general): \cr $\widetilde{p} \geq \min (p_1^{(L)}, p_1^{(R)}, p_2) - 2$ \cr $1 \leq \widetilde{r} \leq \widetilde{p} - 1$ \cr} $
\caption{Summary of the requirements for the spaces. It starts with the given $C^0$-matching, two-patch geometry where the patches meet $C^0$ at the interface. The gluing data needs to be at least $C^1$. A sufficient condition for that is $\hat{r}_2 \geq 2$. Then the discrete space is constructed as in Definition~\ref{defi::subsetwithhat}. Furthermore, $p_2 \geq 3$ is needed for constructing the interface space. For simplicity, we only assume matching spaces at the interface. Next, the space for the approximated gluing data is defined. Here, one needs $\widetilde{p}$ to be sufficiently large and $\widetilde{r}\geq 1$ to obtain an optimal convergence rate.} \label{fig::summaryrequirements}
\end{figure}

\section{Numerical experiments} \label{sec::numericalexperiments}

In the following we perform numerical experiments on four two-patch geometries. In those geometries the patches meet $C^0$ at the interface in accordance with Assumption~\ref{ass::c0conformityatinterface}. On each geometry we solve Problem~\ref{problem:discrete-problem} with the exact solution $u(x,y) = (\cos(4 \pi x) - 1) (\cos(4 \pi y) - 1)$, using the approximate $C^1$ space $\widetilde{\mathcal{V}}_h^1$ described in Section~\ref{sec::constructionapproximateC1basis}. 

Let for simplicity $p = p_1^{(L)} = p_1^{(R)} = p_2$ and $r = r_1^{(L)} = r_1^{(R)} = r_2$. Satisfying Assumption~\ref{ass::minimumregularity} and~\ref{ass::c0conforming}, we have $\widetilde{\mathcal{V}}_h^1 \subset C^0(\Omega)$ and $\widetilde{\mathcal{V}}_h^1 |_{\Omega^{(S)}} \subset C^1(\Omega^{(S)})$, for $S\in\{L,R\}$. The geometries are described in Subsection~\ref{sec::geometriesfornumericaltest}. On each geometry we compute the jump of the normal derivative, satisfying the estimate in Theorem~\ref{thm::boundsjumptV1}. The results are reported in Subsection~\ref{sec::jumperror}. In Subsection~\ref{sec::errorestimate} we present the convergence rates of the error measured in $L^2$-, $H^1$- and $H^2$-norms for various polynomial degrees $\widetilde{p}$ and $p \geq 3$, see Assumption~\ref{ass::minimumdegree}. The observed rates are consistent with Conjecture~\ref{con::aprioirierror}. In Subsection~\ref{sec::regularity}, we conclude the numerical tests with comparisons of varying spline regularity $r$. All tests are implemented within the open-source C++ library G+Smo, cf.~\cite{gismoweb}.

\subsection{The geometries for the numerical tests} \label{sec::geometriesfornumericaltest}

The numerical tests are performed on four two-patch geometries shown in Subfigures~\ref{fig::geo0}-\ref{fig::geo41}. Example~I is an AS-$G^1$ geometry as described in Definition~\ref{def::asg1}, whereas Examples~II,~III and~IV are non AS-$G^1$ geometries. The corresponding exact solutions are shown in Subfigures~\ref{fig::sol0}-\ref{fig::sol41}. The first two Examples~I and~II describe the same domain, but have different parametrizations. While Example~I, which is composed of bilinear patches, has a straight interface, Example~II is composed of bicubic patches and has a curved interface. Example~III and~IV are a geometries which both have a corner at one end of the interface. There the boundary space is modified as described in Subsection~\ref{sec::smoothboundary}. Furthermore, Example~III depicts a quarter of a plate with a circular hole, where the circular arc is approximated by a cubic B-spline curve. The Examples~I,~II and~III are constructed with no internal knots, i.e.\ the geometries are constructed with Bézier patches. Therefore, one can set $\hat{r}_1 = \hat{r}_2 = \infty$ and the gluing data is $C^\infty$-smooth. In contrast, Example~IV is a B-spline geometry with degree $\hat{p}_2=3$ and regularity $\hat{r}_2=2$ ergo the gluing data is only $C^1$-smooth. For each geometry the gluing data is computed as defined in~\eqref{def::alpha_beta} and in~\eqref{def::beta_S}. One can see that the gluing data for Example~I is linear, see Subfigure~\ref{fig::gd0}, while for Examples~II-IV the gluing data is not linear, see Subfigures~\ref{fig::gd1}-\ref{fig::gd41}. 

\begin{figure}[h!]
  \begin{subfigure}[b]{0.24\textwidth}
  \centering
  \includegraphics[width=.8\textwidth]{plots/ex0_mesh.png}
  \subcaption{Example I} \label{fig::geo0}
  \end{subfigure}
  \begin{subfigure}[b]{0.24\textwidth}
  \centering
  \includegraphics[width=.8\textwidth]{plots/ex1_mesh.png}
  \subcaption{Example II} \label{fig::geo1}
  \end{subfigure}
  \begin{subfigure}[b]{0.24\textwidth}
  \centering
  \includegraphics[width=.8\textwidth]{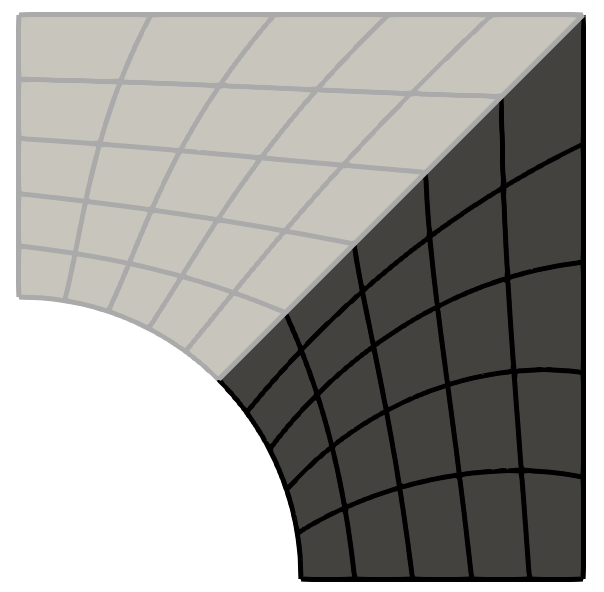}
  \subcaption{Example III} \label{fig::geo8}
  \end{subfigure}
  \begin{subfigure}[b]{0.24\textwidth}
  \centering
  \includegraphics[width=\textwidth]{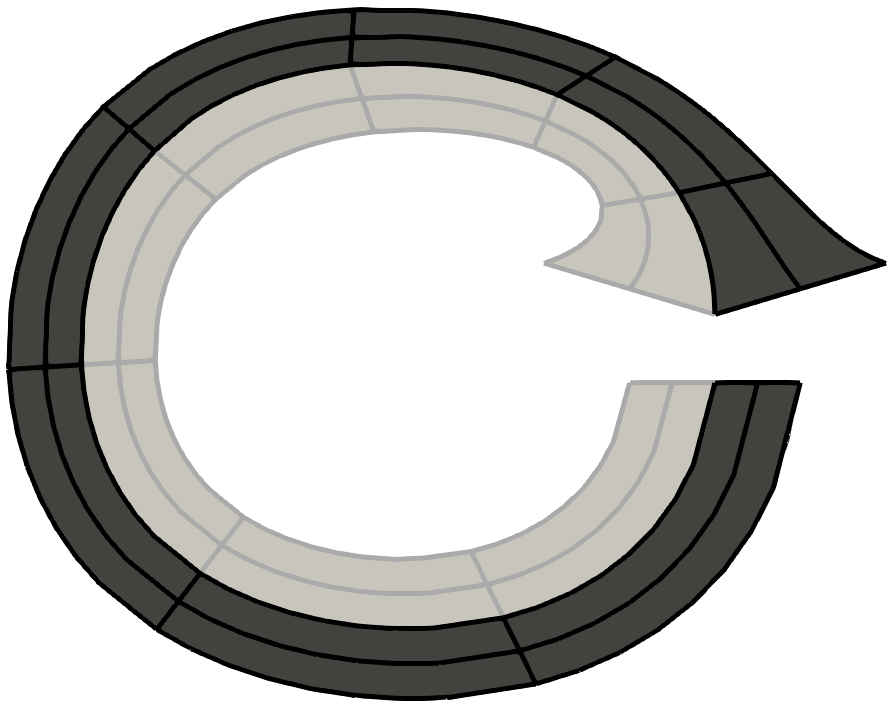}
  \subcaption{Example IV} \label{fig::geo41}
  \end{subfigure}

      \begin{subfigure}[b]{0.24\textwidth}
  \centering
  \includegraphics[width=.8\textwidth]{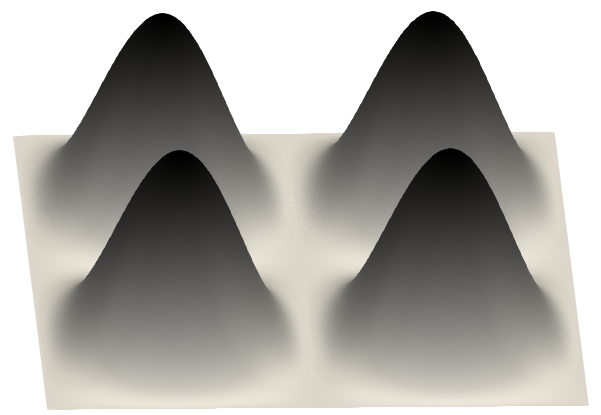}
  \subcaption{Exact solution of Ex.\ I} \label{fig::sol0}
  \end{subfigure}
  \begin{subfigure}[b]{0.24\textwidth}
  \centering
  \includegraphics[width=.8\textwidth]{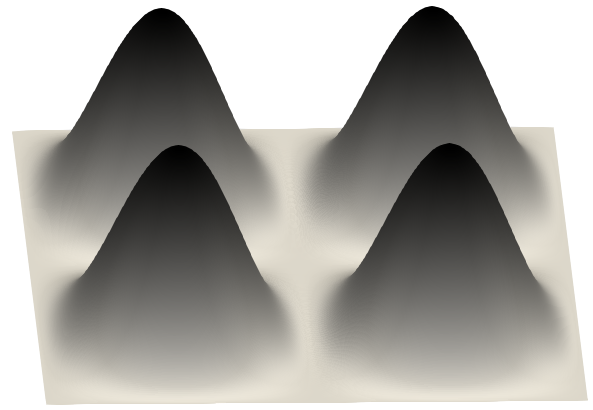}
  \subcaption{Exact solution of Ex.\ II} \label{fig::sol1}
  \end{subfigure}
    \begin{subfigure}[b]{0.24\textwidth}
  \centering
  \includegraphics[width=.92\textwidth]{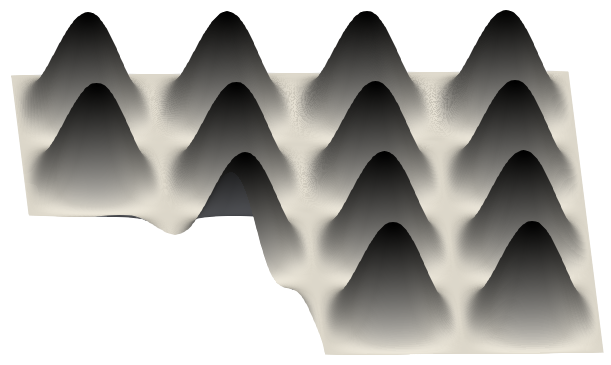}
  \subcaption{Exact solution of Ex.\ III} \label{fig::sol8}
  \end{subfigure}
  \begin{subfigure}[b]{0.24\textwidth}
  \centering
  \includegraphics[width=.8\textwidth]{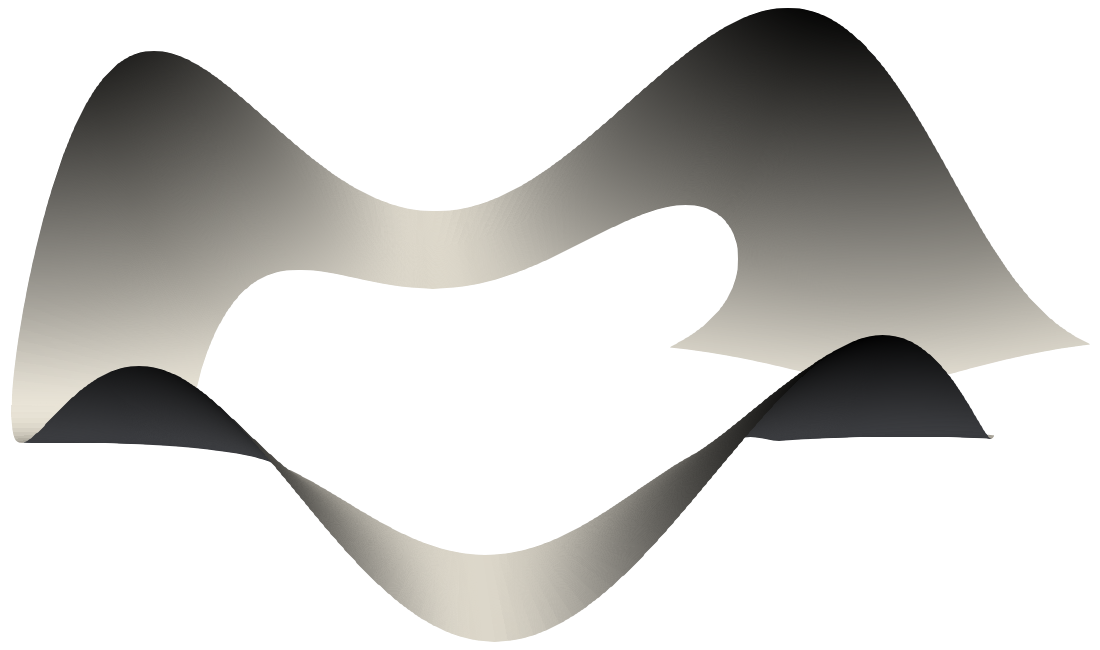}
  \subcaption{Exact solution of Ex.\ IV} \label{fig::sol41}
  \end{subfigure}

\begin{subfigure}[b]{0.24\textwidth}
  \centering
  \includegraphics[width=\textwidth]{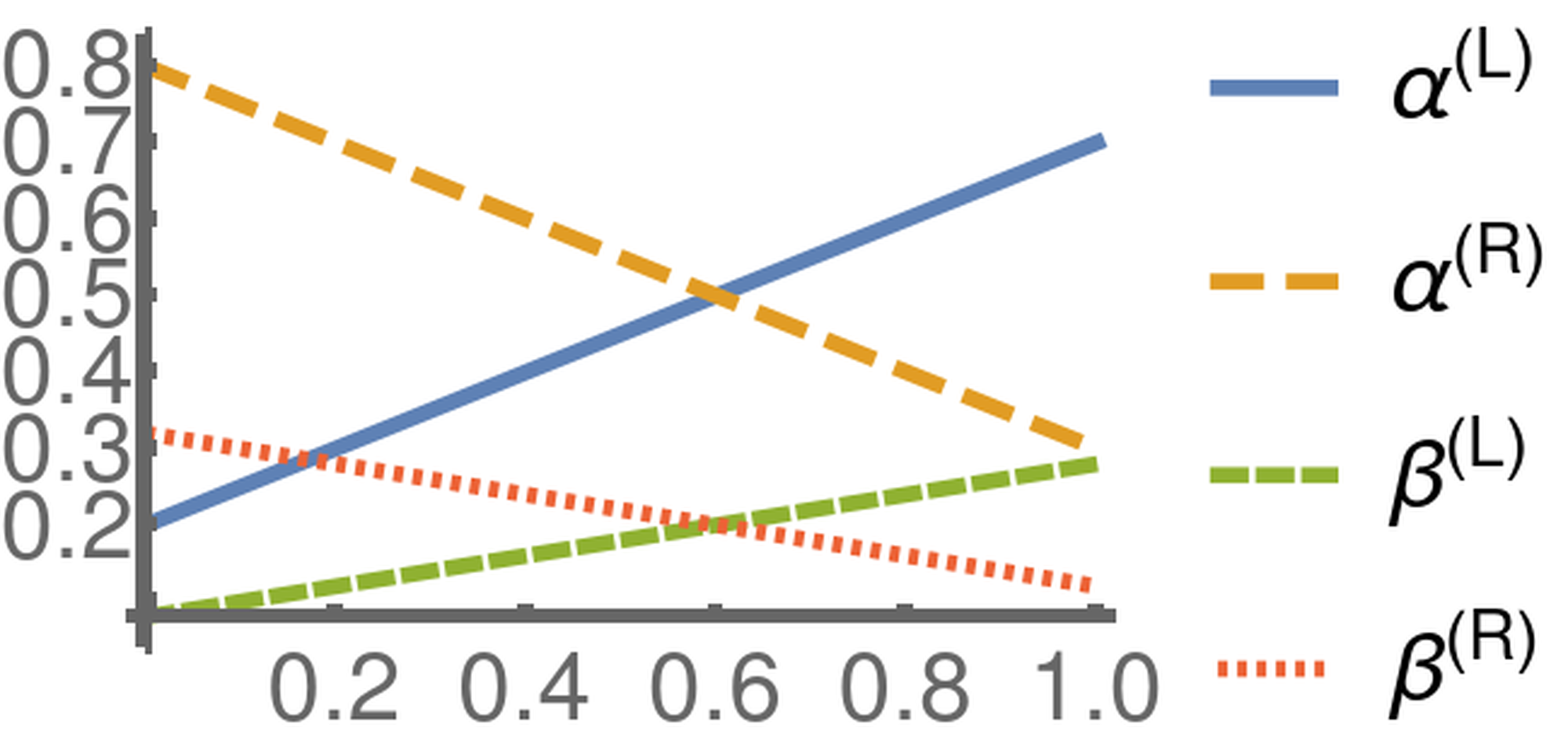}
  \subcaption{Gluing data of Ex.\ I}  \label{fig::gd0}
  \end{subfigure}
  \begin{subfigure}[b]{0.24\textwidth}
  \centering
  \includegraphics[width=\textwidth]{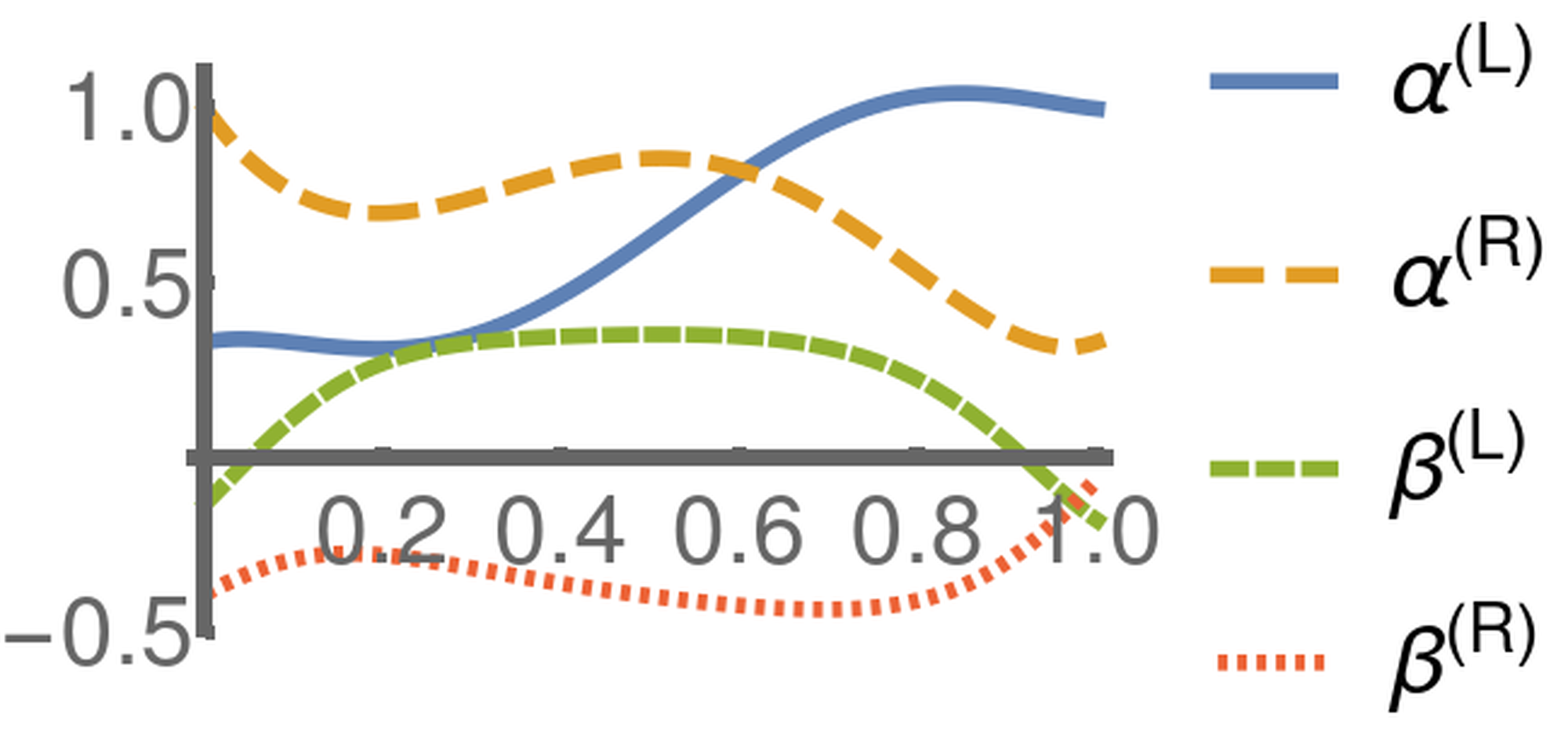}
  \subcaption{Gluing data of Ex.\ II} \label{fig::gd1}
  \end{subfigure}
  \begin{subfigure}[b]{0.24\textwidth}
  \centering
  \includegraphics[width=\textwidth]{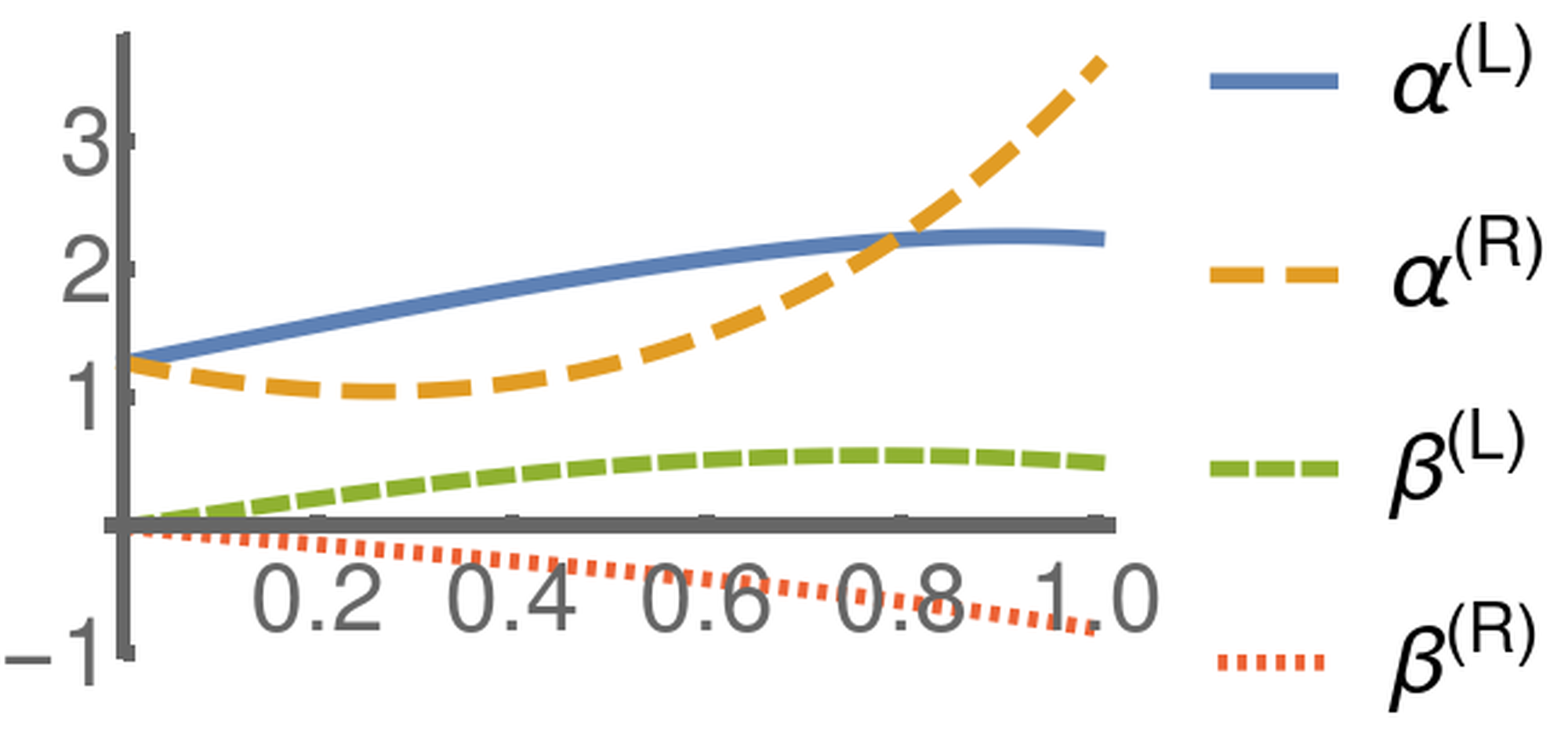}
  \subcaption{Gluing data of Ex.\ III} \label{fig::gd8}
  \end{subfigure}
  \begin{subfigure}[b]{0.24\textwidth}
  \centering
  \includegraphics[width=\textwidth]{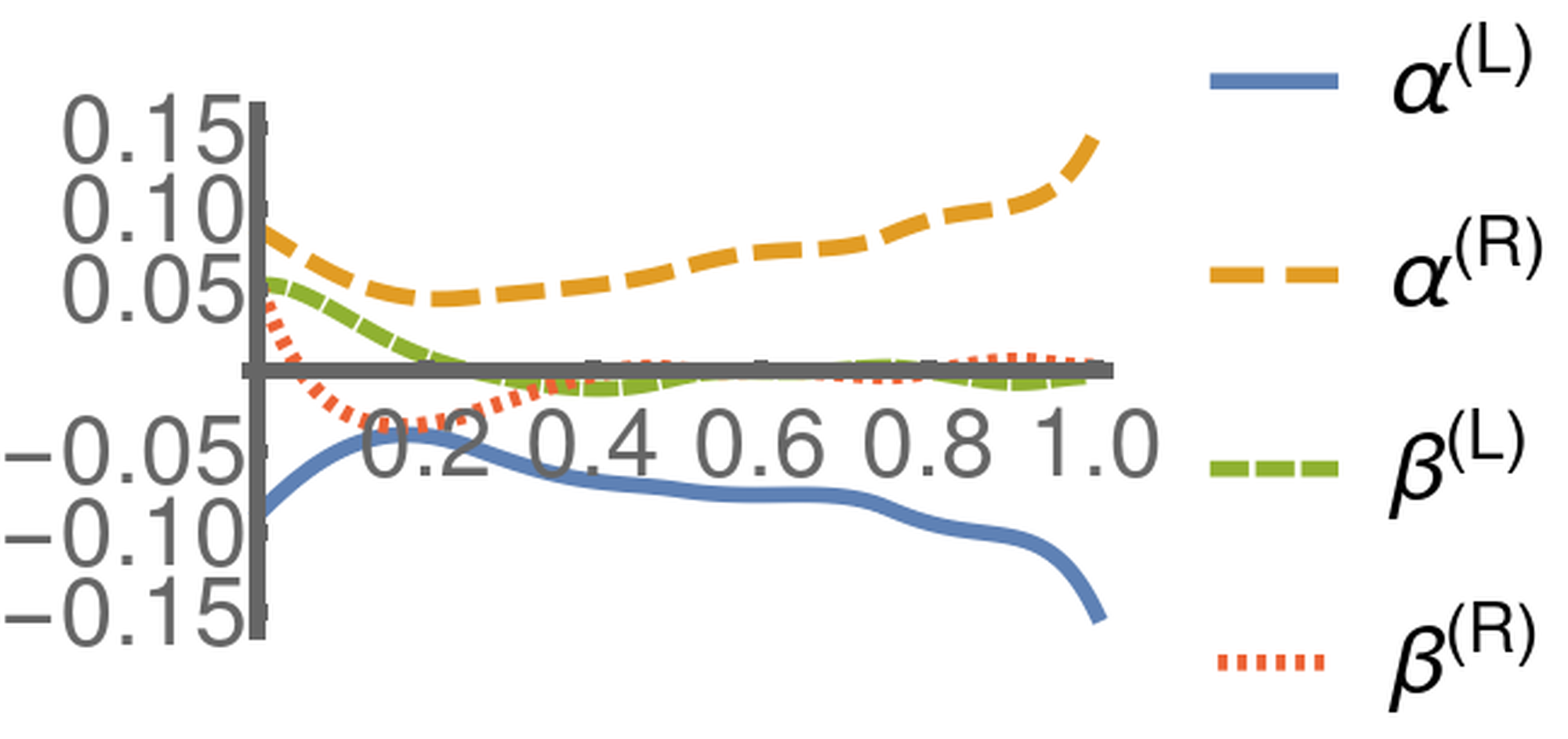}
  \subcaption{Gluing data of Ex.\ IV} \label{fig::gd41}
  \end{subfigure}
  
  \caption{The four choosen geometries for the numerical results with the corresponding exact solutions and gluing data.} 
\end{figure}

\subsection{Convergence of the jump of the normal derivative} \label{sec::jumperror}

In this section, we provide convergence results for the jump of the normal derivative of the discrete solution at the interface, that is, we compute
\[
\| \jump{\partial_{\f n} \varphi_h } \|_{L^2(\Gamma)} = 
\left(\int_\Gamma \left( \partial_{\f n} \varphi_h^{(R)} - \partial_{\f n} \varphi_h^{(L)} \right)^2 \mathrm{d}s \right)^{1/2},
\]
where $\varphi_h$ is the solution of Problem~\ref{problem:discrete-problem}. In all four examples we fixed the polynomial degree to $p=3$ and the regularity to $r=1$. The approximate gluing data is computed for splines of degree $\widetilde{p} \in \{1,2,3,4 \}$ and with maximum regularity $\widetilde{r} = \widetilde{p}-1$. The results are shown in Figure~\ref{fig::jumperror}. Since the first geometry is AS-$G^1$, the normal jump is (numerically) zero (see Subfigure~\ref{fig::plot_gluingData_AS}) and thus the discrete solution is $C^1$-smooth at the interface, up to tolerance. In the other Examples~II-IV the geometry is not AS-$G^1$. The observed convergence rate $\widetilde{p}+1$ of the normal jump is consistent with Theorem~\ref{thm::boundsjumptV1}, as can be seen in Subfigures~\ref{fig::plot_gluingData},~\ref{fig::plot_gluingData_geo8} and~\ref{fig::plot_gluingData_geo41}.

\begin{figure}[h!] 
  \begin{subfigure}[c]{0.24\textwidth}
    \includegraphics[width=\textwidth]{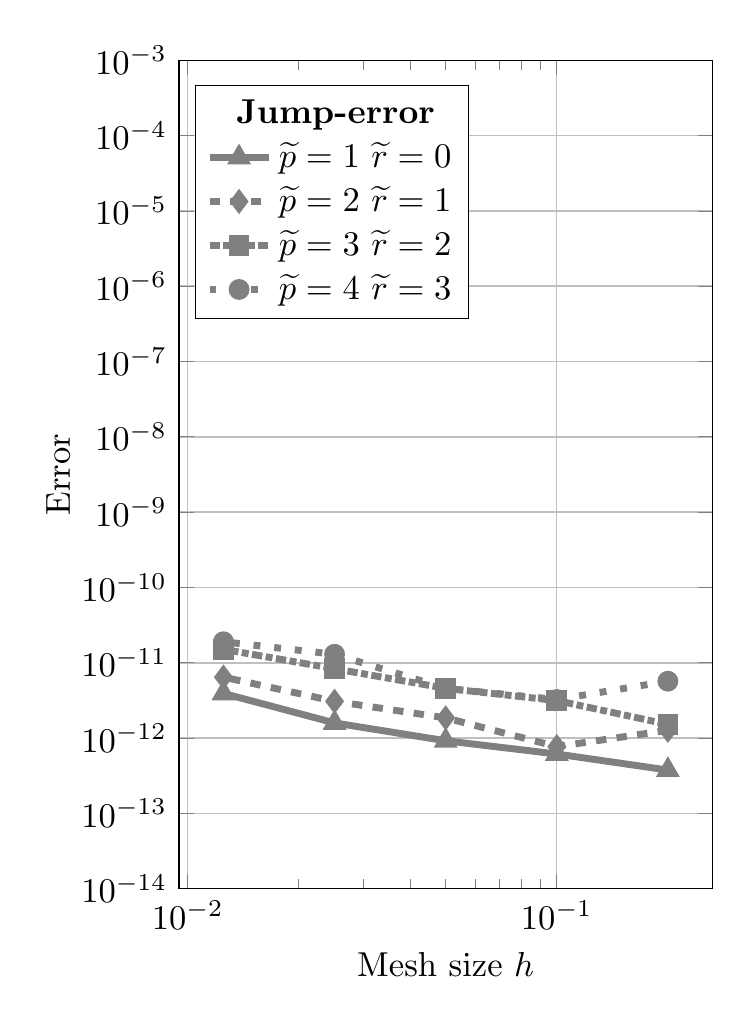} 
\subcaption{The normal jump on Ex.\ I. } \label{fig::plot_gluingData_AS}
\end{subfigure}
\begin{subfigure}[c]{0.24\textwidth}
\centering
    \includegraphics[width=\textwidth]{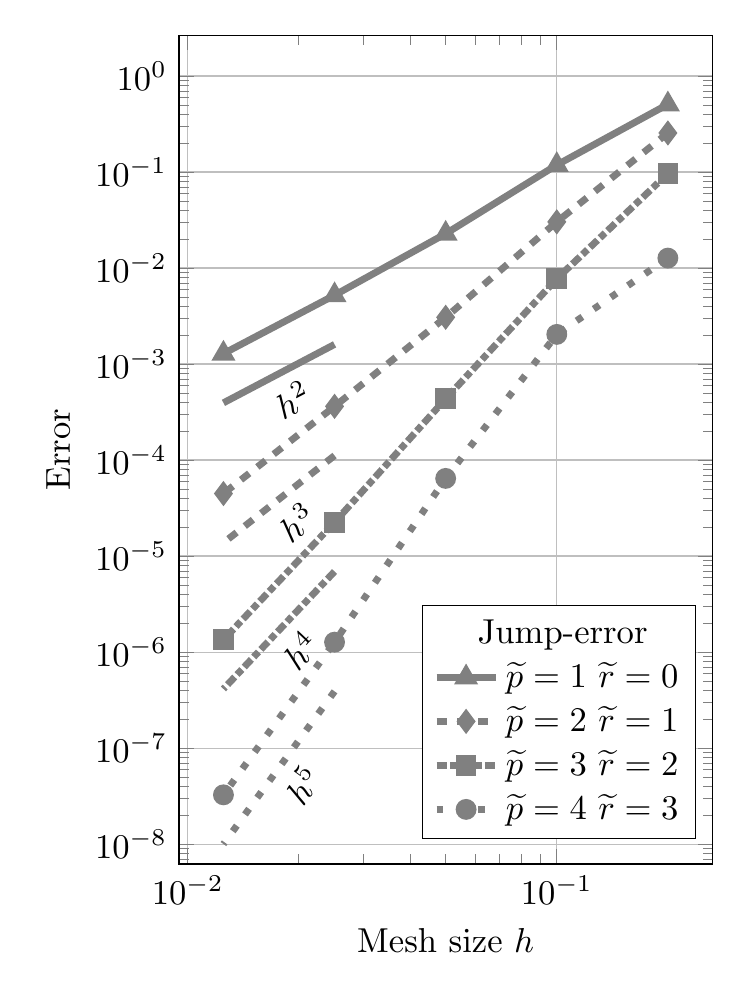}
  \subcaption{The normal jump on Ex.\ II.} \label{fig::plot_gluingData}
\end{subfigure}
\begin{subfigure}[c]{0.24\textwidth}
\centering
    \includegraphics[width=\textwidth]{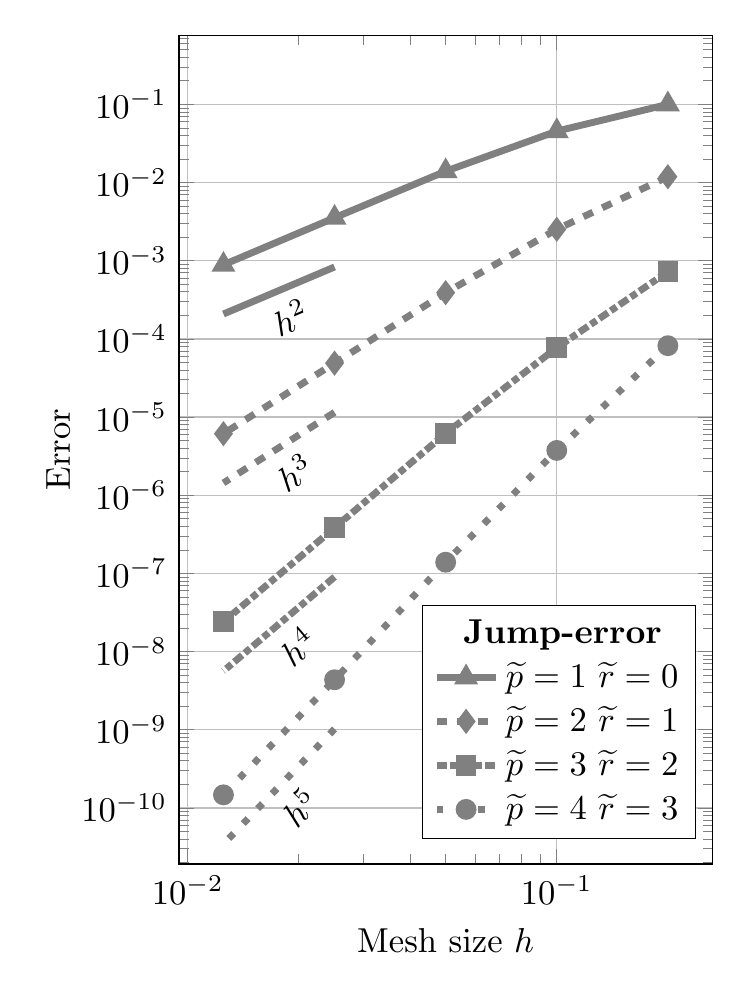}
  \subcaption{The normal jump on Ex.\ III.} \label{fig::plot_gluingData_geo8}
\end{subfigure}
\begin{subfigure}[c]{0.24\textwidth}
\centering
    \includegraphics[width=\textwidth]{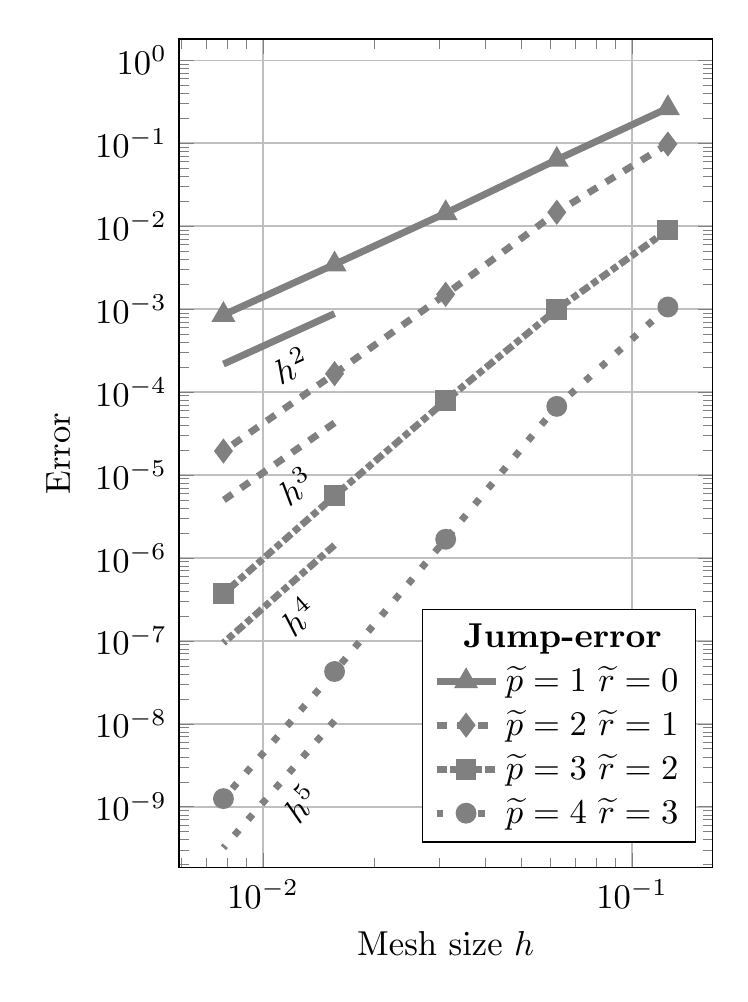}
  \subcaption{The normal jump on Ex.\ IV.} \label{fig::plot_gluingData_geo41}
\end{subfigure}
  \caption{The jump of the normal derivative at the interface for Examples I-IV.} \label{fig::jumperror}
\end{figure}

\subsection{Dependence of convergence rates on approximation of gluing data} \label{sec::errorestimate}

We compare the convergence rates of the error measured in the $L^2$-, $H^1$- and $H^2$-norms for varying polynomial degree $p$ of the spline space and varying polynomial degree $\widetilde{p}$ of the approximate gluing data. We expect that the approximate gluing data must be of degree $\widetilde{p} \geq p - 2$, such that the convergence rates are optimal, as stated in Conjecture~\ref{con::aprioirierror}.

Therefore, we plot the $L^2$-, $H^1$- and $H^2$-errors for Examples~I to~IV for polynomial degrees $p \in \{ 3,4,5\}$. The approximate gluing data is constructed with the lowest polynomial degree $\widetilde{p} = \max(p-2,2)$ and highest possible regularity $\widetilde{r} = \widetilde{p}-1$ to achieve the smallest number of degrees of freedom for the approximate $C^1$ construction. The results are summarized in Figure~\ref{fig::ex1error}.

For the AS-$G^1$ geometry in Example I, we obtain optimal convergence rates for all polynomial degrees $p$. The rate is independent of the degree of the approximate gluing data, which follows from the fact that the gluing data is linear and the normal jump vanishes for all $\widetilde{p}\geq 1$. Hence, according to Proposition~\ref{prop::asthanc1} one obtains an exactly $C^1$-smooth space. The errors are plotted in Subfigures~\ref{fig::plot_gluingData_error_AS},~\ref{fig::plot_gluingData_error_p4_AS} and~\ref{fig::plot_gluingData_error_p5_AS} for degrees $p = 3$, $p = 4$ and $p = 5$, respectively.

While all errors converge optimally for Example~I this is not the case for Examples~II,~III and~IV, which we study in the following. One can see that for those non AS-$G^1$ geometries the rates are not optimal for any degree $p$ if the gluing data is approximated with $(\widetilde{p},\widetilde{r}) = (1,0)$. The reason for this is that the space $\widetilde{\mathcal{A}}_{\Gamma}$ of interface functions is not conforming, as it is not $C^1$ in a neighborhood of the interface (see~\eqref{eq::spaceforbasisfunctions}). Hence, the space violates Assumption~\ref{ass::minimumregularity} which leads to non-optimal convergence rates.

Note that for $p=3$ the convergence rates are optimal if the gluing data is approximated with $\widetilde{p} \geq 2$, see Subfigures~\ref{fig::plot_gluingData_error},~\ref{fig::plot_gluingData_error_geo8} and~\ref{fig::plot_gluingData_error_geo41}. Increasing the degree for the approximate gluing data does not reduce the errors significantly as the curves overlap.

A similar behaviour can be observed for polynomial degree $p=4$, as can be seen in Subfigures~\ref{fig::plot_gluingData_error_p4},~\ref{fig::plot_gluingData_error_p4_geo8} and~\ref{fig::plot_gluingData_error_p4_geo41}. Optimal rates are obtained for $\widetilde{p} \geq 2$.

For polynomial degree $p=5$ approximating the gluing data with splines with $\widetilde{p} = 2$ is not enough. One requires at least $\widetilde{p} = 3$ to reach an optimal error rate as shown in Subfigures~\ref{fig::plot_gluingData_error_p5},~\ref{fig::plot_gluingData_error_p4_geo8} and~\ref{fig::plot_gluingData_error_p4_geo41}. In Subfigure~\ref{fig::plot_gluingData_error_p5_geo8}, one cannot observe a difference between $\widetilde{p} = 2$ and $\widetilde{p} = 3$. A possible reason is that the normal jump for Example IV is significant smaller than the $H^2$-error. We expect that as the mesh becomes more refined the consistency error from the jump of the normal derivative will dominate the $H^2$-error and the convergence rate will deteriorate, similar to the behaviour for Examples~II and~IV.

\begin{figure}[hp!]  
  \begin{subfigure}[c]{0.24\textwidth}
  \centering
  \includegraphics[width=\textwidth]{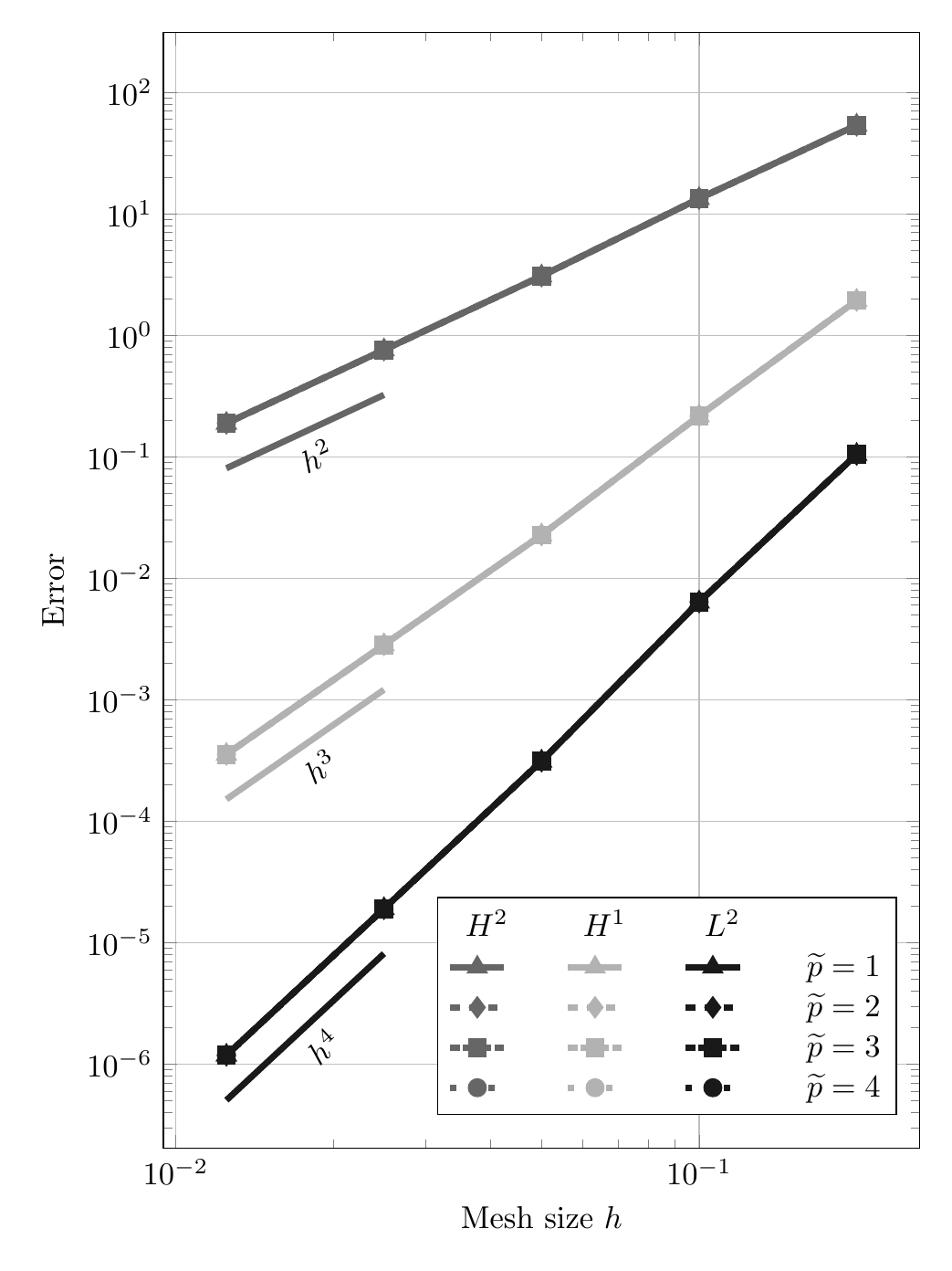}
  \subcaption{Ex.\ I with $p = 3$, $r = 1$.} \label{fig::plot_gluingData_error_AS}
  \end{subfigure}
  \begin{subfigure}[c]{0.24\textwidth}
  \centering
  \includegraphics[width=\textwidth]{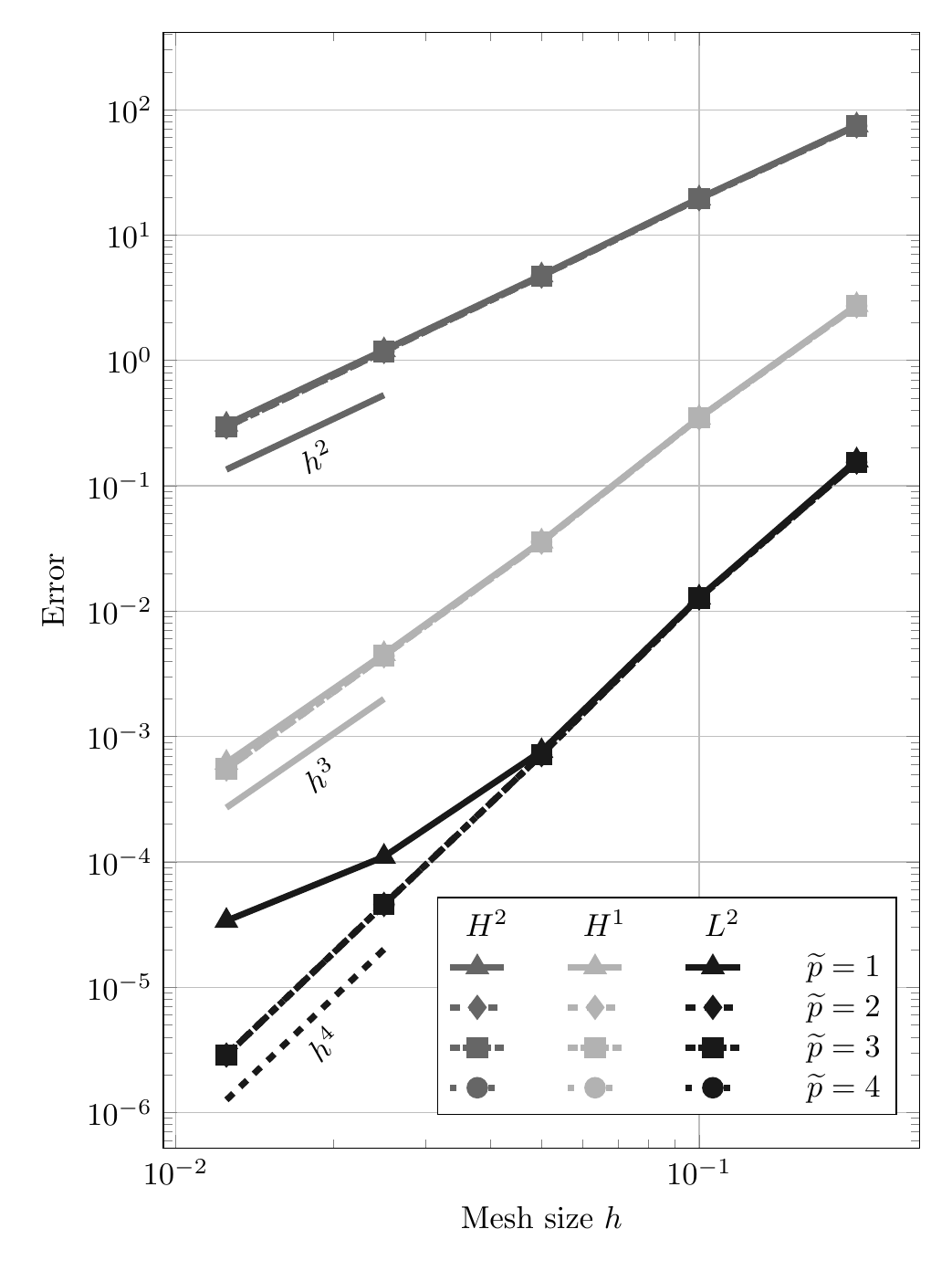}
\subcaption{Ex.\ II with $p = 3$, $r = 1$.} \label{fig::plot_gluingData_error}
  \end{subfigure}
      \begin{subfigure}[c]{0.24\textwidth}
  \centering
  \includegraphics[width=\textwidth]{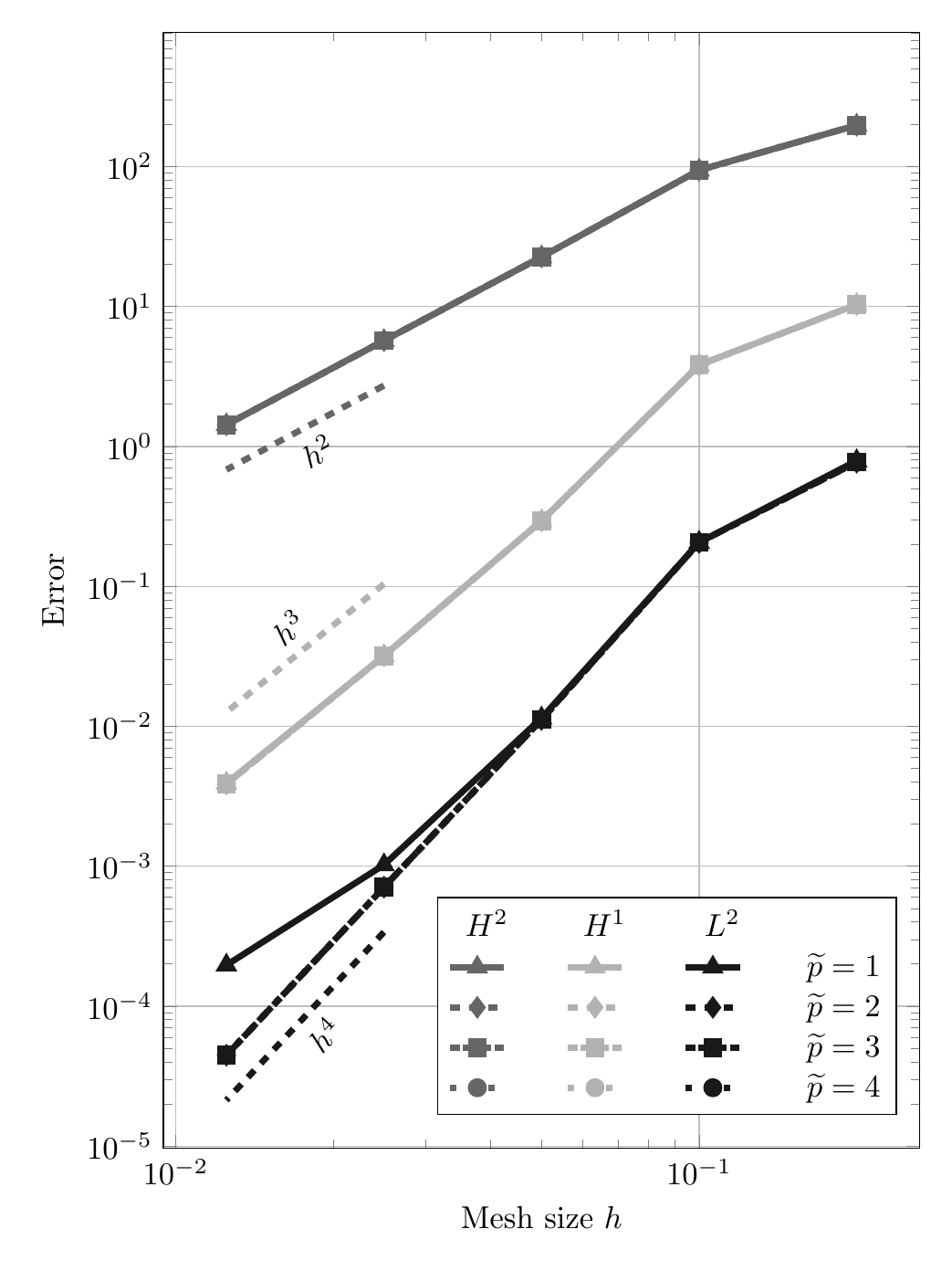}
  \subcaption{Ex.\ III with $p = 3$, $r = 1$.} \label{fig::plot_gluingData_error_geo8}
  \end{subfigure}
    \begin{subfigure}[c]{0.24\textwidth}
  \centering
  \includegraphics[width=\textwidth]{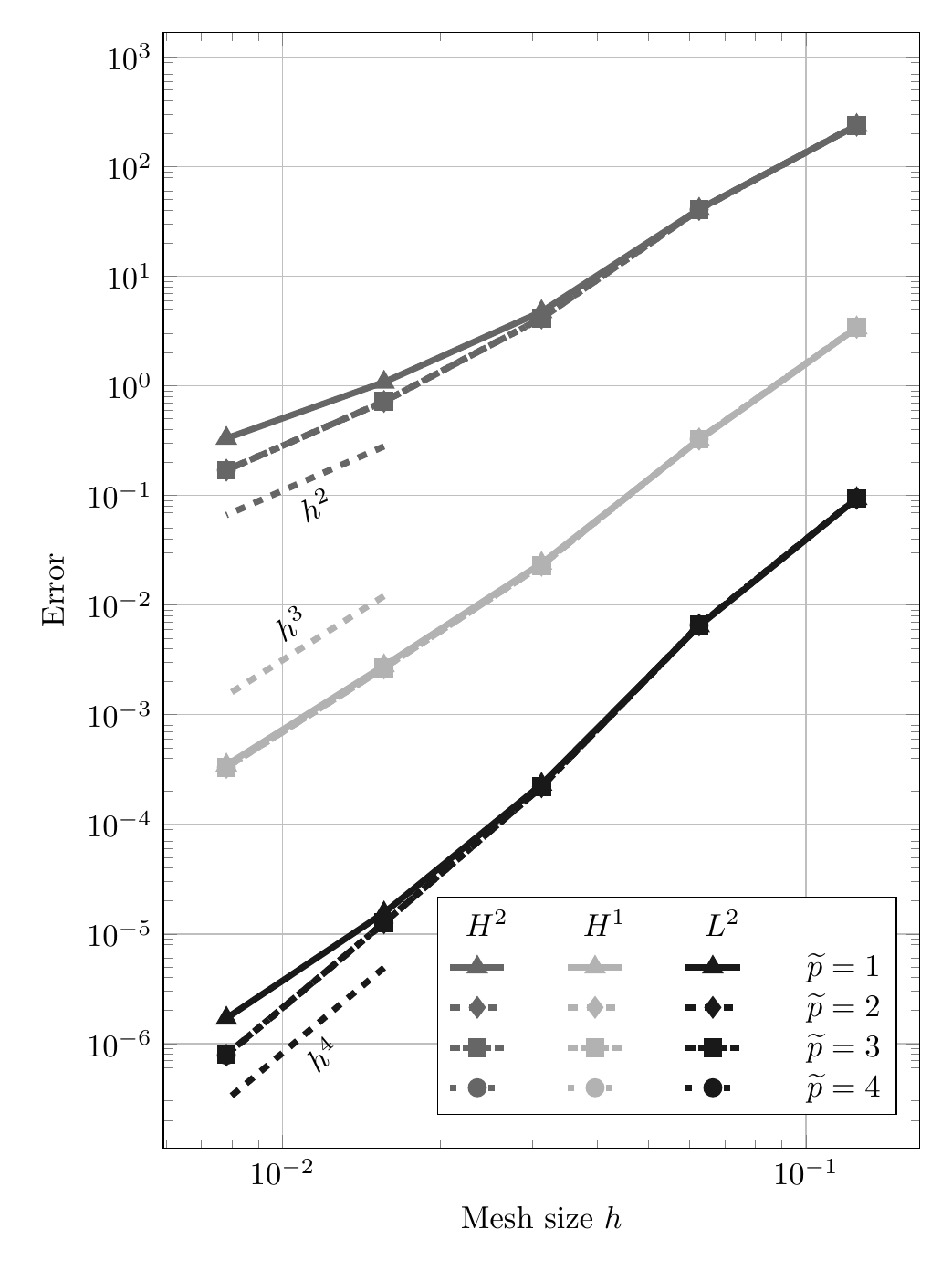}
  \subcaption{Ex.\ IV with $p = 3$, $r = 1$.} \label{fig::plot_gluingData_error_geo41}
  \end{subfigure}

  \begin{subfigure}[b]{0.24\textwidth}
  \centering
  \includegraphics[width=\textwidth]{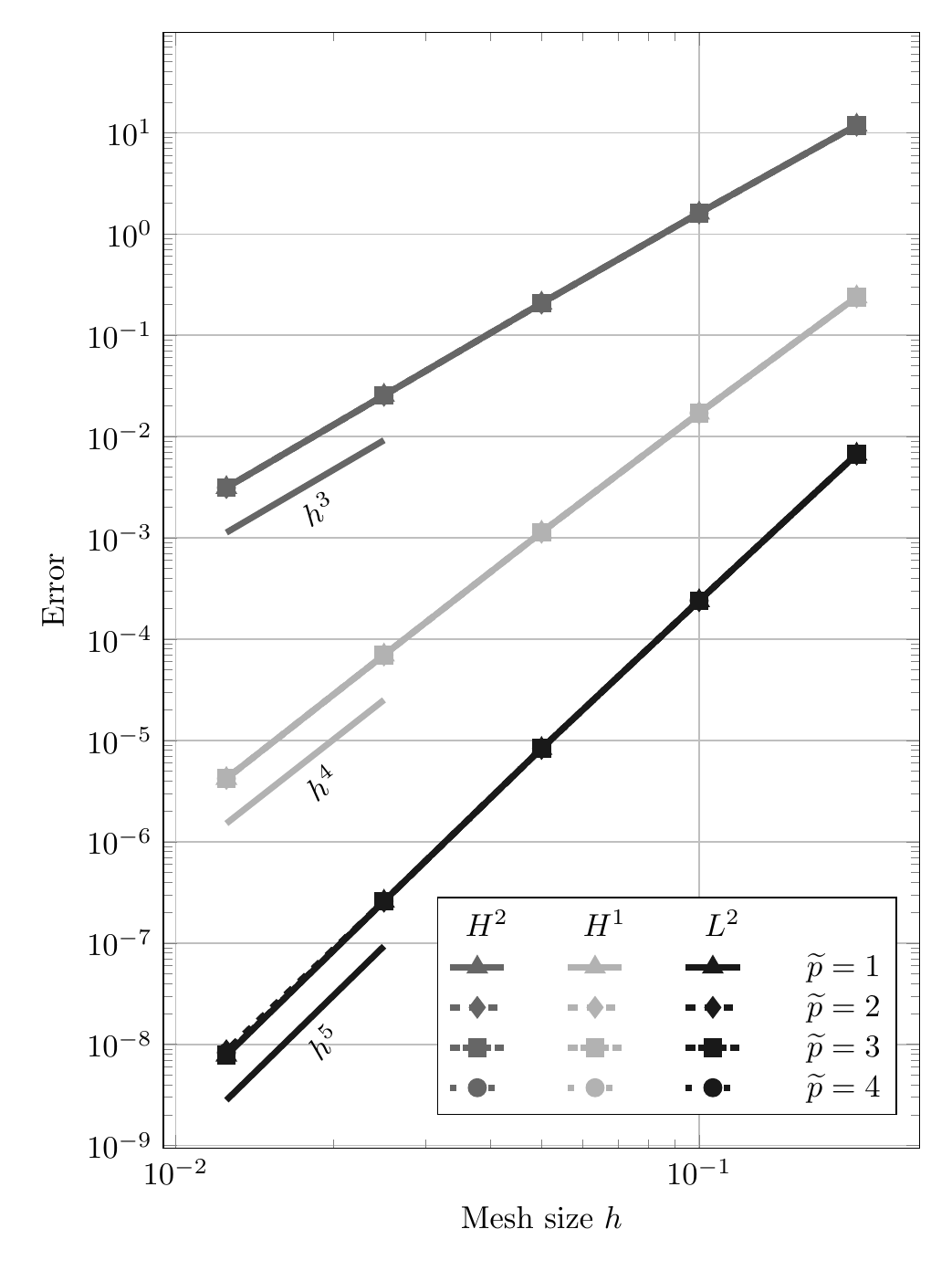}
  \subcaption{Ex.\ I with $p = 4$, $r = 1$.} \label{fig::plot_gluingData_error_p4_AS}
  \end{subfigure}
  \begin{subfigure}[b]{0.24\textwidth}
  \centering
  \includegraphics[width=\textwidth]{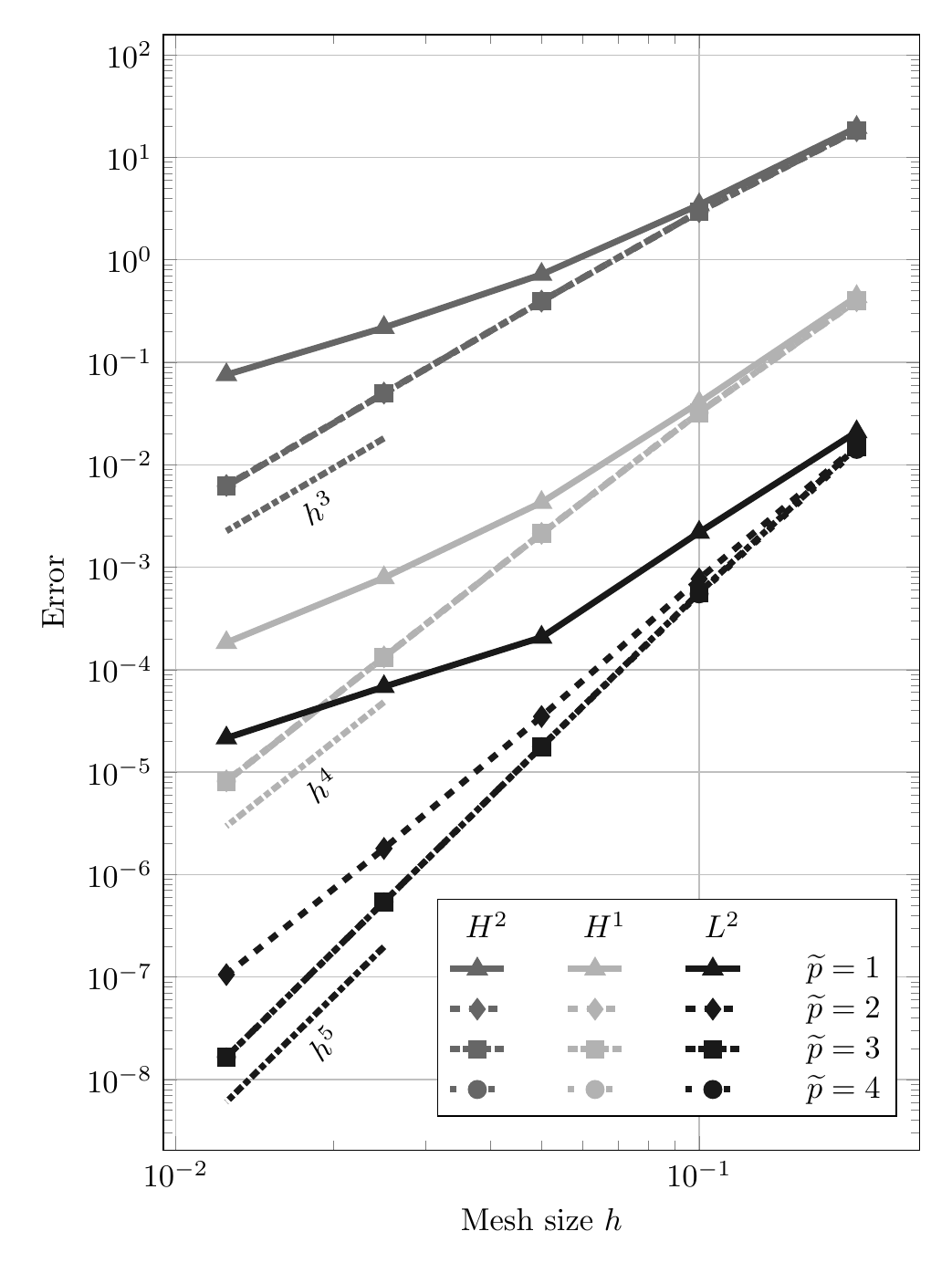}
\subcaption{Ex.\ II with $p = 4$, $r = 1$.} \label{fig::plot_gluingData_error_p4}
  \end{subfigure}
      \begin{subfigure}[b]{0.24\textwidth}
  \centering
  \includegraphics[width=\textwidth]{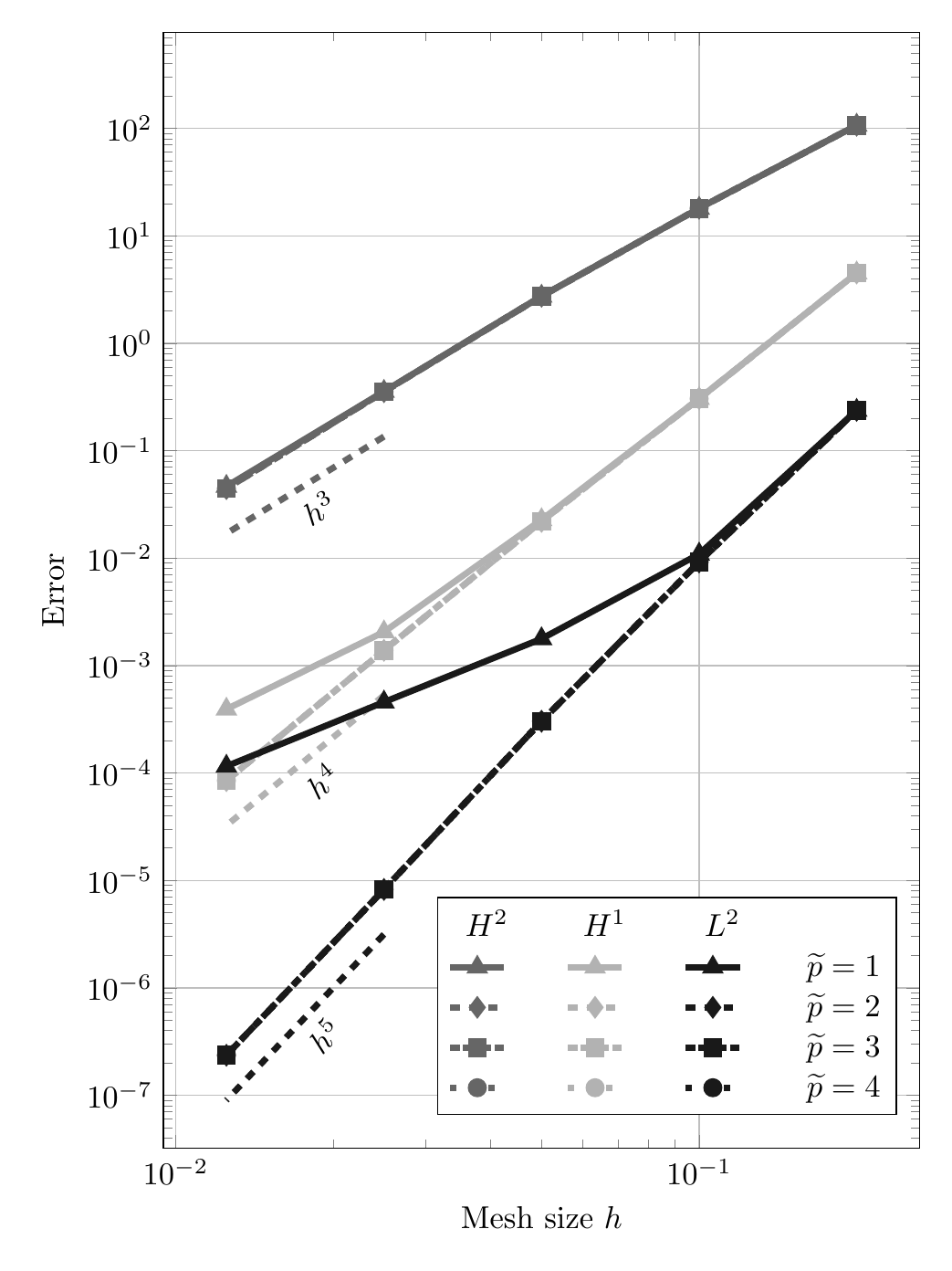}
  \subcaption{Ex.\ III with $p = 4$, $r = 1$.} \label{fig::plot_gluingData_error_p4_geo8}
  \end{subfigure}
    \begin{subfigure}[b]{0.24\textwidth}
  \centering
  \includegraphics[width=\textwidth]{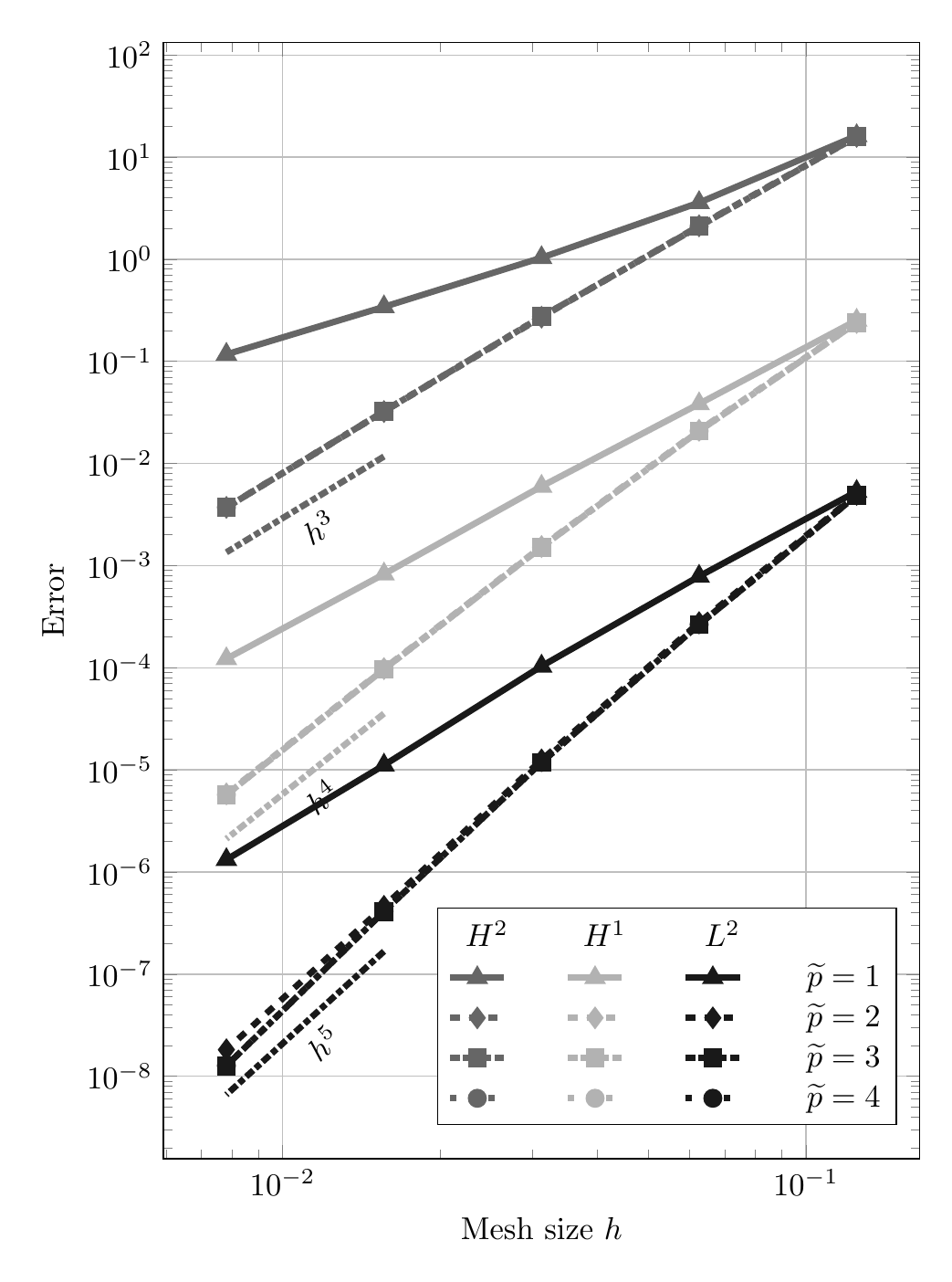}
  \subcaption{Ex.\ IV with $p = 4$, $r = 1$.} \label{fig::plot_gluingData_error_p4_geo41}
  \end{subfigure}

  \begin{subfigure}[b]{0.24\textwidth}
  \centering
  \includegraphics[width=\textwidth]{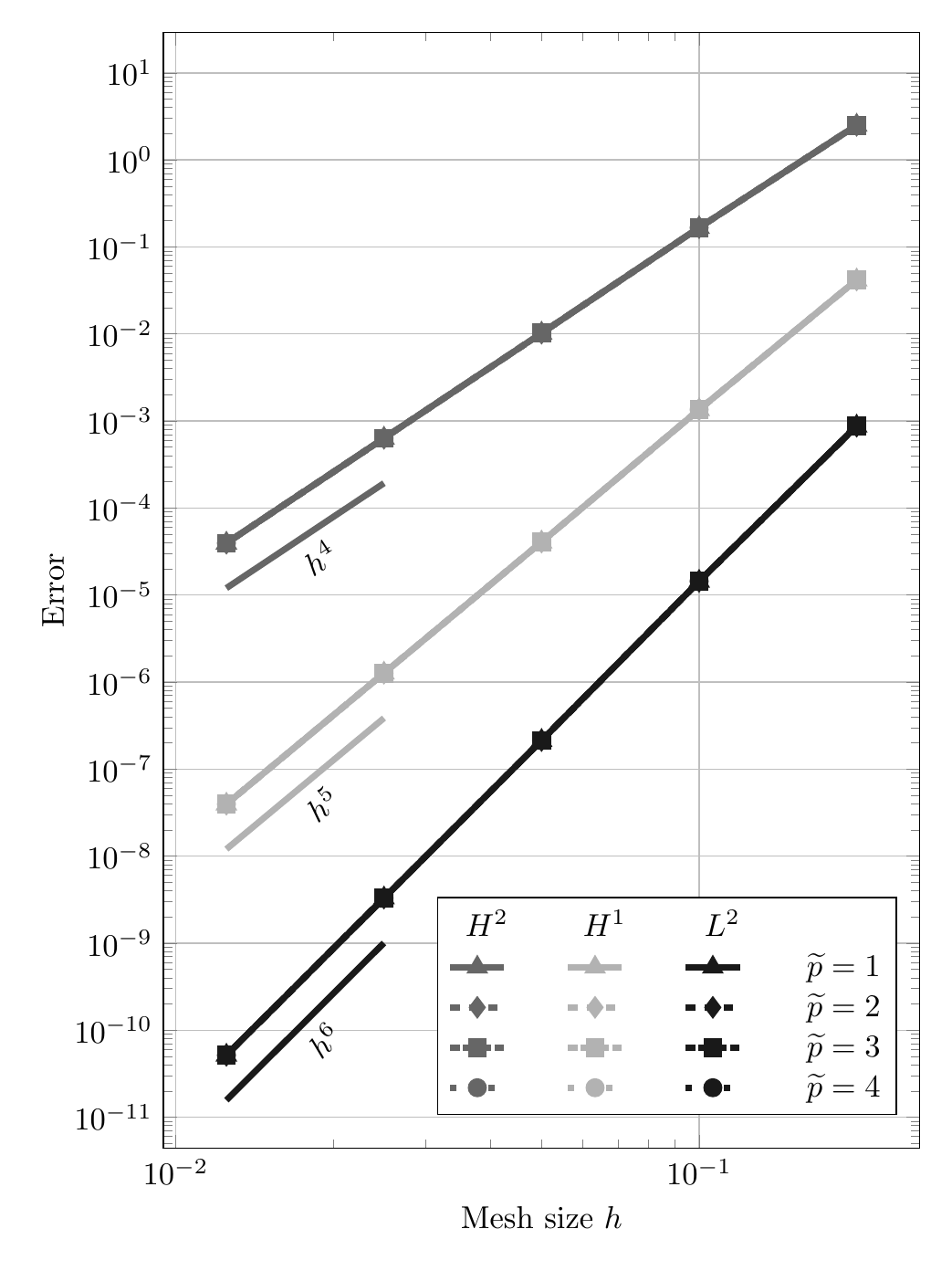}
  \subcaption{Ex.\ I with $p = 5$, $r = 1$.} \label{fig::plot_gluingData_error_p5_AS}
  \end{subfigure}
  \begin{subfigure}[b]{0.24\textwidth}
  \centering
  \includegraphics[width=\textwidth]{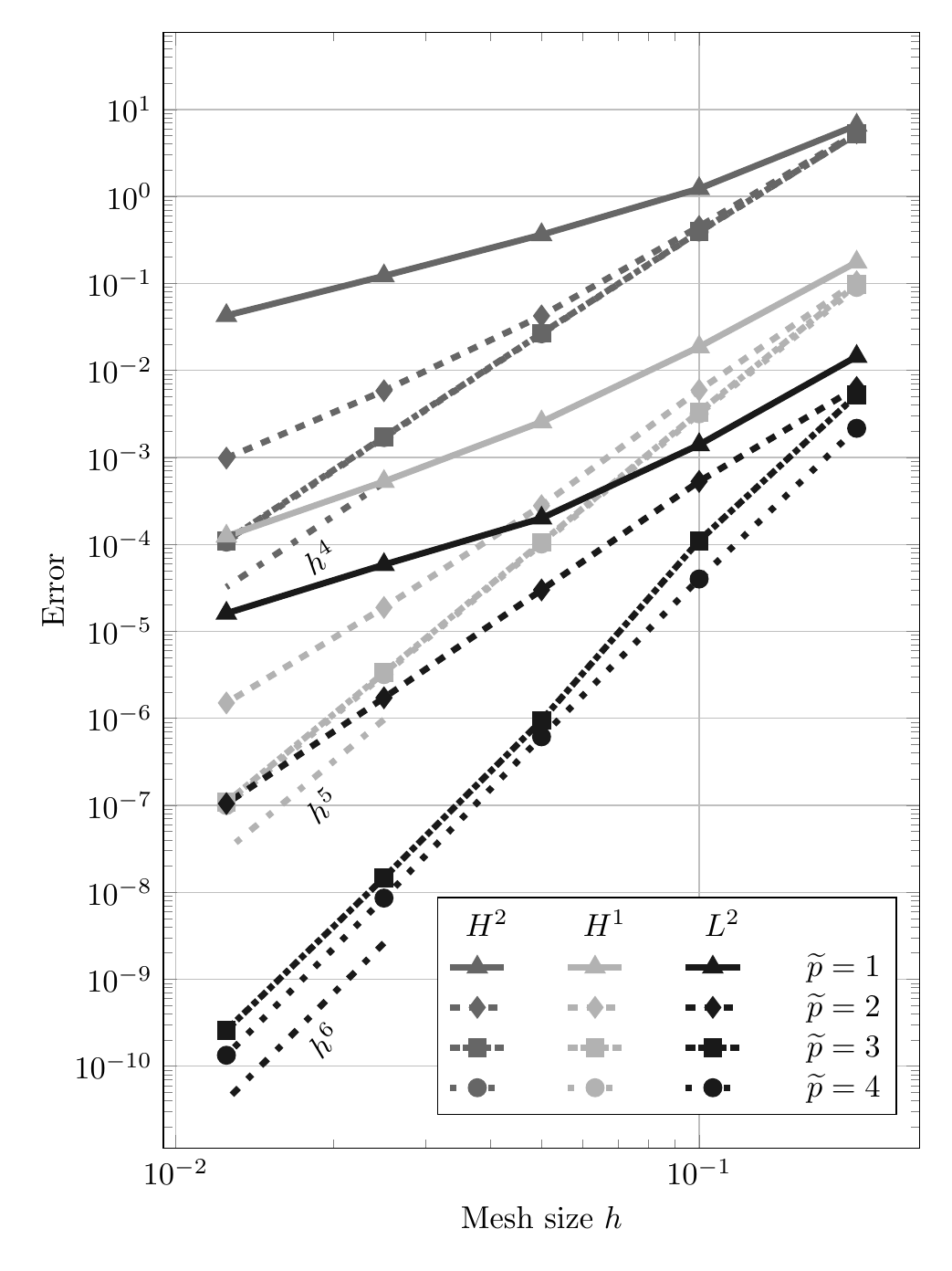}
\subcaption{Ex.\ II with $p = 5$, $r = 1$.} \label{fig::plot_gluingData_error_p5}
  \end{subfigure}
      \begin{subfigure}[b]{0.24\textwidth}
  \centering
  \includegraphics[width=\textwidth]{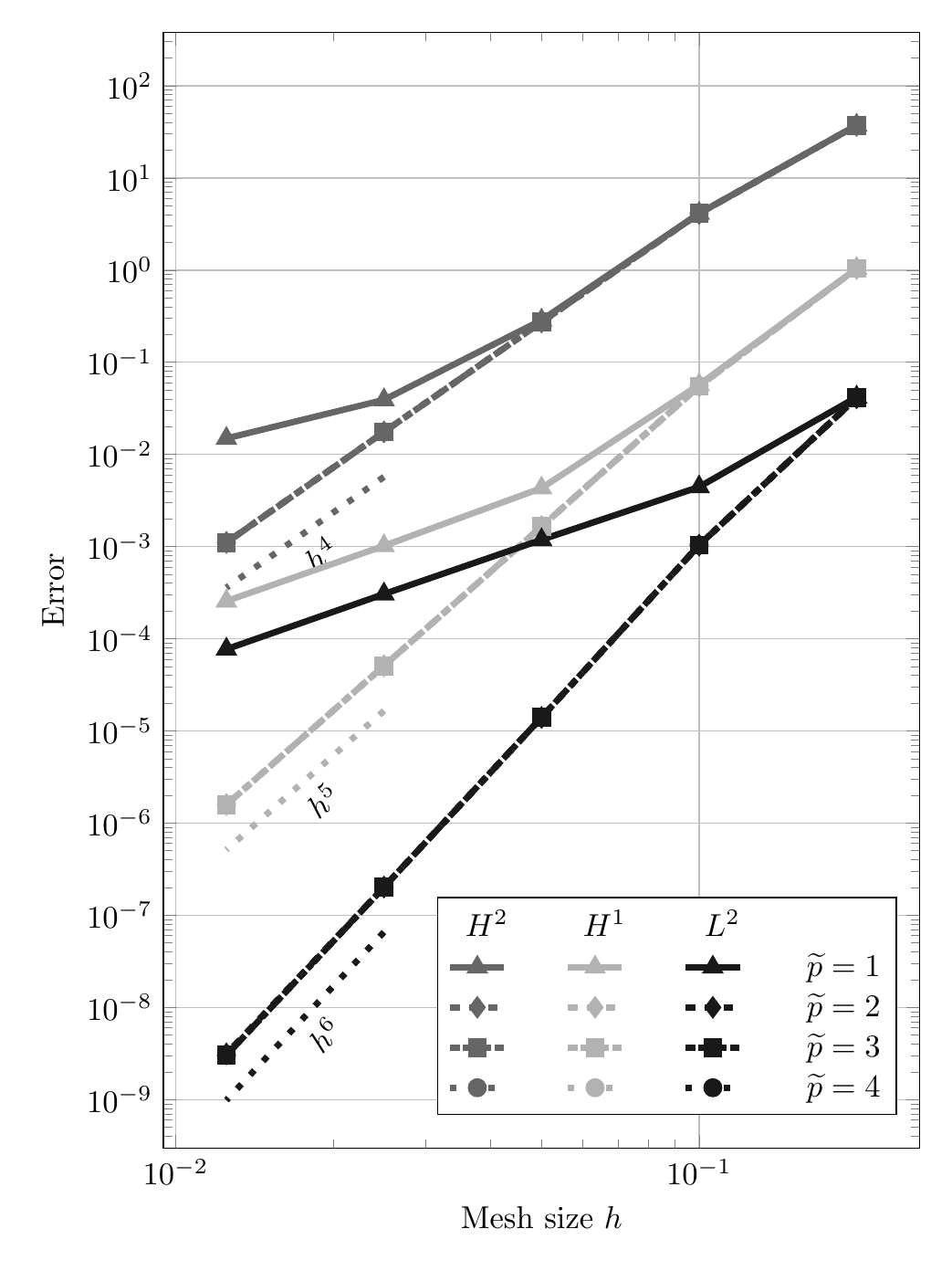}
  \subcaption{Ex.\ III with $p = 5$, $r = 1$.} \label{fig::plot_gluingData_error_p5_geo8}
  \end{subfigure}
    \begin{subfigure}[b]{0.24\textwidth}
  \centering
  \includegraphics[width=\textwidth]{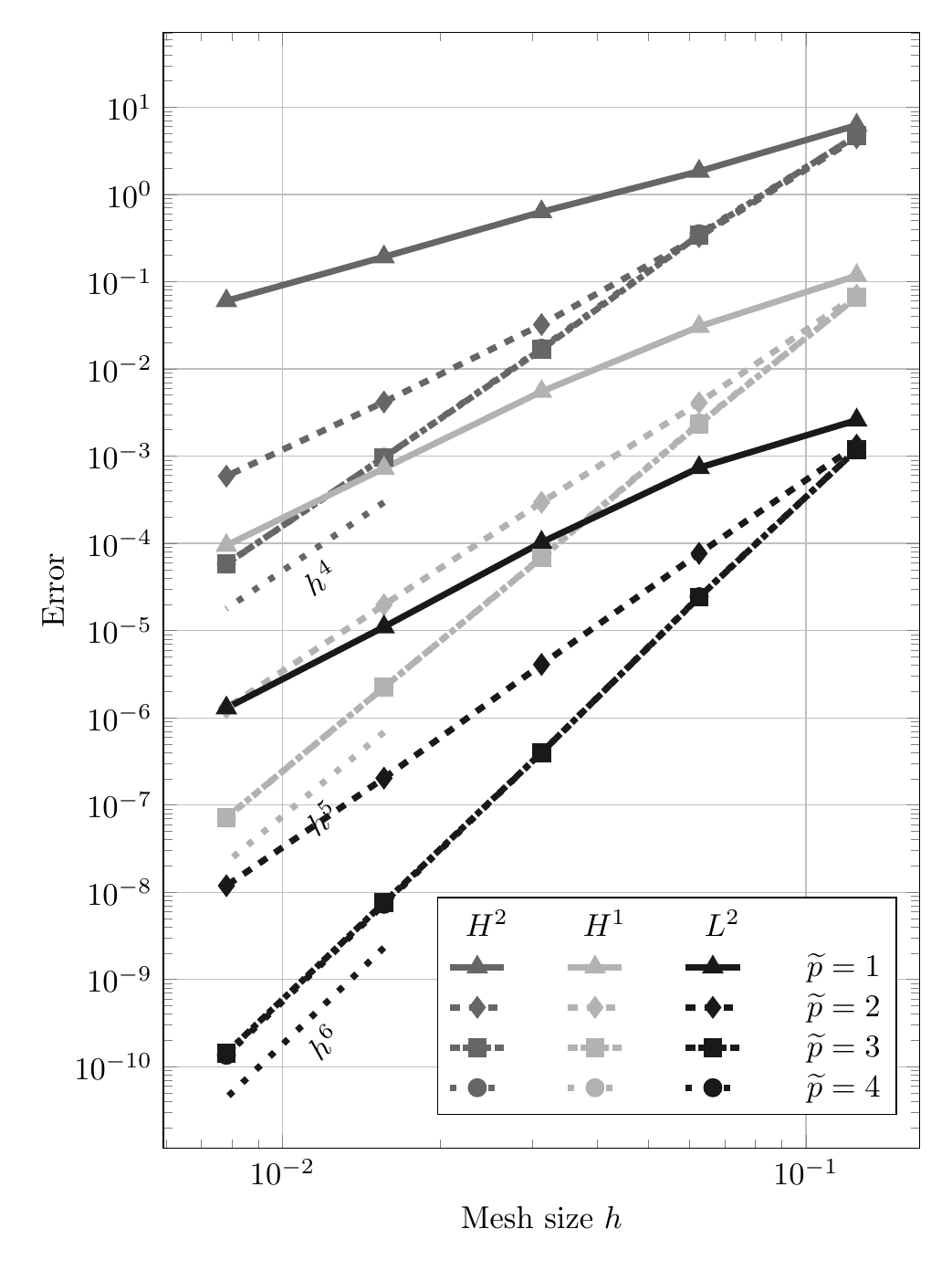}
  \subcaption{Ex.\ IV with $p = 5$, $r = 1$.} \label{fig::plot_gluingData_error_p5_geo41}
  \end{subfigure}

  \caption{Convergence rates for different polynomial degrees $p$ and varying approximations of the gluing data.} \label{fig::ex1error}
\end{figure}

\subsection{Different regularity $r$} \label{sec::regularity}

In the following we focus on different regularities and therefore on different numbers of degrees of freedom. For each polynomial degree $p \in \{ 3,4,5\}$ the gluing data is approximated by splines with $\widetilde{p} = \max(p-2,2)$ and $\widetilde{r} = \widetilde{p}-1$. The results are shown in Figure~\ref{fig::plot_regularity}. One can see that in all examples the errors for $r \leq p-2$ are quite similar with respect to the number of degrees of freedom, whereas the errors with $r = p-1$ are almost the same with less than half of the degrees of freedom. This shows that using splines of maximum regularity, a comparable error can be achieved with significantly fewer degrees of freedom. Hence, the underlying linear system is much smaller and the computation time can be reduced significantly. Note that for the  standard AS-$G^1$ basis construction as developed in~\cite{kapl2017dimension} the regularity is bounded by $r\leq p-2$ globally to obtain the nested, isogeometric spaces ${\mathcal{V}}^1_{h}$. For the construction we propose here, reduced regularity is only needed for the interface space $\widetilde{\mathcal{A}}_{\Gamma}$, whereas the patch-interior spaces $\mathcal{A}^{(S)}_{\circ}$, for $S\in\{L,R\}$, can be constructed with $r=p-1$. As a consequence the approximately $C^1$-smooth spaces $\widetilde{\mathcal{V}}^1_{h}$ are not nested, even though the underlying $C^0$-smooth spaces ${\mathcal{V}}^0_{h}$ are.

\begin{figure}[hp!]  
  \begin{subfigure}[b]{0.24\textwidth}
  \centering
    \includegraphics[width=\textwidth]{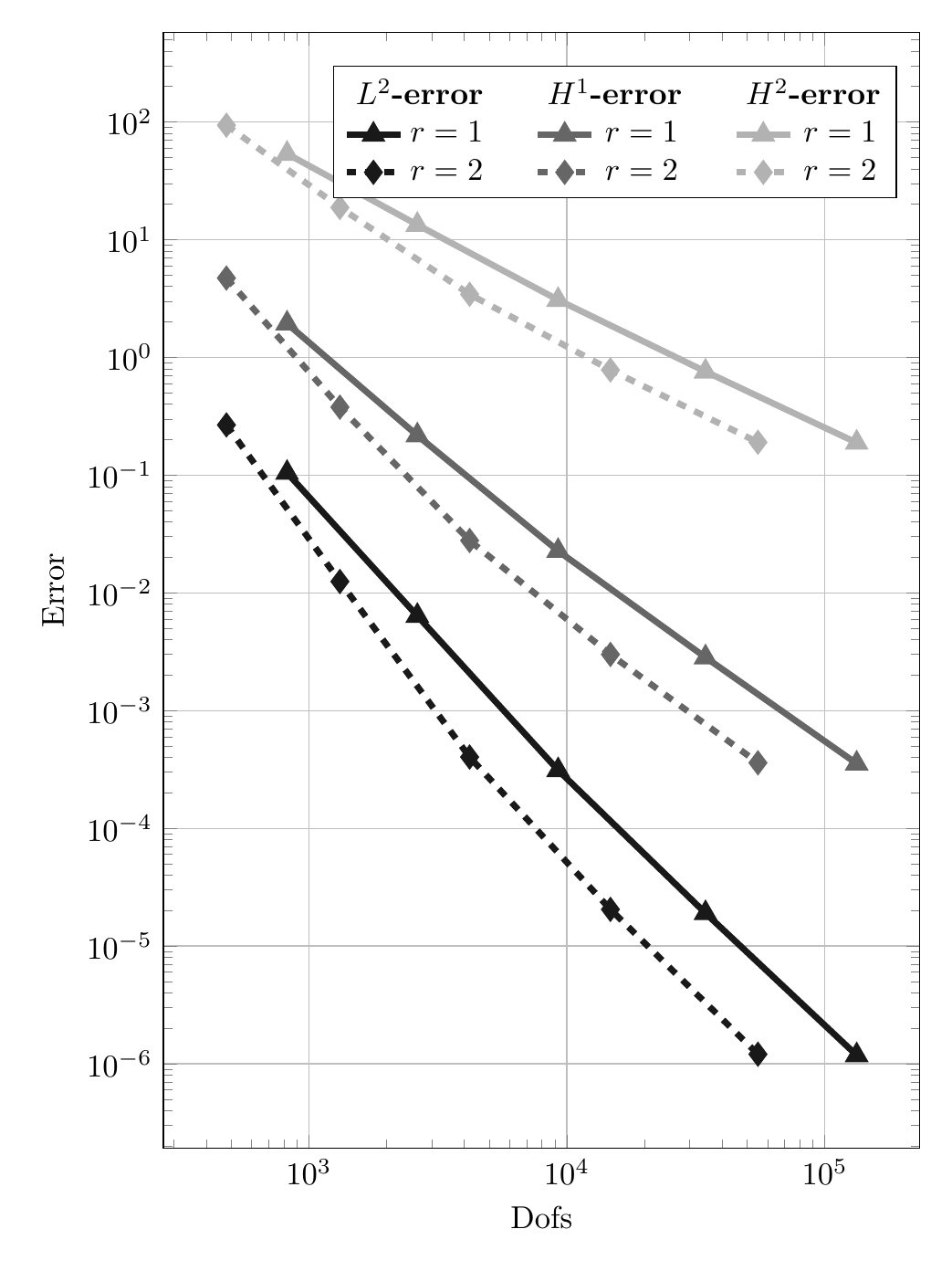}
  \subcaption{Ex.\ I: $p=3$, $(\widetilde{p},\widetilde{r}) = (2,1)$.} \label{fig::plot_regularity_error1_AS}
  \end{subfigure}
  \begin{subfigure}[b]{0.24\textwidth}
  \centering
    \includegraphics[width=\textwidth]{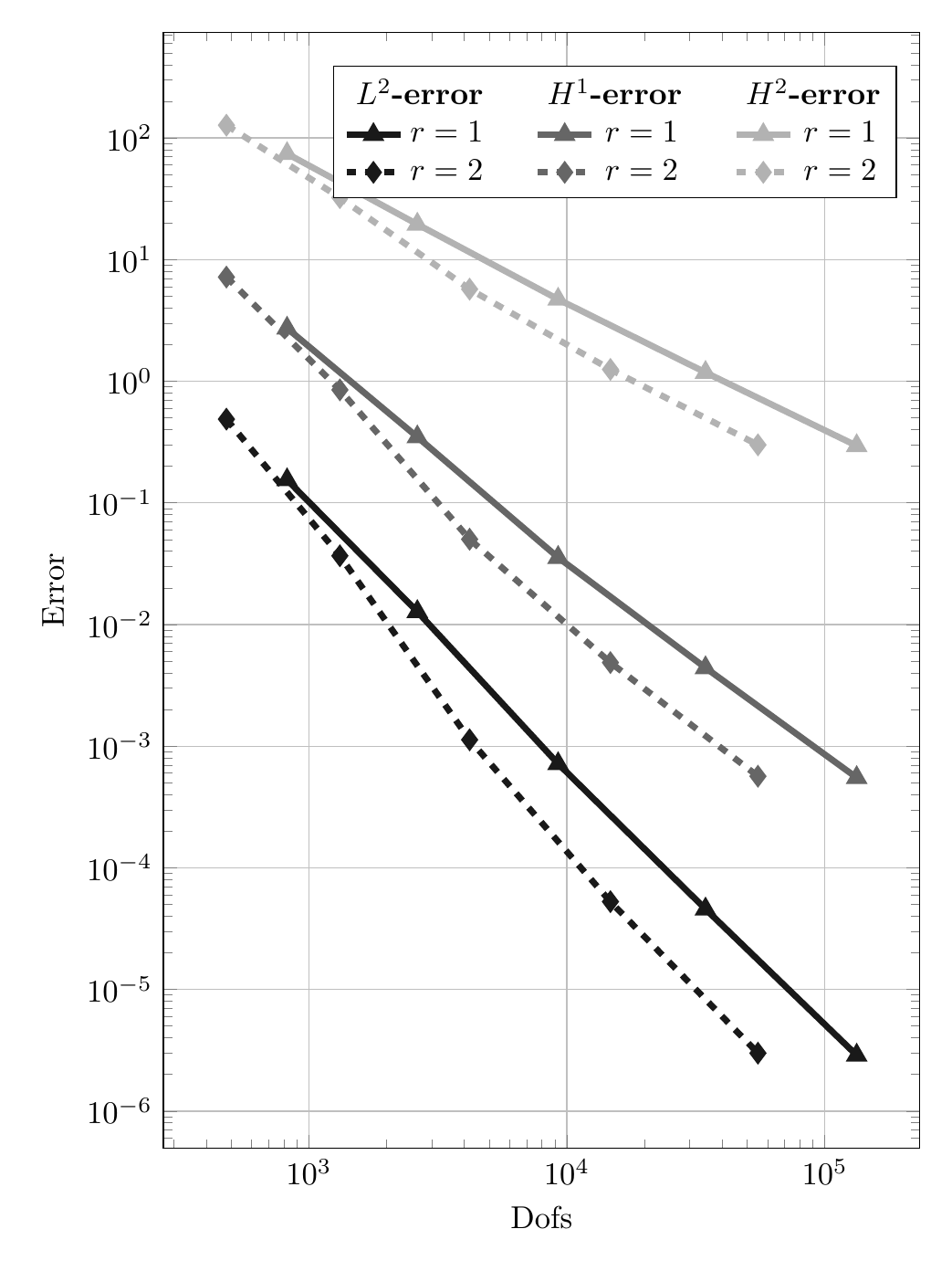}
  \subcaption{Ex.\ II: $p=3$, $(\widetilde{p},\widetilde{r}) = (2,1)$.} \label{fig::plot_regularity_error1}
  \end{subfigure}
      \begin{subfigure}[b]{0.24\textwidth}
  \centering
    \includegraphics[width=\textwidth]{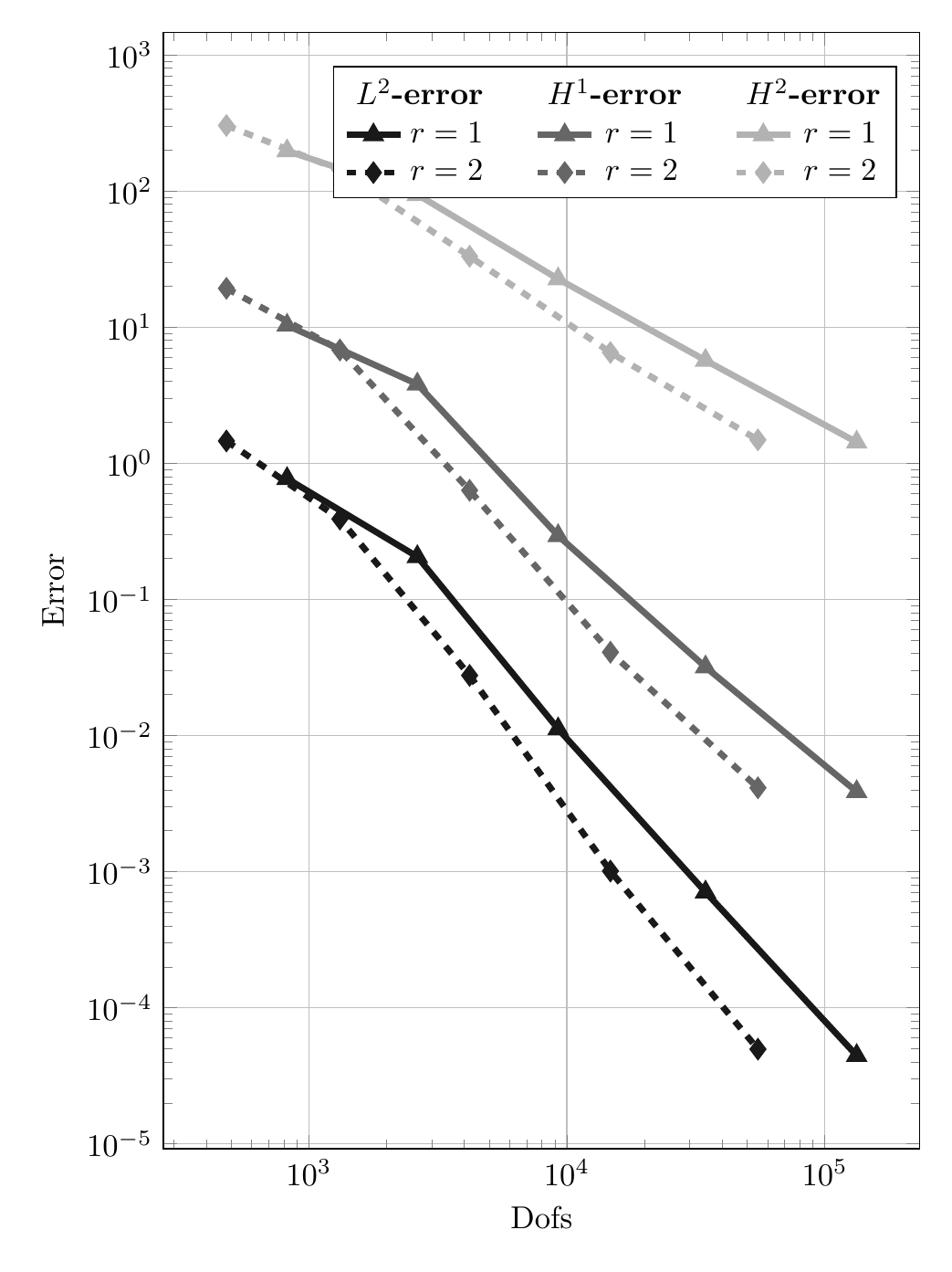}
  \subcaption{Ex.\ III: $p=3$, $(\widetilde{p},\widetilde{r}) = (2,1)$.} \label{fig::plot_regularity_error1_geo8}
  \end{subfigure}
    \begin{subfigure}[b]{0.24\textwidth}
  \centering
    \includegraphics[width=\textwidth]{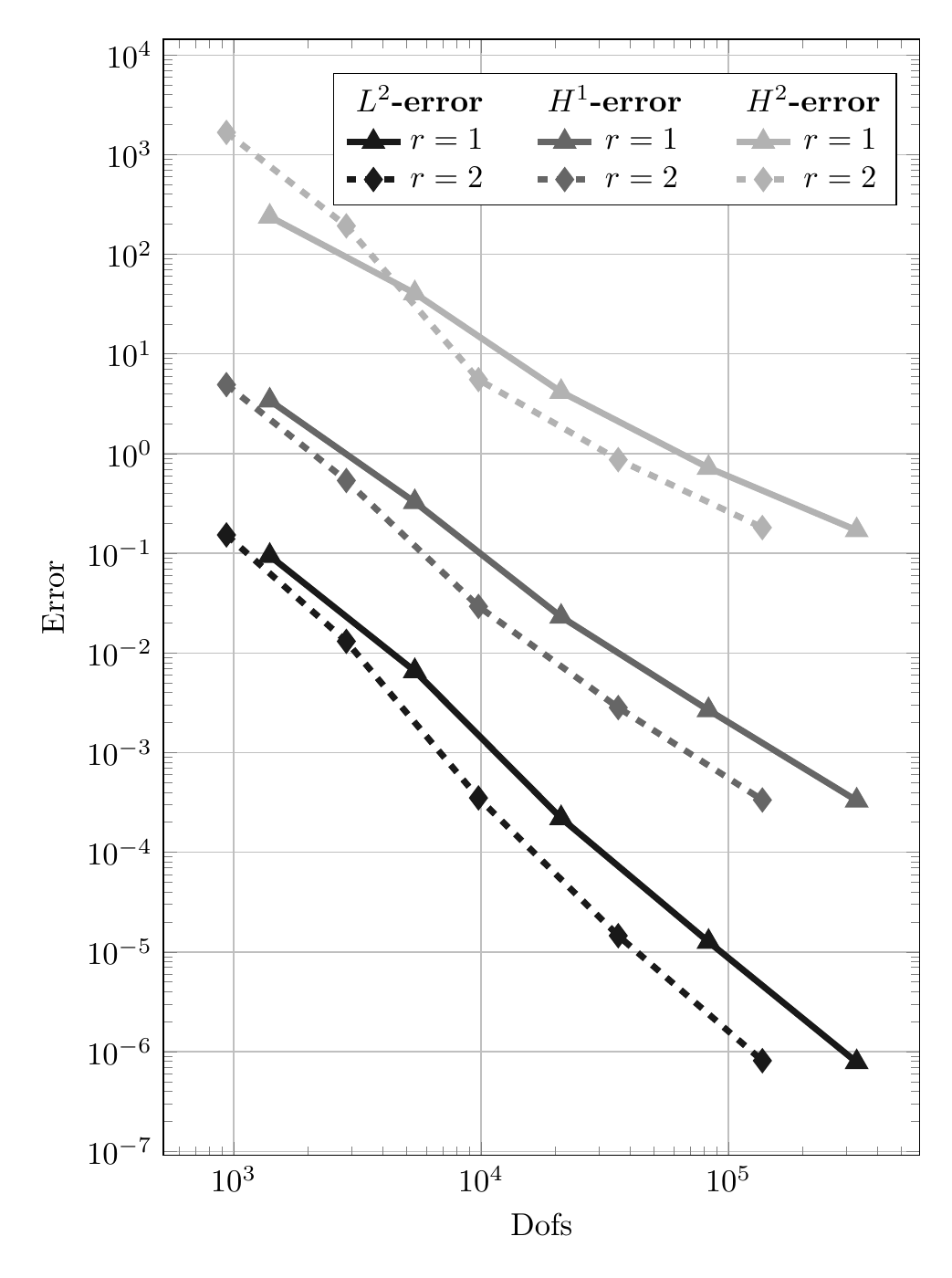}
  \subcaption{Ex.\ IV: $p=3$, $(\widetilde{p},\widetilde{r}) = (2,1)$.} \label{fig::plot_regularity_error1_geo41}
  \end{subfigure}

  \begin{subfigure}[b]{0.24\textwidth}
  \centering
    \includegraphics[width=\textwidth]{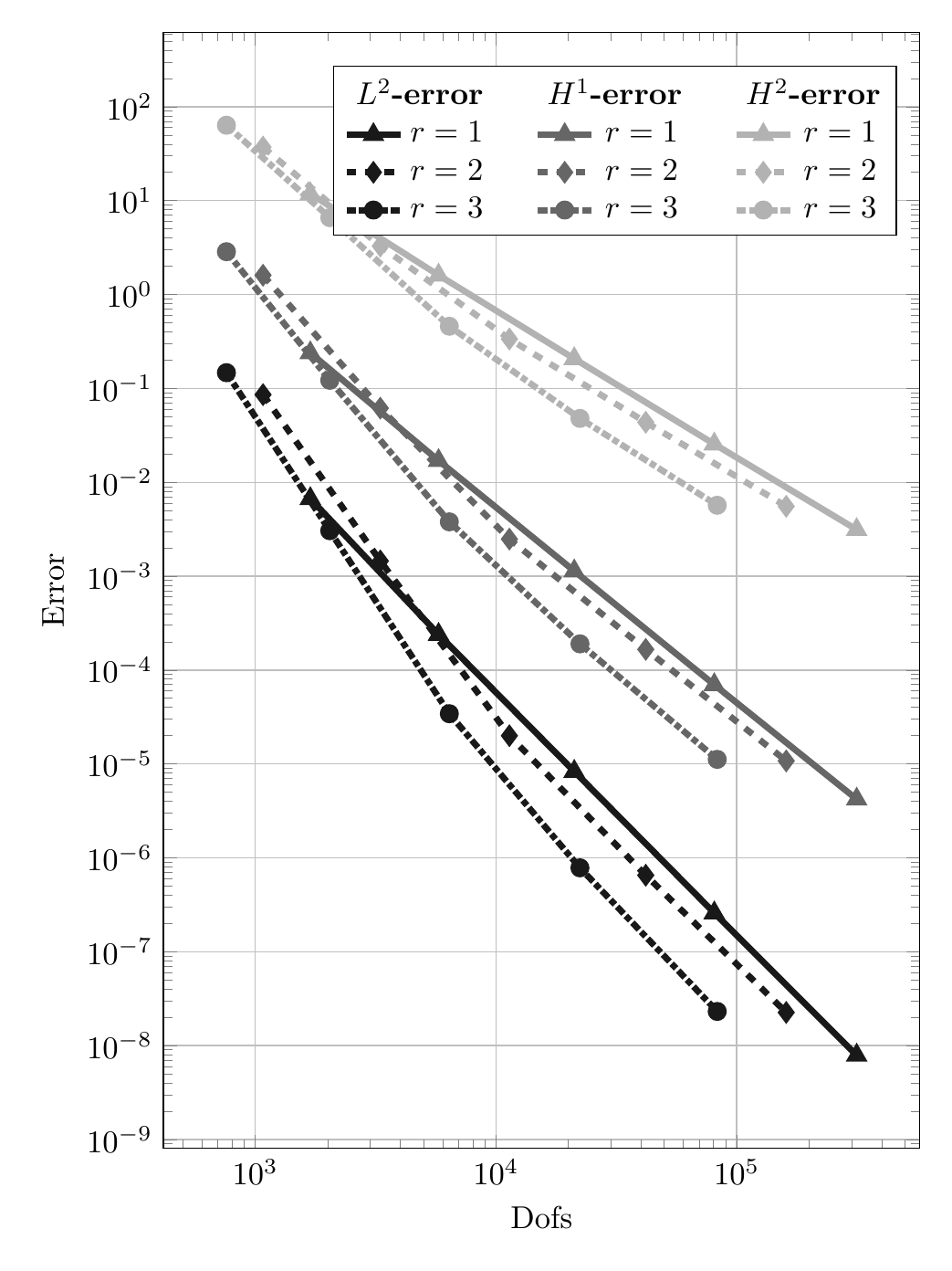}
  \subcaption{Ex.\ I: $p=4$, $(\widetilde{p},\widetilde{r}) = (2,1)$.} \label{fig::plot_regularity_error2_AS}
  \end{subfigure}
  \begin{subfigure}[b]{0.24\textwidth}
  \centering
    \includegraphics[width=\textwidth]{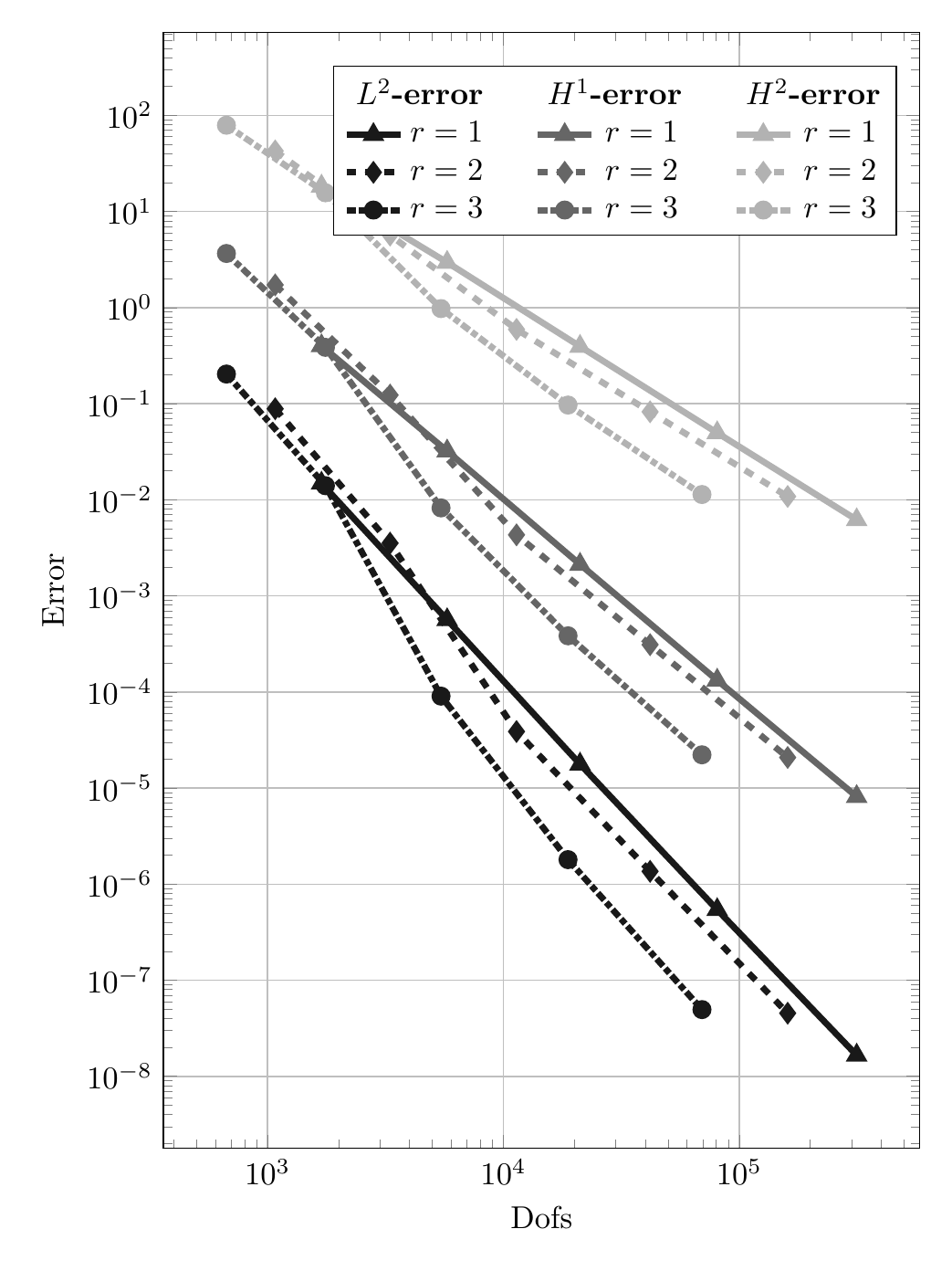}
  \subcaption{Ex.\ II: $p=4$, $(\widetilde{p},\widetilde{r}) = (2,1)$.} \label{fig::plot_regularity_error2}
  \end{subfigure}
      \begin{subfigure}[b]{0.24\textwidth}
  \centering
    \includegraphics[width=\textwidth]{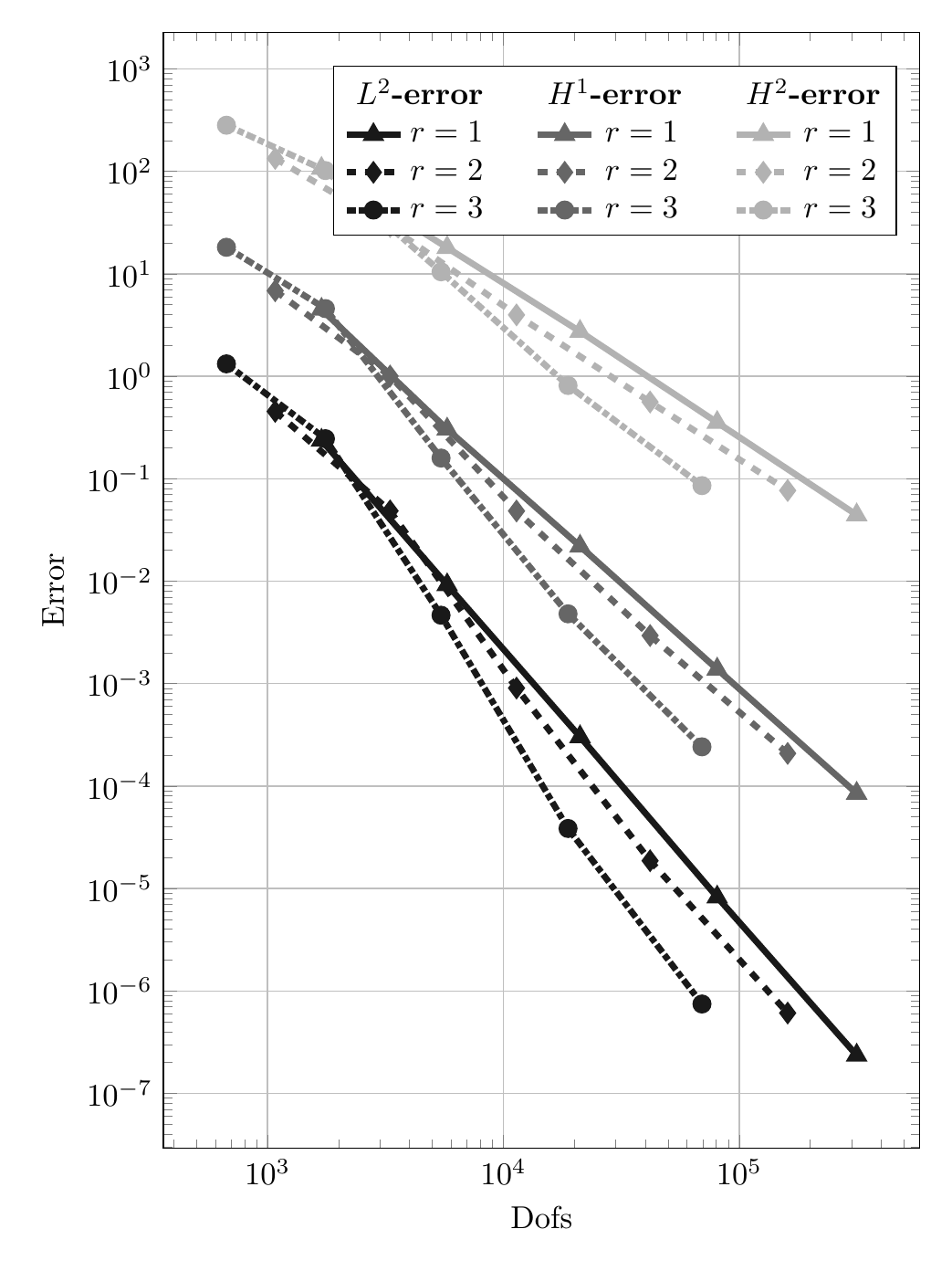}
  \subcaption{Ex.\ III: $p=4$, $(\widetilde{p},\widetilde{r}) = (2,1)$.} \label{fig::plot_regularity_error2_geo8}
  \end{subfigure}
    \begin{subfigure}[b]{0.24\textwidth}
  \centering
    \includegraphics[width=\textwidth]{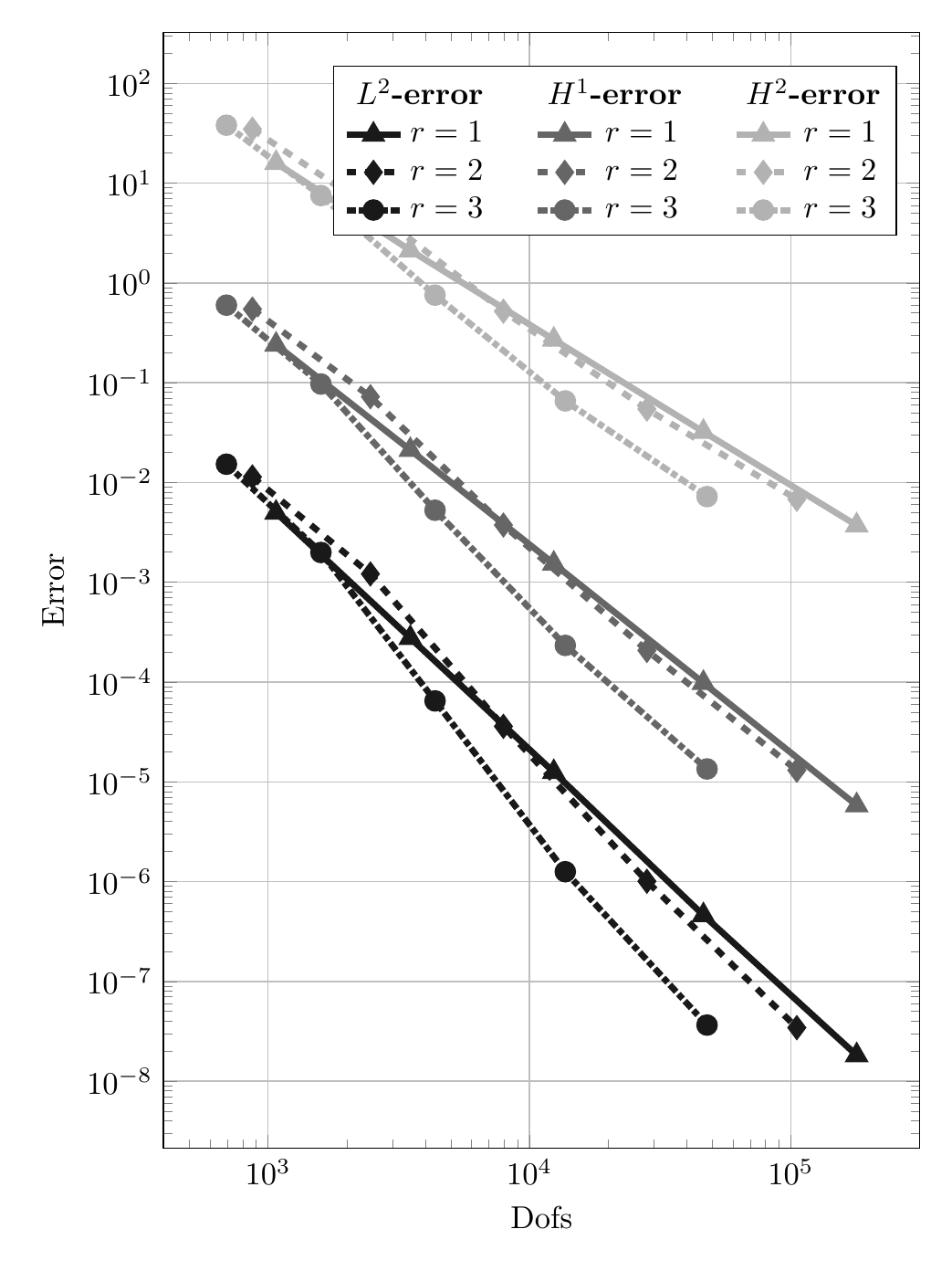}
  \subcaption{Ex.\ IV: $p=4$, $(\widetilde{p},\widetilde{r}) = (2,1)$.} \label{fig::plot_regularity_error2_geo41}
  \end{subfigure}

    \begin{subfigure}[b]{0.24\textwidth}
  \centering
    \includegraphics[width=\textwidth]{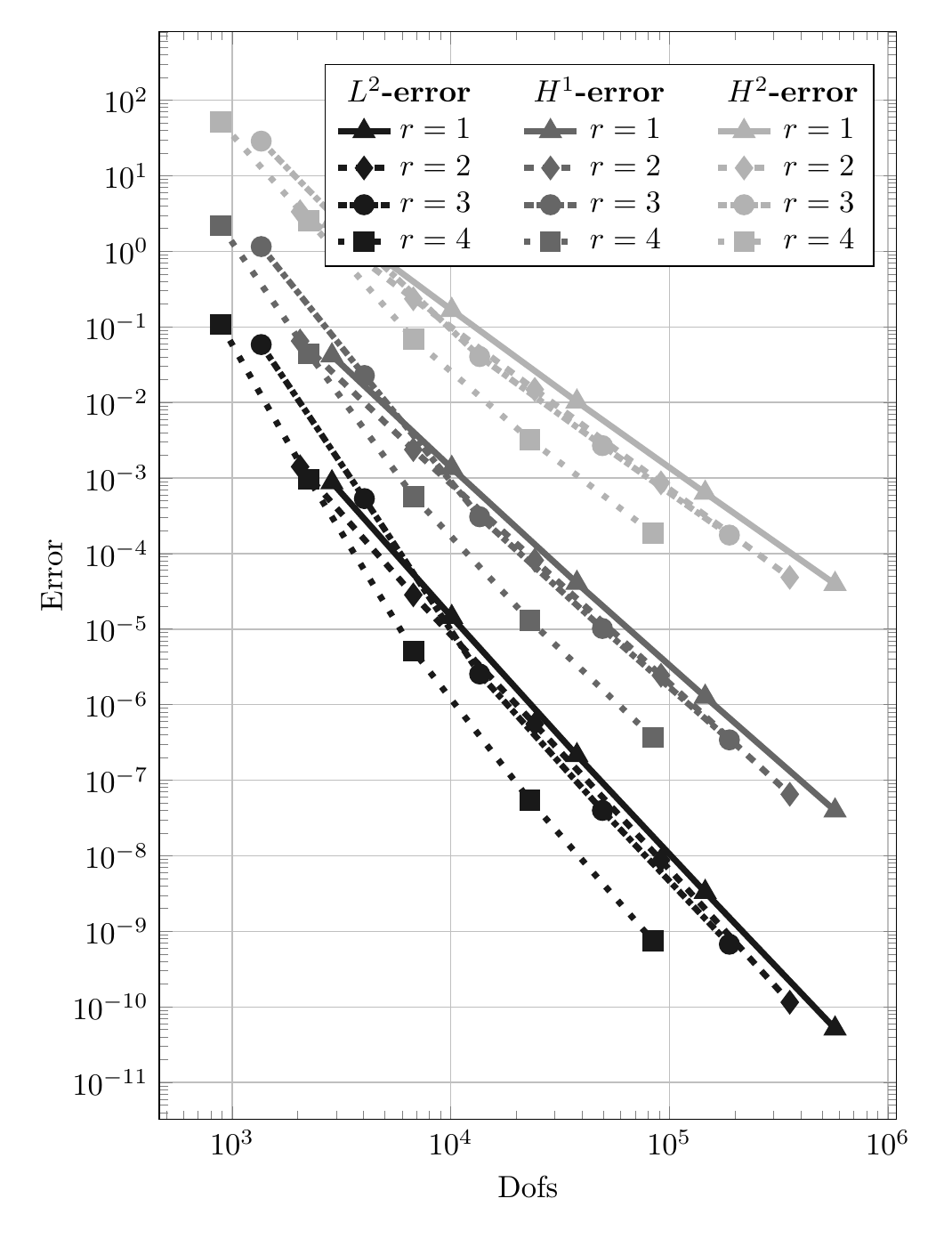}
  \subcaption{Ex.\ I: $p=5$, $(\widetilde{p},\widetilde{r}) = (3,2)$.} \label{fig::plot_regularity_error3_AS}
  \end{subfigure}
  \begin{subfigure}[b]{0.24\textwidth}
  \centering
    \includegraphics[width=\textwidth]{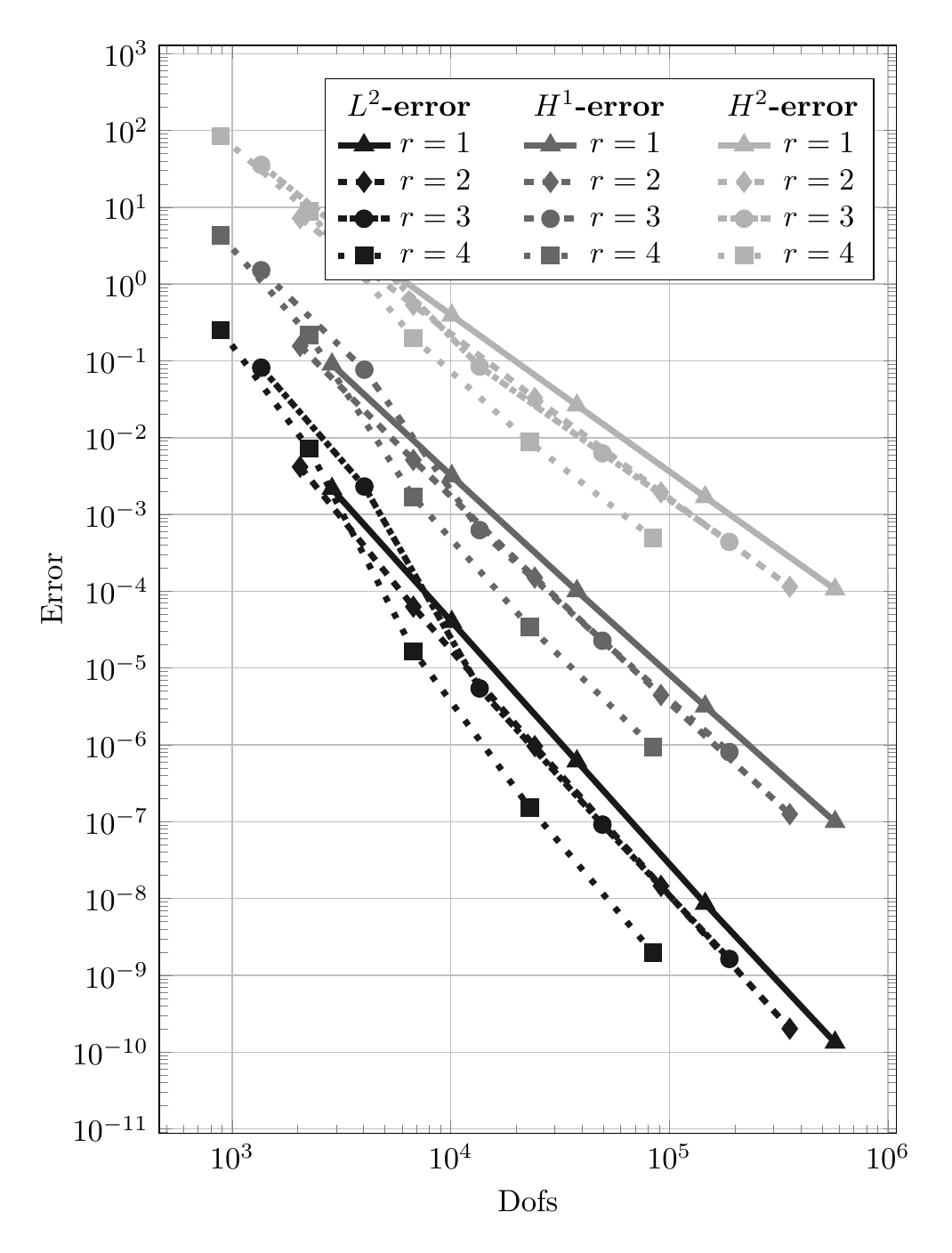}
  \subcaption{Ex.\ II: $p=5$, $(\widetilde{p},\widetilde{r}) = (3,2)$.} \label{fig::plot_regularity_error3}
  \end{subfigure}
      \begin{subfigure}[b]{0.24\textwidth}
  \centering
    \includegraphics[width=\textwidth]{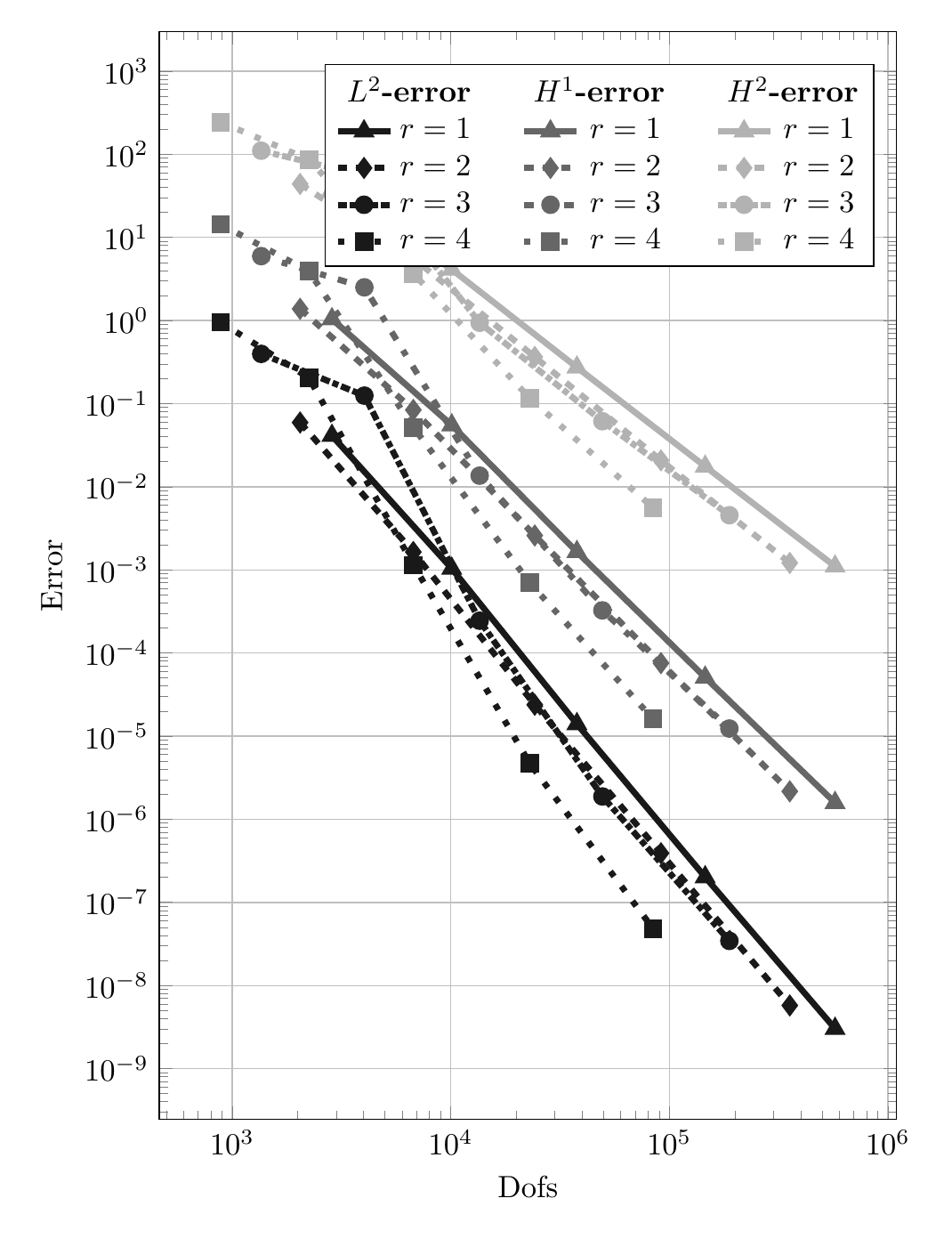}
  \subcaption{Ex.\ III: $p=5$, $(\widetilde{p},\widetilde{r}) = (3,2)$.} \label{fig::plot_regularity_error3_geo8}
  \end{subfigure}
    \begin{subfigure}[b]{0.24\textwidth}
  \centering
    \includegraphics[width=\textwidth]{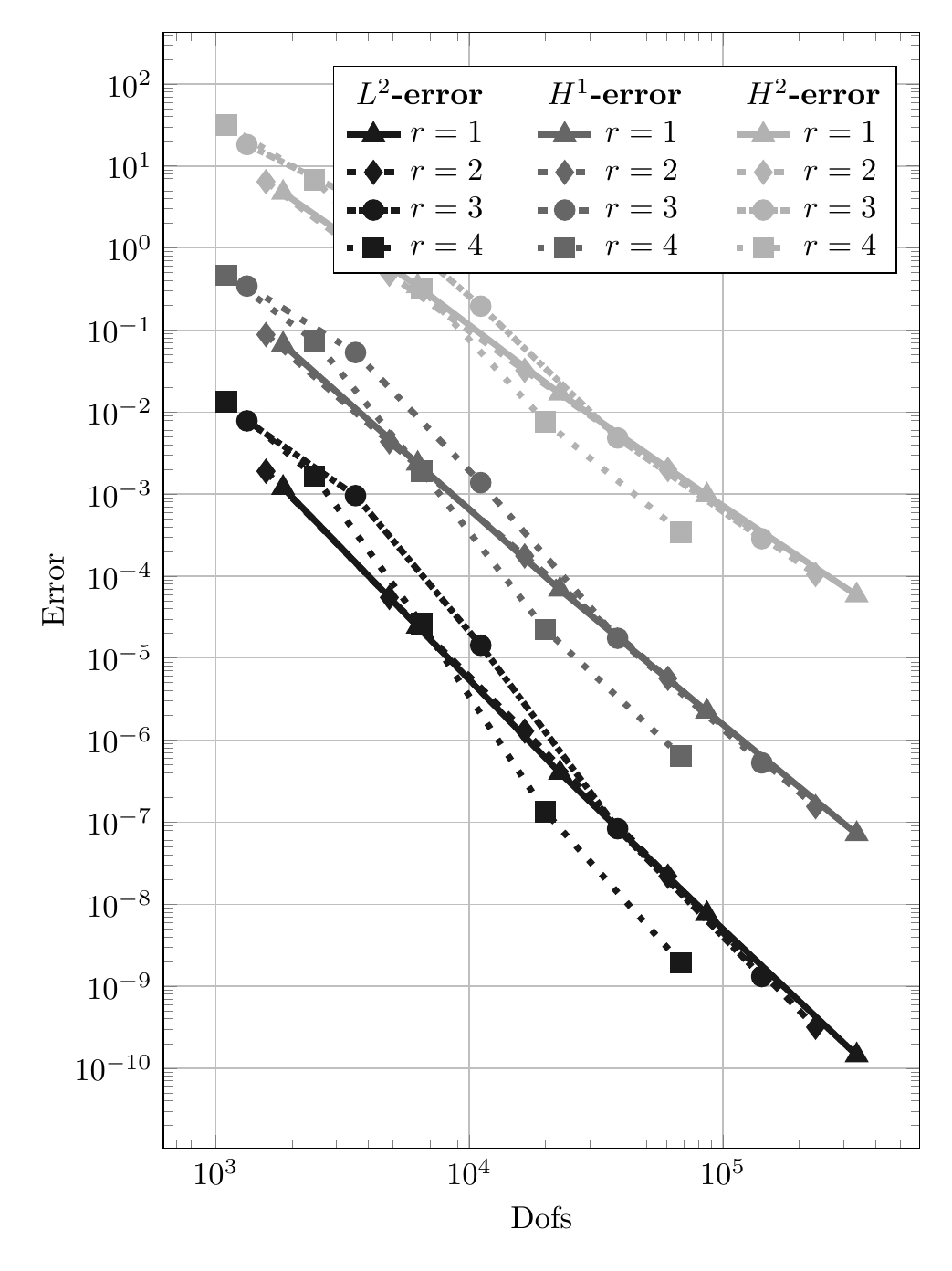}
  \subcaption{Ex.\ IV: $p=5$, $(\widetilde{p},\widetilde{r}) = (3,2)$.} \label{fig::plot_regularity_error3_geo41}
  \end{subfigure}
   
  \caption{Convergence rates for different regularities: showing error vs. number of degrees of freedom. Note that only plots where $r \leq p - 1$ are shown.} \label{fig::plot_regularity}
\end{figure}

\section{Conclusion and future work}

In this paper, we constructed and studied approximately $C^1$-smooth spaces $\widetilde{\mathcal{V}}^1_h$ over general two-patch domains which can be used for solving fourth order problems. Following the approach in~\cite{kapl2017dimension}, the basis construction for the space is simple and combines standard patch-wise basis functions with specific interface functions. The construction in~\cite{kapl2017dimension}, which is based on~\cite{KaViJu15,collin2016analysis}, describes a basis for analysis-suitable $G^1$ parametrizations. In that case, the parametrizations need to satisfy the AS-$G^1$ condition which requires the gluing data to be linear functions. This is a severe restriction, as most (generic) spline parametrizations are not AS-$G^1$. Thus, in general, a reparametrization as developed in~\cite{kapl2017isogeometric} is necessary.

Our approach relaxes this AS-$G^1$ condition on the geometry and is applicable on most two-patch domains. The only requirement is, that the gluing data is $C^1$, which is always satisfied if the patch parametrizations are at least $C^2$. Generic spline parametrizations yield gluing data which is either piecewise polynomial of high degree or rational. Instead of reparametrizing the domain to obtain linear gluing data, we base our construction on an approximation of the given gluing data. Therefore, we only obtain approximate $C^1$-smoothness at the interface. In other words, we get a jump of the normal derivative at the interface. The space $\widetilde{\mathcal{V}}^1_h$ is given as the direct sum of the subspaces $\mathcal{A}_{\circ}^{(L)}$, $\mathcal{A}_{\circ}^{(R)}$, which are the patch-interior spaces, and $\widetilde{\mathcal{A}}_{\Gamma}$, which is the interface space. The patch interior spaces are standard isogeometric spaces which have vanishing function value and vanishing gradient at the interface. The interface space is composed of functions that span traces as well as functions that span (approximate) normal derivatives at the interface.

Since the construction of the interface space is based on approximated, nonlinear gluing data, the space $\widetilde{\mathcal{A}}_{\Gamma}$ is of higher polynomial degree and lower regularity locally near the interface, as stated in~\eqref{eq::spaceforbasisfunctions}. As a consequence the approximately $C^1$-smooth spaces $\{\widetilde{\mathcal{V}}^1_{h}\}_h$ are not nested, even though the underlying family of $C^0$-smooth spaces $\{{\mathcal{V}}^0_{h}\}_h$ is refined by knot insertion and therefore nested. The advantage with using the approximated gluing data compared to~\cite{kapl2017dimension} is that we allow geometries that are not necessarily AS-$G^1$ geometries. Furthermore, we show that by allowing non-nested spaces, splines of maximum regularity can be used away from the interface. In contrast, the standard construction over AS-$G^1$ parametrizations requires $r \leq p-2$.

In the future, we want to extend the construction of the approximately $C^1$-smooth spaces to multi-patch domains. A possible approach is to introduce vertex spaces by interpolation, as in~\cite{kapl2019argyris}, thus enforcing $C^2$ super-smoothness at all vertices. Another challenge is to obtain a construction that results in nested spaces. This may be done by avoiding spline spaces of locally higher degree and reduced regularity. Moreover, a construction with uniform degree $p$ everywhere allows the use of a standard Gaussian quadrature rule in all elements, whereas the construction proposed in this paper requires a quadrature rule of higher order in all elements neighboring the interface. Furthermore, due to the non-standard structure of the space near the interface, a complete numerical analysis of the proposed approach is beyond the scope of this paper. An extension of the construction to surfaces is also of practical relevance, since for many applications, a given surface geometry is only smooth up to some prescribed tolerance. Hence, the $C^1$-smoothness of functions defined on such surfaces may also be imposed only approximately.

\section*{Acknowledgments}

Both authors are supported by the Austrian Science Fund (FWF) and the government of Upper Austria through the project P~30926-NBL entitled ``Weak and approximate $C^1$-smoothness in isogeometric analysis''. Moreover, Thomas Takacs is partially supported by the Linz Institute of Technology (LIT) and the government of Upper Austria through the project LIT-2019-8-SEE-116 entitled ``PARTITION – PDE-aware isogeometric discretization based on neural networks''. All support is gratefully acknowledged.

\end{document}